\documentclass[10pt]{amsart}

\usepackage[dvips]{graphicx}
\usepackage{pst-all} 

\newtheorem{thm}{Theorem}[section]
\newtheorem{prop}[thm]{Proposition}
\newtheorem{conj}[thm]{Conjecture}

\newtheorem{lem}[thm]{Lemma}

\theoremstyle{definition}

\numberwithin{equation}{section}
\newtheorem{rem}[thm]{Remark}
\newtheorem{ex}[thm]{Example}
\newtheorem{defn}[thm]{Definition}

\newtheorem{ass}[thm]{Assumption}

\newtheorem{caut}[thm]{Caution}

\def\bbP{\mathbb{P}}
\def\bbZ{\mathbb{Z}}
\def\bbQ{\mathbb{Q}}

\def\bfS{\mathbf{S}}

\def\bfM{\mathbf{M}}
\def\bfQ{\mathbf{Q}}
\def\bfO{\mathbf{O}}
\def\bfA{\mathbf{A}}

\def\ve{\varepsilon}

\begin{document}

\bibliographystyle{amsalpha}

\title[Exact WKB analysis and cluster algebras II]
{Exact WKB analysis and cluster algebras II: simple poles, orbifold points, and generalized cluster algebras}

\author{Kohei Iwaki}
\address{\noindent Research Institute 
for Mathematical Sciences, Kyoto University, Kyoto,
657-8501, Japan}
\email{iwaki@kurims.kyoto-u.ac.jp}

\author{Tomoki Nakanishi}
\address{\noindent Graduate School of Mathematics, Nagoya University, 
Chikusa-ku, Nagoya,
464-8604, Japan}
\email{nakanisi@math.nagoya-u.ac.jp}

\subjclass[2010]{13F60, 34M60}

\date{}
\maketitle

\begin{abstract}
This is a continuation of developing mutation theory 
in  exact WKB analysis using the framework of cluster algebras.
Here we study the
Schr\"odinger equation on a compact Riemann surface
with turning points of simple-pole type. 
We show that the orbifold triangulations by Felikson, Shapiro, and Tumarkin
provide a natural framework of describing the mutation of Stokes graphs,
where simple poles correspond to orbifold points.
We then show that under the mutation of Stokes graphs
around simple poles the Voros symbols mutate
as the variables of generalized cluster algebras introduced by
Chekhov and Shapiro.
\end{abstract}


\section{Introduction}
This is a continuation of the paper \cite{Iwaki14a},
where the mutation theory in {\em exact WKB analysis\/} 
using the framework of {\em cluster algebras\/} was initiated.
The exact WKB analysis is a method to study the WKB solutions of the 
Schr\"odinger equation using the Borel resummation.
It was initiated by Voros \cite{Voros83}, and developed 
by Aoki, Kawai, and Takei \cite{Aoki91,Kawai05}, 
Delabaere, Dillinger, and Pham \cite{Delabaere93,Delabaere99}, and others.
Meanwhile, cluster algebras were introduced by Fomin and Zelevinsky
\cite{Fomin03a,Fomin07}. 
They were originally introduced to study Lie theory, 
but nowadays they are recognized as a common algebraic/combinatorial 
structure underlying in various areas in mathematics.

Let us briefly summarize the main result of  
\cite{Iwaki14a} highlighting some key words.
We study the {\em WKB solutions\/} of 
the {\em Schr\"odinger equation\/} on a compact Riemann surface.
The principal term of the {\em potential \/} of 
the Schr\"odinger equation determines 
a {\em quadratic differential\/} on the surface,
and the latter further determines a graph on the surface 
called the {\em Stokes graph}.
The Stokes graph plays a central role to study the local and 
global analytic properties of the WKB solutions, 
and the {\em Voros symbols} describe the monodoromy data 
of the WKB solutions. 
Under a continuous deformation of the potential, 
the Stokes graph may change its topology.
We call this phenomenon the {\em mutation\/} of Stokes graphs.
Meanwhile, the Stokes graph induces a {\em triangulation\/} 
of a bordered surface with marked points in the sense of \cite{Fomin08}.
Furthermore, the mutation of Stokes graphs corresponds to the mutation of 
triangulations called {\em signed flips\/} and {\em signed pops}.
These facts provide a natural bridge between exact WKB analysis 
and cluster algebra theory with the help of the
{\em surface realization\/} of cluster algebras developed by 
\cite{Gekhtman05,Fock03b,Fomin08,Fomin08b}.
Along the mutation of Stokes graphs,
the Voros symbols also mutate (or jump) due to the {\em Stokes phenomenon}. 
The main result of \cite{Iwaki14a} is that the Voros symbols mutate as
the variables of the corresponding cluster algebra. To be more precise,
they mutate by {\em signed mutations\/} and {\em signed pops\/} 
of the {\em extended seeds}, which are  certain extensions of 
the ordinary mutation of cluster algebras.

In \cite{Iwaki14a} it was technically essential to assume that 
the Schr{\"o}dinger equation has no 
turning points of {\em simple-pole type} 
({\em simple-pole} for short).
In this paper we lift this restriction and treat the case where 
there are simple poles.
The connection formula of the WKB solutions around a simple pole was
obtained by Koike \cite{Koike00}.
This is the counterpart of the celebrated Voros' connection formula
of the WKB solutions around a usual {\em turning point\/}.
Based on Koike's connection formula, 
we obtain the following results in this paper.

(a). Firstly, we obtain the Delabaere-Dillinger-Pham (DDP) type jump 
formula for the Voros symbols at the reduction of a
Stokes segment involving a simple pole 
(Theorem \ref{thm:Stokes-auto-II}). 
Our formula is closely related to the result of 
Kamimoto et al.\ \cite{Kamimoto10} 
(see Remark \ref{remark:known-results}).

(b). Secondly, we reformulate the above jump formula in view of 
cluster algebra theory,
where we have two novel features in contrast to the previous 
case in \cite{Iwaki14a}.

\begin{itemize}
\item[(1).] Instead of using the surface realization 
of cluster algebras, we need to use the {\em orbifold realization\/} 
of cluster algebras  by Felikson, Shapiro, and Tumarkin \cite{Felikson11}. 
The simple poles correspond to the {\em orbifold points\/} therein.

\item[(2).] The Voros symbols mutate as the variables of 
{\em generalized cluster algebras\/} introduced by 
Chekhov and Shapiro \cite{Chekhov11} instead of
ordinary cluster algebras (Theorem \ref{thm:mut2}). 
To be more precise, they mutate by {\em signed mutations\/},
which are certain extensions of the ordinary mutation of 
generalized cluster algebras.
\end{itemize}

The organization of the paper is as follows.
In Section 2 we consider the Schr\"odinger equation 
on a compact Riemann surface especially having simple poles.
Then, we obtain the DDP type jump formula
for the Voros symbols at the reduction
of a Stokes segment involving a simple pole.
In Section 3 we show how the above reduction is related to
the signed flips of labeled Stokes triangulations of an orbifold.
In Section 4 we reformulate our jump formula 
in view of the above correspondence to Stokes triangulations,
and show that the Voros symbols
mutate as the variables of generalized cluster algebras
at the mutation of Stokes graphs involving simple poles.

\bigskip
{\em Acknowledgements.} We greatly thank Dylan Thurston, who suggested 
us the link between simple poles and orbifold points at the workshop 
``Cluster algebras and related topics" held at Mathematisches 
Forschungsinstitut Oberwolfach in December, 2013.
We also thank Anna Felikson, Anne-Sophie Gleitz, 
Takahiro Kawai, Tatsuya Koike, Michael Shapiro 
and Yoshitsugu Takei for useful communications and discussions.
K.I is supported by Research Fellowships of 
JSPS KAKENHI Grant Number 13J02831. 

\section{Exact WKB analysis of Schr{\"o}dinger 
equation with simple poles}

\subsection{Schr{\"o}dinger equation and assumptions} 

Let us consider a Schr{\"o}dinger equation 
(i.e., a second order linear differential equation containing a 
large parameter $\eta$ which corresponds to the inverse of 
the Planck constant $\hbar$) for a complex function 
($(-{1}/{2})$-form, to be precise) $\psi(z,\eta)$ 
defined on a compact connected Riemann surface $\Sigma$ 
which also depends on $\eta$:
\begin{equation} \label{eq:Sch}
\left( \frac{d^{2}}{dz^{2}} - \eta^{2} Q(z,\eta) \right) \psi(z,\eta)= 0,
\end{equation}
\[
Q(z,\eta) = Q_{0}(z) + \eta^{-2} Q_{2}(z).
\]
Here \eqref{eq:Sch} is a local expression of our equation in 
a local coordinate $z$ of $\Sigma$. The coefficients 
$Q_{0}(z)$, $Q_{2}(z)$ are meromorphic functions.
If we take a coordinate transformation $z=z(\tilde{z})$, 
we obtain an equation of the same form as \eqref{eq:Sch} for
the coordinate $\tilde{z}$ as follows:
\begin{equation} \label{eq:Sch-w}
\left( \frac{d^{2}}{d\tilde{z}^{2}} - 
\eta^{2} \tilde{Q}(\tilde{z},\eta) \right) 
\tilde{\psi}= 0, \quad \tilde{\psi}(\tilde{z},\eta) 
= \psi\bigl(z(\tilde{z}),\eta\bigr)
\left(\frac{dz(\tilde{z})}{d\tilde{z}}\right)^{-1/2},
\end{equation}
\begin{equation} \label{eq:transformation-of-potentail}
\tilde{Q}(\tilde{z},\eta) = 
Q\bigl(z(\tilde{z}),\eta\bigr)
\left(\frac{dz(\tilde{z})}{d\tilde{z}}\right)^{2}-
\frac{1}{2}\eta^{-2}\{ z(\tilde{z}); \tilde{z}\}.
\end{equation}
Here $\{ z(\tilde{z}); \tilde{z}\}$ is the Schwarzian 
derivative:
\[
\{ z(\tilde{z}); \tilde{z}\} = 
\left({\frac{d^{3}z(\tilde{z})}{d\tilde{z}^{3}}}\biggl/
{\frac{dz(\tilde{z})}{d\tilde{z}}}\right)-\frac{3}{2}
\left({\displaystyle \frac{d^{2}z(\tilde{z})}{d\tilde{z}^{2}}}\biggl/
{\displaystyle \frac{dz(\tilde{z})}{d\tilde{z}}}\right)^{2}.
\]
In particular, the transformation law 
\begin{equation} \label{eq:transformation-Q0}
\tilde{Q}_{0}(\tilde{z}) = Q_{0}\bigl(z(\tilde{z})\bigr)
\left(\frac{dz}{d\tilde{z}}\right)^{2}
\end{equation}
of the principal terms of the potential of the Schr\"{o}dinger 
equations coincides with that of a {\em meromorphic quadratic differential}, 
that is, a meromorphic section of the line bundle 
$\omega_{\Sigma}^{\otimes2}$, where $\omega_{\Sigma}$ is the holomorphic 
cotangent bundle on $\Sigma$. 
\begin{defn}
\label{def:quad-diff}
{\em The quadratic differential associated with 
the Schr\"{o}dinger equation \eqref{eq:Sch}} 
is the meromorphic quadratic 
differential on $\Sigma$ which is locally given by
\begin{equation} \label{eq:quad-diff}
\phi = Q_{0}(z)dz^{\otimes2}.
\end{equation}
Here $Q_{0}(z)$ is the principal term of 
the potential $Q(z,\eta)$ of the Schr\"{o}dinger equation 
in a local coordinate $z$. 
\end{defn}

The zeros and simple poles of $\phi$ play important roles 
in exact WKB analysis. 
\begin{defn}
[{\cite[Definition 2.6]{Kawai05}, \cite{Koike00}}]
Let $\phi$ be the quadratic differential defined by \eqref{eq:quad-diff}.
\begin{itemize}
\item %
{\em Turning points} of \eqref{eq:Sch} are zeros of $\phi$.  
\item %
{\em Turning points of simple-pole type} of \eqref{eq:Sch}
are simple poles of $\phi$. 
We call these points {\em simple poles} of \eqref{eq:Sch} 
to distinguish them from 
the usual turning points defined above. 
\end{itemize}
\end{defn}

The set of zeros, simple poles and 
poles of order $\ge 2$ of $\phi$ are denoted by 
$P_0$, $P_{\rm s}$ and $P_{\infty}$, respectively.
We also set $P = P_0 \cup P_{\rm s} \cup P_{\infty}$. 

In our previous paper \cite{Iwaki14a} which disclosed a 
relationship between exact WKB analysis and cluster algebras, 
we assumed that $\phi$ does {\em not} have any simple poles. 
In this paper we will extend the results of \cite{Iwaki14a} 
to the case that $\phi$ may have simple poles. 

We impose a similar assumption as \cite{Iwaki14a}
for the potential $Q(z,\eta)$ in \eqref{eq:Sch}.
\begin{ass} \label{ass:zeros-and-poles}
Let $\phi$ be the quadratic differential associated with \eqref{eq:Sch}.
\begin{itemize}
\item %
$\phi$ has at least one zero, and at least one pole. %
\item %
All zeros of $\phi$ are simple. %
\item %
Suppose that a point $p \in \Sigma$ is a pole of $Q_{2}(z)$.
Then, $p$ is a pole of $\phi$. %
\item %
For any simple pole $s$, $Q_{2}(z)$ has a pole of order at most 2 at $s$. 
\item %
Suppose that a point $p \in P_{\infty}$ is a pole of $\phi$ of order $2$, 
and $z$ be a local coordinate around $p$ satisfying $z(p)=0$. 
Then, $Q_{2}(z)$ has a pole of order $2$ at $z=0$ and satisfies
\begin{equation}
Q_{2}(z) = -\frac{1}{4z^{2}} \left( 1 + O(z) \right) ~~
\text{as $z \rightarrow 0$}.
\end{equation} %
\item %
Suppose that a point $p \in P_{\infty}$ is a pole of $\phi$ 
of order $m \ge 3$. Then, $Q_2(z)$ may have pole at $p$ and 
\begin{equation}
(\text{pole order of $Q_2(z)$ at $p$}) < 1 + \frac{m}{2}.
\end{equation} 
\end{itemize}
\end{ass}

\subsection{WKB solutions}

For the equation \eqref{eq:Sch}, we can construct a pair 
of formal solutions, called the {\em WKB solutions}, 
in the following form (see \cite[Section 2]{Kawai05}):
\begin{equation} \label{eq:WKBsol}
\psi_{\pm}(z,\eta) = \frac{1}{\sqrt{S_{\rm odd}(z,\eta)}}
\exp\left(\pm\int^{z}S_{\rm odd}(z,\eta)~dz\right),
\end{equation}
where $S_{\rm odd}(z,\eta)$
is a formal (Laurent) series in $\eta^{-1}$ defined as 
the {\em odd part} of the formal solution 
$S(z,\eta) = \eta S_{-1}(z) + S_0(z) + \eta^{-1} S_1(z) + \cdots$ 
of the Riccati equation 
\begin{equation} \label{eq:Riccati}
\frac{dS}{dz}+S^{2} = \eta^{2}Q(z,\eta)
\end{equation}
associated with \eqref{eq:Sch}. Namely, if we denote by 
$S^{(\pm)}(z,\eta) = \pm \eta \sqrt{Q_{0}(z)} + \cdots$ 
the two formal solutions of \eqref{eq:Riccati}, 
then their odd part and even part are defined by
\begin{equation} \label{eq:Sodd-and-Seven}
S_{\rm odd}(z,\eta) = \frac{1}{2}\left( 
S^{(+)}(z,\eta) - S^{(-)}(z,\eta) \right),\quad
S_{\rm even}(z,\eta) = \frac{1}{2}
\left(S^{(+)}(z,\eta) + S^{(-)}(z,\eta) \right).
\end{equation}

It is known that the (formal series valued) 1-form 
$S_{\rm odd}(z,\eta)dz$ transforms as 
\begin{equation} \label{eq:Sodd-is-coordinate-free}
\tilde{S}_{\rm odd}(\tilde{z},\eta) = 
S_{\rm odd}\bigl( z(\tilde{z}),\eta \bigr) 
\frac{dz(\tilde{z})}{d\tilde{z}}
\end{equation}
under the coordinate transformation $z = z(\tilde{z})$. 
(See \cite[Corollary 2.17]{Kawai05}, \cite[Proposition 2.7 (b)]{Iwaki14a}.)
Here $\tilde{S}_{\rm odd}(\tilde{z},\eta)$ is the odd part of 
the formal solution of the Riccati equation associated 
with \eqref{eq:Sch-w}. 
Therefore, the 1-form $S_{\rm odd}(z,\eta)dz$ is globally defined 
(but multi-valued) on $\Sigma\setminus{P}$, and 
its integrals are independent of the choice of the local coordinate. 

Each coefficient of $S_{\rm odd}(z,\eta)$
may have singularities at points in $P$. 
However, for any $a \in P_{0}$ (resp., for any $s \in P_{\rm s}$), 
we can define the integral 
\[
\int_a^z S_{\rm odd}(z,\eta)dz \quad  
\left({\rm resp.,} \int_s^z S_{\rm odd}(z,\eta)dz  \right)
\]
in the sense of contour integral 
(see \cite[Section 2]{Kawai05}). 
On the other hand, we cannot define an integral of
$S_{\rm odd}(z,\eta)dz$ from any point $p \in P_{\infty}$ because 
the principal term $\eta \sqrt{Q_0(z)}dz$ is singular at $p$. 
However, subtracting the principal term, we can define an 
integral from $p$ due to Assumption \ref{ass:zeros-and-poles}.
\begin{prop} [{\cite[Proposition 2.8]{Iwaki14a}}]
\label{prop:Sodd-reg}
Under Assumption \ref{ass:zeros-and-poles}, 
for any point $p \in P_{\infty}$ and any local coordinate $z$ of 
$\Sigma$ around $p$ such that $z=0$ at $p$,  the formal power 
series valued 1-form 
\begin{equation} \label{eq:Sodd-reg}
S_{\rm odd}^{\rm reg}(z,\eta) dz := \left( S_{\rm odd}(z,\eta) 
- \eta \sqrt{Q_0(z)} \right) dz
\end{equation}
is integrable at $z = 0$. Namely, for any $n \ge 0$, 
there exists a real number $\ell > -1$ such that 
\begin{equation}
S_{{\rm odd},n}(z) = O(z^{\ell}) \quad \text{as $z\rightarrow 0$},
\end{equation}
where $S_{{\rm odd},n}(z)$ is the coefficient of $\eta^{-n}$ 
in $S^{\rm reg}_{\rm odd}(z,\eta)$. 
Moreover, all coefficients of $S_{\rm odd}^{\rm reg}(z,\eta)$
are holomorphic at $p$ if it is an even order pole of $\phi$.
\end{prop}
We call $S_{\rm odd}^{\rm reg}(z,\eta)$ the {\em regular part}
of $S_{\rm odd}(z,\eta)$. Proposition \ref{prop:Sodd-reg}
implies that the integral 
\[
\int_p^z S_{\rm odd}^{\rm reg}(z,\eta) dz 
= \int_p^z \left( S_{\rm odd}(z,\eta) - 
\eta \sqrt{Q_0(z)} \right) dz 
\]
from any point $p \in P_{\infty}$ is well-defined. 

\subsection{Borel resummation method}

By the formal series expansion, we have
\begin{equation} \label{eq:WKBsol3}
\psi_{\pm}(z,\eta)=\exp\left(\pm\eta\int^{z}
\sqrt{Q_{0}(z)}~dz\right) \eta^{-1/2}
\sum_{k=0}^{\infty} \eta^{-k} \psi_{\pm, k}(z). 
\end{equation}
The expansion \eqref{eq:WKBsol3} is a divergent series 
of $\eta^{-1}$ in general. 
To give \eqref{eq:WKBsol3} an
analytic interpretation, we employ the 
{\em Borel resummation method} (for a formal series of $\eta$). 
For the convenience of the readers, we give a rough definition 
of Borel resummation method. (See \cite{Costin09} for details.)
\begin{defn}
[{e.g., \cite[\S 1]{Kawai05}}] 
If the Laplace integral 
\begin{equation} \label{eq:Borel sum}
{\mathcal S}[\psi_{\pm}](z,\eta)=
\int_{\mp {\mathfrak s}(z)}^{\infty}
e^{-\eta \hspace{+.1em} y}\psi_{\pm,B}(z,y) dy
\end{equation}
is well-defined and becomes an analytic function 
of sufficiently large $\eta > 0$, and also becomes a 
holomorphic function of $z$ in a neighborhood 
of a point $z_{0} \in \Sigma\setminus P$, 
we call \eqref{eq:Borel sum} the {\em Borel sum} of 
the WKB solution $\psi_{\pm}(z,\eta)$. 
Here ${\mathfrak s}(z) = \int^{z} \sqrt{Q_{0}(z)}dz$ and 
\begin{equation}\label{eq:Bore-transform}
\psi_{\pm,B}(z,y) = \sum_{k=1}^{\infty}
\frac{\psi_{\pm,k}(z)}{\Gamma(k+1/2)}\hspace{+.1em}
(y\pm{\mathfrak s}(z))^{k-1/2}
\end{equation}
is the {\it Borel transform} of $\psi_{\pm}(z,\eta)$. 
The path of the integral \eqref{eq:Borel sum}
is taken along a straight line parallel to 
the positive real axis. 
\end{defn}

We also use a simplified notation $\Psi_{\pm} = {\mathcal S}[\psi_{\pm}]$ 
for the Borel sums of the WKB solutions.
The convergence of the Borel transform $\psi_{\pm,B}(z,y)$  
near $y=\mp{\mathfrak s}(z)$ can be shown easily 
(\cite[Lemma 2.5]{Kawai05}). 
The Borel summability of $\psi_{\pm}(z,\eta)$ demands that 
$\psi_{\pm,B}(z,y)$ can be extended to an analytic function of $y$ 
on a domain containing 
$\{y \in {\mathbb C}~|~ {\rm Re}(y) \ge 0 \}$, 
and grows at most exponentially when $|y| \rightarrow +\infty$ 
on the domain. 
If the Borel sum is well-defined, it is asymptotically expanded 
to the original WKB solution as $\eta \rightarrow +\infty$. 

\subsection{Stokes graphs, Stokes segments and Borel summability}
\label{section:Stokes-graph}

Many properties of the WKB solutions, such as the Borel summability
(i.e., well-definedness of the Borel sum \eqref{eq:Borel sum}),  
can be read off from the geometry of {\em Stokes graph},  
which is a graph described by {\em trajectories} of 
the quadratic differential $\phi$ defined in \eqref{eq:quad-diff}.
Here a trajectory of a quadratic differential $f(z)dz^{\otimes2}$ 
is a leaf of the foliation on $\Sigma$ defined by the equation 
\[
{\rm Im} \int^z \sqrt{f(z)} dz = {\rm constant}.
\]
See \cite{Strebel84} and \cite{Bridgeland13} for properties 
of trajectories of quadratic differentials.

\begin{defn}
[{\cite[Definition 2.6]{Kawai05}, \cite{Koike00}}]
\label{def:Stokes-curve}
Let $\phi$ be the quadratic differential 
associated with \eqref{eq:Sch}.
A {\em Stokes curve} of \eqref{eq:Sch} is a trajectory of $\phi$ 
one of whose end-points is a turning point or a simple pole 
of \eqref{eq:Sch}. Namely, in a local coordinate $z$ on $\Sigma$, 
Stokes curves emanating from $a \in P_0 \cup P_{\rm s}$ are defined as
\begin{equation} \label{eq:StokesCurves}
{\rm Im} \int_{a}^{z}\sqrt{Q_{0}(z)}~dz =0.
\end{equation}
\end{defn}

In the whole of the article, we assume the following:
\begin{ass} \label{ass:trajectory}
The quadratic differential $\phi$ associated with \eqref{eq:Sch} 
has {\em no recurrent trajectory}. Here a recurrent trajectory 
is a trajectory which is not a closed trajectory and has the 
limit set consisting of more than one point in at least one direction
(see \cite[Section 3.4]{Bridgeland13}). 
\end{ass}

Under Assumption \ref{ass:trajectory}, any Stokes curve emanates 
from a turning point or a simple pole flows into a pole of $\phi$ 
or a turning point (see \cite[Section 3.4]{Bridgeland13}). 
In other words, turning points, poles of $\phi$ and Stokes curves 
define a graph on the Riemann surface $\Sigma$. 

\begin{defn}
\begin{itemize}
\item %
The {\em Stokes graph} of \eqref{eq:Sch} is a graph on $\Sigma$ 
whose vertices consist of zeros and poles of $\phi$, 
and whose edges are Stokes curves. 
\item %
Faces of the Stokes graph are called {\em Stokes regions}. 
\end{itemize}
\end{defn}

  \begin{figure}
  \begin{minipage}{0.31\hsize}
  \begin{center}
  \includegraphics[width=35mm]{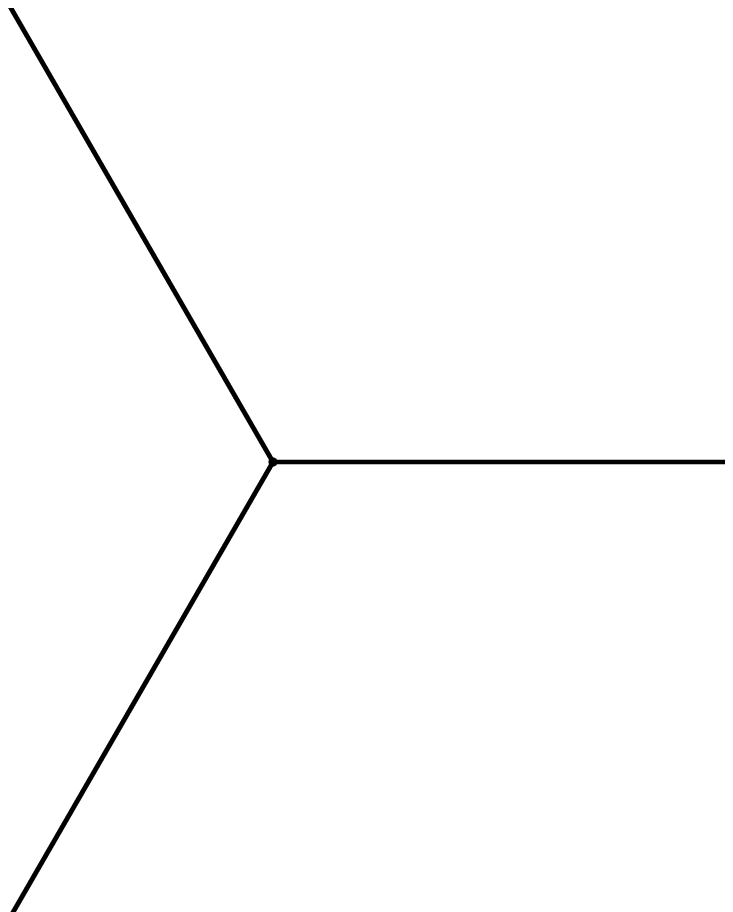} \\
  {\small (a) : $z$.}
  \end{center}
  \label{fig:Airy}
  \end{minipage} 
  \begin{minipage}{0.31\hsize}
  \begin{center}
  \includegraphics[width=35mm]{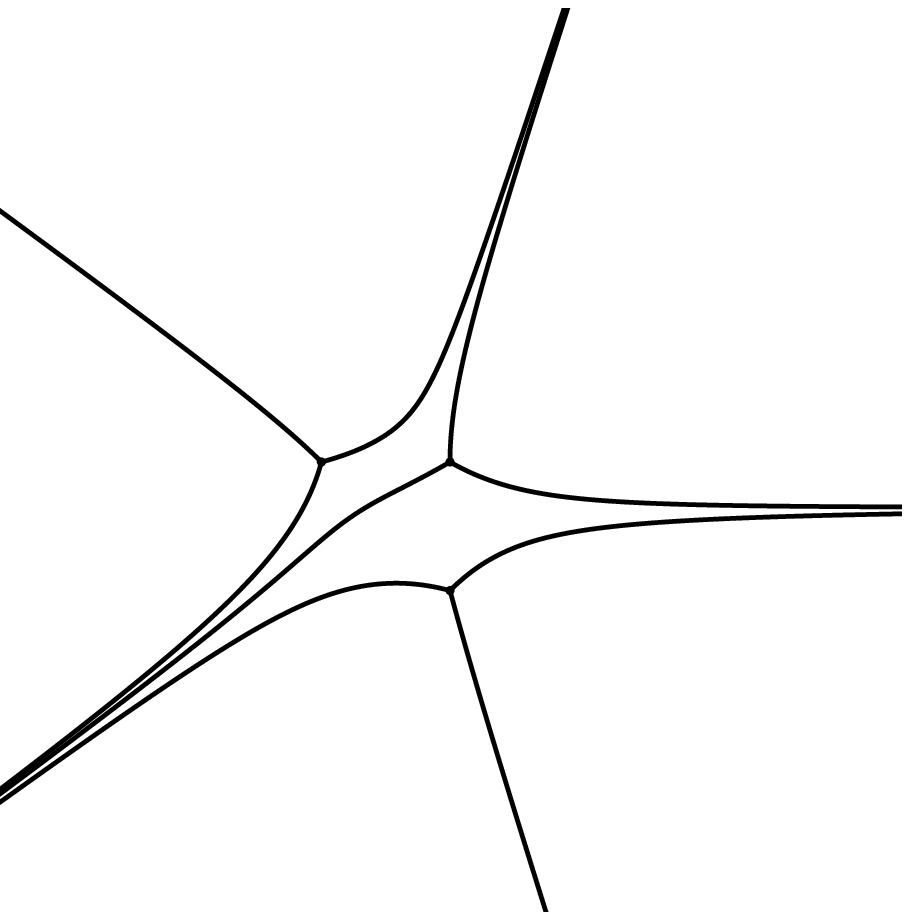} \\
  {\small (b) : $z(z+1)(z+i)$.}
  \end{center}
  \label{fig:cubic}
  \end{minipage}  
  \begin{minipage}{0.31\hsize}
  \begin{center}
  \includegraphics[width=35mm]{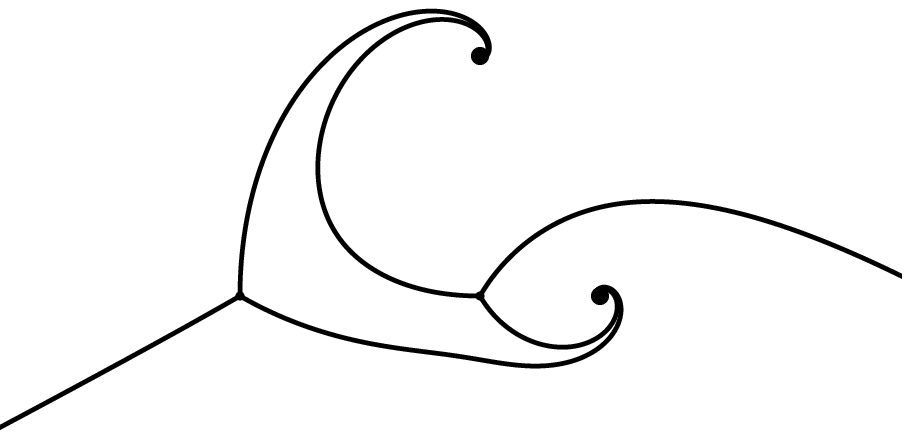} \\[+.2em]
  {\small (c) : $\displaystyle \frac{z(z+2)}{(z-1)^{2}(z-2i)^{2}}$.}
  \end{center}
  \label{fig:hypergeometric}
  \end{minipage}  
\vspace{+.5em}
  \begin{minipage}{0.31\hsize}
  \begin{center}
  \includegraphics[width=35mm]{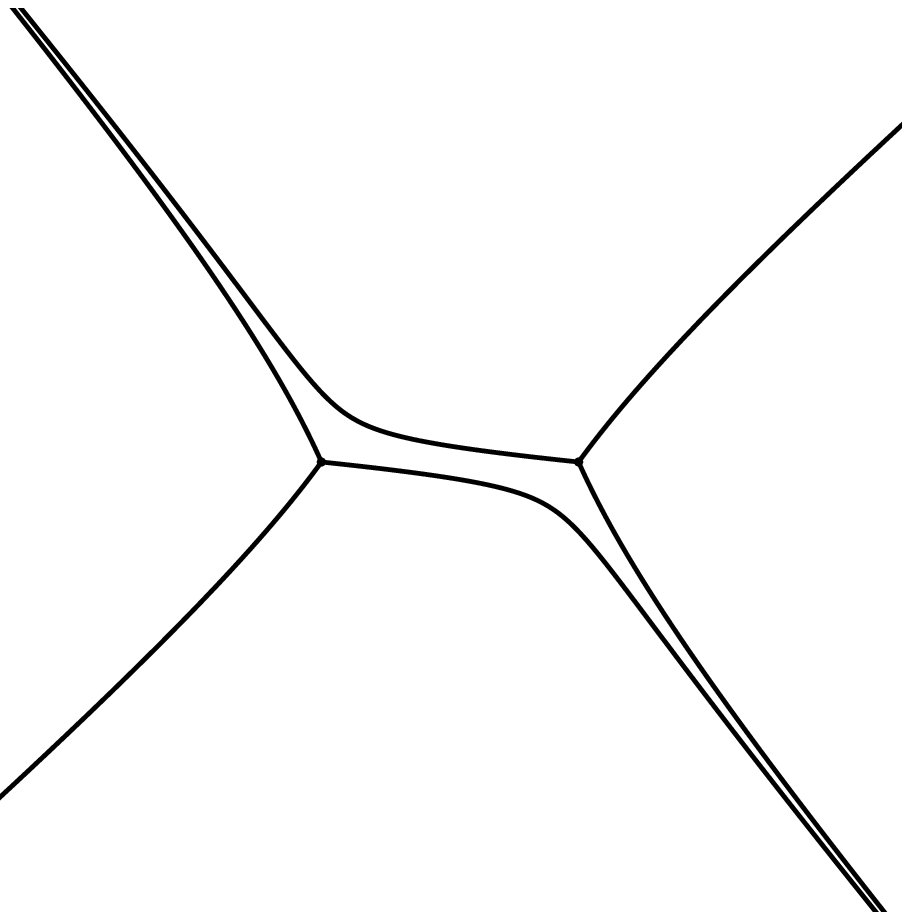} \\
  {\small (d) : $e^{+i\pi/10}(1-z^{2})$.}
  \end{center}
  \label{fig:Weber-minus}
  \end{minipage} 
  \begin{minipage}{0.31\hsize}
  \begin{center}
  \includegraphics[width=35mm]{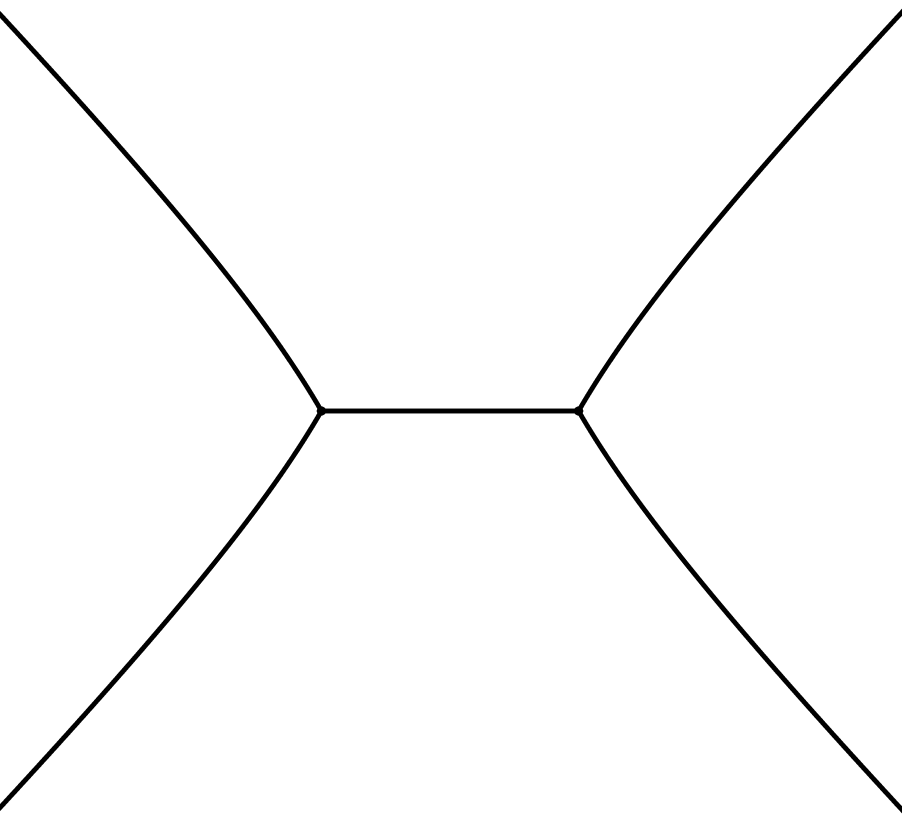} \\
  {\small (e) : $1-z^{2}$.}
  \end{center}
  \label{fig:Weber-0}
  \end{minipage} 
  \begin{minipage}{0.31\hsize}
  \begin{center}
  \includegraphics[width=35mm]{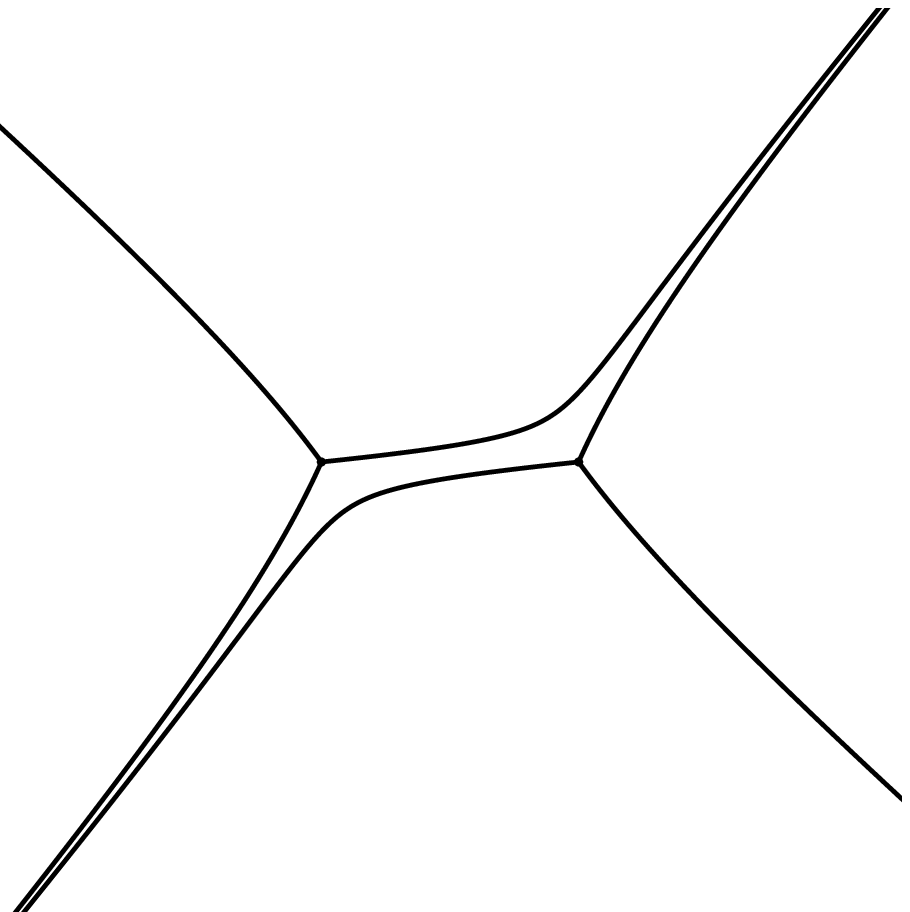} \\
  {\small (f) : $e^{-i\pi/10}(1-z^{2})$.}
  \end{center}
  \label{fig:Weber-plus}
  \end{minipage}   
\vspace{+.5em}
  \begin{minipage}{0.31\hsize}
  \begin{center}
  \includegraphics[width=40mm]{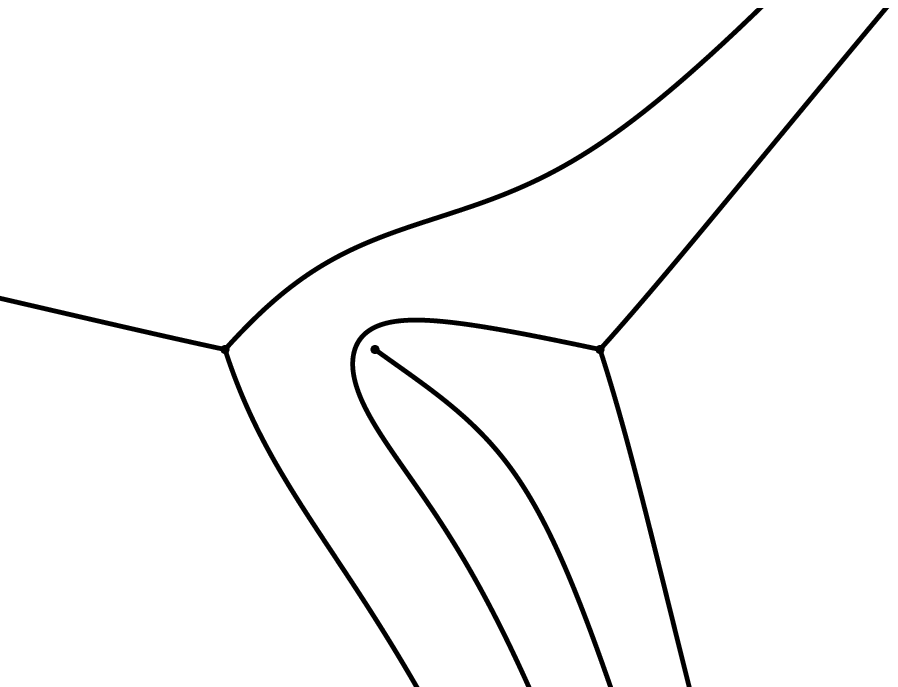} \\[+.3em]
  {\small (g) : $\displaystyle -\frac{e^{+\pi i/5}(z-2)(z+3)}{z+1}$}
  \end{center}
  \label{fig:saddle2-plus}
  \end{minipage} 
  \begin{minipage}{0.31\hsize}
  \begin{center}
  \includegraphics[width=40mm]{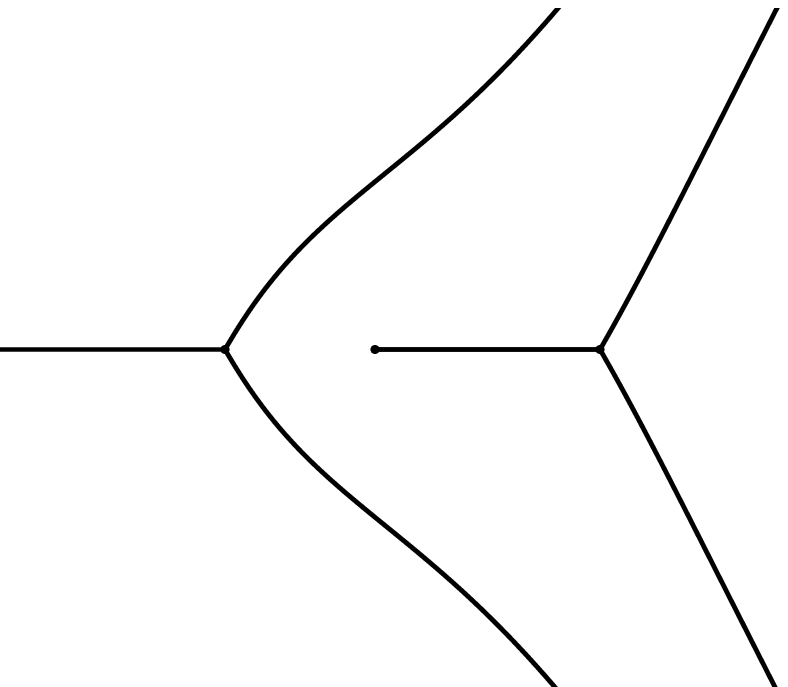} \\[+.3em]
   {\small (h) : $\displaystyle -\frac{(z-2)(z+3)}{z+1}$}
  \end{center}
  \label{fig:saddle2-0}
  \end{minipage} 
  \begin{minipage}{0.31\hsize}
  \begin{center}
  \includegraphics[width=40mm]{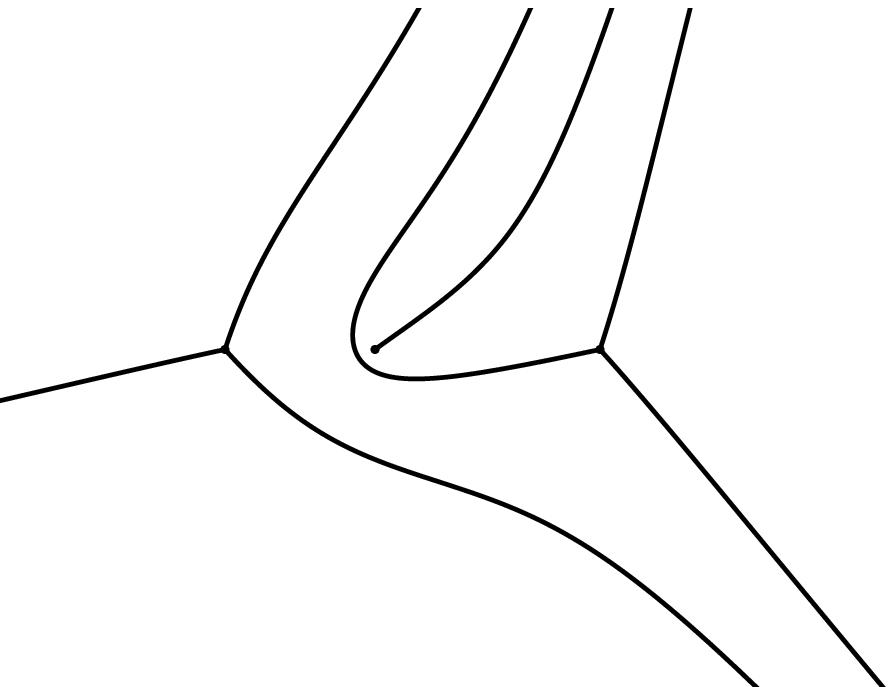} \\[+.3em]
  {\small (i) : $\displaystyle -\frac{e^{-\pi i/5}(z-2)(z+3)}{z+1}$}
  \end{center}
  \label{fig:saddle2-minus}
  \end{minipage}   
\caption{Examples of Stokes graphs. 
The rational functions represent $Q_0(z)$.}  
\label{fig:examples-of-Stokes-graphs}
 \end{figure}

The Stokes graph is defined only from the principal term 
$Q_0(z)$ in the potential (or from the quadratic differential $\phi$). 
We sometimes write $G=G(\phi)$ for the Stokes graph 
when we emphasize the dependence on $\phi$. 

\begin{figure}
\begin{center}
\begin{pspicture}(-1.73,-1.5)(6.7,2)
\psset{linewidth=0.5pt}
%
\psline(0,0)(0,2)
\psline(0,0)(1.73,-1)
\psline(0,0)(-1.73,-1)
\psline(4.63,0)(7.2,0)
\pscurve(1.35,-1)(0,-0.5)(-1.35,-1)
\pscurve(0.2,1.6)(0.4325,0.25)(1.53,-0.6)
\pscurve(-0.2,1.6)(-0.4325,0.25)(-1.53,-0.6)
\pscurve(6.75,0.8)(3.8,0)(6.75,-0.8)
\rput[c]{0}(0,0){$\times$}
\rput[c]{0}(4.5,0){$\otimes$}
\rput[c]{0}(0,-1.7){{(a) around a turning point}}
\rput[c]{0}(5.5,-1.7){{(b) around a simple pole}}
\end{pspicture}
\end{center}
\caption{Foliations around a turning point and a simple pole.}
\label{fig:foliation1}
\end{figure}
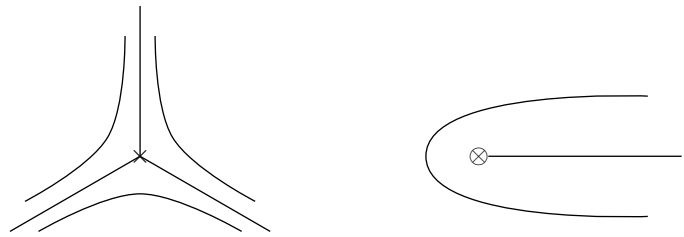

Figure \ref{fig:examples-of-Stokes-graphs} 
depicts examples of Stokes graphs on ${\mathbb P}^{1}$.
From each turning point (resp., simple pole) 
{\em three} Stokes curves (resp., {\em one} Stokes curve)
emanate. Trajectories behave as in Figure \ref{fig:foliation1} 
near a turning point and a simple pole. 
In what follows, we use the symbols $\times$ and $\otimes$
for a turning point and a simple pole, respectively. 
Also, the symbol $\bullet$ is used to represent a point in $P_{\infty}$.

\begin{defn}
\begin{itemize}
\item %
A Stokes curve connecting points $a$ and $b$ 
($a, b \in P_{0}\cup P_{\rm s}$)
is called a {\em Stokes segment}. %
\item %
If the Stokes graph does not contain any Stokes segment, 
the Stokes graph (or the quadratic differential $\phi$)
is said to be {\em saddle-free}. 
\end{itemize}
\end{defn}
A Stokes segment is nothing but a {\em saddle trajectory} of $\phi$ 
(see \cite{Strebel84}). Typical examples of Stokes segments 
are shown in Figure \ref{fig:Stokes-segments}: 
\begin{itemize}
\item %
A Stokes segment of {\em type I} connects two distinct turning points 
$a_1$ and $a_2$. Figure \ref{fig:examples-of-Stokes-graphs} (e) depicts 
an example of a type I Stokes segment. 
\item %
A Stokes segment of {\em type II} connects a simple pole $s$
and a turning point $a$. Figure \ref{fig:examples-of-Stokes-graphs} (h) 
depicts an example of a type II Stokes segment. 
\item %
A Stokes segment of {\em type III} forms a closed loop whose 
end point is a turning point $a$. 
It appears around a double pole $p$ of $\phi$. 
\end{itemize}

\begin{figure}
\begin{center}
\begin{pspicture}(-2,-1.5)(9.,1.2)
\psset{linewidth=0.5pt}
\psset{fillstyle=solid, fillcolor=black}
\psset{fillstyle=none}
\psline(-0.6,0)(0.6,0)
\psline(-0.6,0)(-1,0.8)
\psline(-0.6,0)(-1,-0.8)
\psline(0.6,0)(1,0.8)
\psline(0.6,0)(1,-0.8)
\rput[c]{0}(-0.6,0){$\times$}
\rput[c]{0}(0.6,0){$\times$}
\rput[c]{0}(-0.45,-0.3){$a_1$}
\rput[c]{0}(0.47,-0.3){$a_2$}
\rput[c]{0}(0,-1.3){type I}
\psset{fillstyle=solid, fillcolor=black}
\psset{fillstyle=none}
\rput[c]{0}(2.7,0){$\otimes$}
\rput[c]{0}(3.9,0){$\times$}
\rput[c]{0}(2.7,-0.3){$s$}
\rput[c]{0}(3.8,-0.3){$a$}
\rput[c]{0}(3.4,-1.3){type II}
\psline(2.83,0)(3.9,0)
\psline(3.9,0)(4.3,0.8)
\psline(3.9,0)(4.3,-0.8)
\psset{fillstyle=solid, fillcolor=black}
\psset{fillstyle=none}
\rput[c]{0}(6.1,0){$\bullet$}
\rput[c]{0}(7.1,0){$\times$}
\rput[c]{0}(6.1,-0.3){$p$}
\rput[c]{0}(7.1,-0.3){$a$}
\rput[c]{0}(6.6,-1.3){type III}
\psline(7.1,0)(7.6,0)
\pscurve(7.1,0)(6.6,0.4)(6.0,0.5)(5.6,0)
(6.0,-0.5)(6.6,-0.4)(7.1,0)
\end{pspicture}
\end{center}
\caption{Examples of Stokes segments. Each figure designates 
a part of the Stokes graph.}
\label{fig:Stokes-segments}
\end{figure}
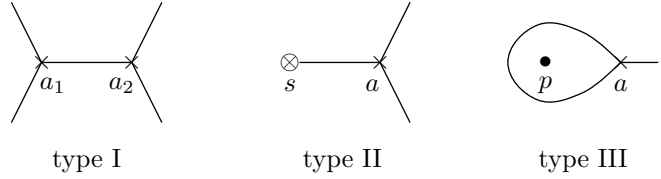

\begin{thm}
[\cite{Koike-Schafke}]\label{thm:summability}
Suppose that the potential $Q(z,\eta)$ of \eqref{eq:Sch} satisfies 
Assumptions \ref{ass:zeros-and-poles} and \ref{ass:trajectory}. 
If a path $\beta$ on $\Sigma\setminus (P_{0} \cup P_{\rm s})$ 
{\em never} intersects with any Stokes segment in the Stokes graph, 
then the formal power series 
$\int_{\beta} S_{\rm odd}^{\rm reg}(z,\eta) dz$ is Borel summable.  
Moreover, if the Stokes graph is {\em saddle-free}, 
then the WKB solutions are Borel summable
on each Stokes region. The Borel sum of the WKB solutions 
give analytic solutions of the Schr{\"o}dinger equation \eqref{eq:Sch}, 
and they are asymptotically expanded to the original WKB solutions 
when $\eta \rightarrow +\infty$. 
\end{thm}

\begin{rem}
Theorem \ref{thm:summability} was proved in \cite{Koike-Schafke} 
when the underlying Riemann surface $\Sigma$ is the 
Riemann sphere ${\mathbb P}^1$. Since the proof by 
\cite{Koike-Schafke} only use local properties in 
a Stokes region, the statement can be generalized to 
the case where $\Sigma$ is a general compact Riemann surface.
\end{rem}

As is inferred from Theorem \ref{thm:summability}, the existence 
of a Stokes segment may break the Borel summability of 
the WKB solutions. Furthermore, a Stokes segment causes a ``jump"
for the WKB solutions etc. (See \cite[Section 7]{Voros83}; 
see also Section \ref{section:stokes-auto} below.)

\subsection{Connection formulas on Stokes curves}
\label{section:connection-formula}

In this subsection we assume that the Stokes graph $G$ is saddle-free. 
Then, by Theorem \ref{thm:summability}, we have the Borel sums of 
the WKB solutions on each Stokes region. 
Here we discuss {\em connection formulas} describing the relations 
between the Borel sums given in adjacent Stokes regions. 
The connection formulas will be used 
in the proof of Theorem \ref{thm:Stokes-auto-II} below. 

First, we recall {Voros' connection formula} on a Stokes curve 
emanating from a {\em turning point}. 
Let $a \in P_{0}$ be a turning point, and $C$ be a Stokes curve 
emanating from $a$. Take any two points $z_1$ and $z_2$ near $C$ 
satisfying the following conditions 
(see Figure \ref{fig:conection-problem1} (a)):
\begin{itemize}
\item %
$z_1$ and $z_2$ are contained in some Stokes regions $D_1$ and $D_2$, 
respectively. The Stokes curve $C$ is a common boundary of $D_1$ 
and $D_2$.
\item %
$D_2$ comes next to $D_1$ in the anticlockwise direction
with the reference point $a$. 
\item %
The line segment connecting $z_1$ and $z_2$ intersects with 
the Stokes graph at exactly one point $z_0$ which lies on $C$. 
\end{itemize}
Take the WKB solutions {\em normalized at the turning point $a$} 
(see \cite[Section 2]{Kawai05}):
\begin{equation} \label{eq:WKB-TP-normalized}
\psi_{\pm,a}(z,\eta) = \frac{1}{\sqrt{S_{\rm odd}(z,\eta)}}
\exp\left(\pm\int_{a}^{z} S_{\rm odd}(z,\eta)dz\right).
\end{equation}
Here we assume that the integral in \eqref{eq:WKB-TP-normalized} 
is defined for $z$ which lies in a small neighborhood 
of $z_1$ or $z_2$, and the path from $a$ to $z$ is given by 
a composition of the following two paths; 
one is the path from $a$ to $z_0$ along $C$, and the other one 
is the straight path from $z_0$ to $z$. 
In order to fix the normalization \eqref{eq:WKB-TP-normalized}
completely, we fix a branch of the square root $\sqrt{Q_0(z)}$ 
on $C$ after taking a branch cut as indicated 
Figure \ref{fig:conection-problem1} (a).
For $j = 1, 2$, let $\Psi_{\pm,a}^{D_j}$ be the Borel sum of 
$\psi_{\pm,a}$ defined in a neighborhood of $z_j$ 
included in the Stokes region $D_j$. We also denote by 
the same symbol $\Psi_{\pm,a}^{D_j}$ its analytic continuation 
to the whole $D_j$. Under the situation, we have the following
connection formula. 

\begin{thm}
[{\cite[Section 6]{Voros83}, \cite[Section 2]{Aoki91}}]
\label{thm:Voros-formula}
The analytic continuations of $\Psi_{\pm,a}^{D_1}$ to $D_2$
across the Stokes curve $C$ satisfy one of 
the following equations:
\begin{align}
\label{eq:Voros-formula-1} ~ \\ \nonumber
(i) : 
\begin{cases} 
\Psi^{D_1}_{+,a} = \Psi^{D_2}_{+,a} + i \Psi^{D_2}_{-,a} \\ 
\Psi^{D_1}_{-,a} = \Psi^{D_2}_{-,a}.
\end{cases}
\quad
(ii) : 
\begin{cases} 
\Psi^{D_1}_{+,a} = \Psi^{D_2}_{+,a}  \\ 
\Psi^{D_1}_{-,a} = \Psi^{D_2}_{-,a} + i \Psi^{D_2}_{+,a}.
\end{cases}
\end{align}
Here the case (i) occurs when 
${\rm Re}(\int_a^z \sqrt{Q_0(z)}dz) > 0$ on $C$, while 
the case (ii) occurs when 
${\rm Re}(\int_a^z \sqrt{Q_0(z)}dz) < 0$ on $C$.
\end{thm}

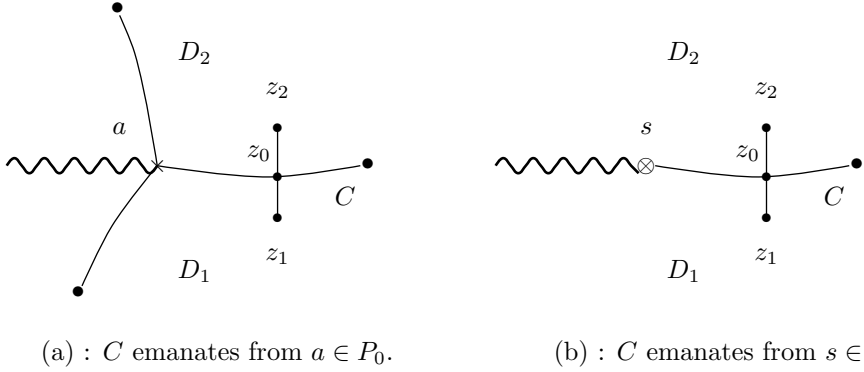
\begin{figure}
\begin{center}
\begin{pspicture}(0,-2.5)(12,3)
%
%
\psset{fillstyle=solid, fillcolor=black}
\psset{fillstyle=none}
\rput[c]{0}(2,0){\small $\times$}
\rput[c]{0}(4.5,-0.4){$C$}
\rput[c]{0}(1.5,0.5){$a$}
\rput[c]{0}(2.5,-1.4){$D_{1}$}
\rput[c]{0}(2.5,1.5){$D_{2}$}
\rput[c]{0}(4.8,0.02){$\bullet$}
\rput[c]{0}(1.47,2.1){$\bullet$}
\rput[c]{0}(0.95,-1.68){$\bullet$}
\rput[c]{0}(3.6,-0.7){\footnotesize $\bullet$}
\rput[c]{0}(3.6,+0.5){\footnotesize $\bullet$}
\rput[c]{0}(3.6,-0.15){\footnotesize $\bullet$}
\rput[c]{0}(3.6,-1.2){$z_1$}\rput[c]{0}(3.6,+1.0){$z_2$}
\rput[c]{0}(3.35,0.17){$z_0$}
\psset{linewidth=0.5pt}
\psline(3.6,-0.7)(3.6,+0.5)
\rput[c]{0}(2.8,-2.5)
{(a) : $C$ emanates from $a \in P_0$.}
\rput[c]{0}(8.5,0){\small $\otimes$}
\rput[c]{0}(11,-0.4){$C$}
\rput[c]{0}(8.5,0.5){$s$}
\rput[c]{0}(9.0,-1.4){$D_{1}$}
\rput[c]{0}(9.0,1.5){$D_{2}$}
\rput[c]{0}(11.3,0.02){$\bullet$}
\rput[c]{0}(10.1,-0.7){\footnotesize $\bullet$}
\rput[c]{0}(10.1,+0.5){\footnotesize $\bullet$}
\rput[c]{0}(10.1,-0.15){\footnotesize $\bullet$}
\rput[c]{0}(10.1,-1.2){$z_1$}\rput[c]{0}(10.1,+1.0){$z_2$}
\rput[c]{0}(9.85,0.17){$z_0$}
\psset{linewidth=0.5pt}
\psline(10.1,-0.7)(10.1,+0.5)
\rput[c]{0}(9.6,-2.5)
{(b) : $C$ emanates from $s \in P_{\rm s}$.}
\psset{linewidth=1pt}
\pscurve(2,0)(1.9,-0.1)(1.8,0)(1.7,0.1)(1.6,0)(1.5,-0.1)(1.4,0)
(1.3,0.1)(1.2,0)(1.1,-0.1)(1,0)(0.9,0.1)(0.8,0)(0.7,-0.1)(0.6,0)
(0.5,0.1)(0.4,0)(0.3,-0.1)(0.2,0)(0.1,0.1)(0,0)
\pscurve(8.4,-0.1)(8.3,0)(8.2,0.1)(8.1,0)(8,-0.1)(7.9,0)
(7.8,0.1)(7.7,0)(7.6,-0.1)(7.5,0)(7.4,0.1)(7.3,0)(7.2,-0.1)(7.1,0)
(7,0.1)(6.9,0)(6.8,-0.1)(6.7,0)(6.6,0.1)(6.5,0)
\psset{linewidth=0.5pt}
\pscurve(2,0)(3.5,-0.15)(4.7,0)
\pscurve(2,0)(1.7,1.5)(1.5,2)
\pscurve(2,0)(1.4,-0.8)(1,-1.6)
\pscurve(8.62,0)(10,-0.15)(11.2,0)
%
\end{pspicture}
\end{center}
\caption{Connection problems on Stokes curves. 
The wiggly lines designate branch cuts to define the 
branch of $\sqrt{Q_0(z)}$.} 
\label{fig:conection-problem1}
\end{figure}

A similar connection formula on a Stokes curve emanating from 
a {\em simple pole} is discovered by Koike. 
Suppose that we have the same situation as above after replacing 
``a turning point $a$" to ``a simple pole $s$" 
as indicated in Figure \ref{fig:conection-problem1} (b). 
Take the WKB solutions {\em normalized at the simple pole $s$} 
(\cite[Section 2]{Koike00})
\begin{equation}
\psi_{\pm,s}(z,\eta) = \frac{1}{\sqrt{S_{\rm odd}(z,\eta)}}\exp\left(
\pm\int_{s}^{z} S_{\rm odd}(z,\eta)dz\right), 
\end{equation}
and define their Borel sums $\Psi^{D_j}_{\pm,s}$ defined on $D_j$ 
($j = 1, 2$), by the same manner as above. 
Then, we have the following connection formula.
 
\begin{thm}
[{\cite[Theorem 2.1]{Koike00}}]
\label{thm:Koike-formula}
The analytic continuations of $\Psi_{\pm,a}^{D_1}$ to $D_2$
across the Stokes curve $C$ satisfy one of 
the following equations:
\begin{align}
\label{eq:Koike-formula-1} ~ \\ \nonumber
(i) : 
\begin{cases} 
\Psi^{D_1}_{+,s} = \Psi^{D_2}_{+,s} 
+ i (t + t^{-1}) \Psi^{D_1}_{-,s} \\[+.5em] 
\Psi^{D_1}_{-,s} = \Psi^{D_2}_{-,s}.
\end{cases}
\quad
(ii) : 
\begin{cases} 
\Psi^{D_1}_{+,s} = \Psi^{D_2}_{+,s}  \\[+.5em] 
\Psi^{D_1}_{-,s} = \Psi^{D_2}_{-,s} + i (t + t^{-1}) \Psi^{D_2}_{+,s}.
\end{cases}
\end{align}
Here 
\begin{equation} \label{eq:t}
t = t(s) = \exp (\pi i \sqrt{1+4b(s)}), \quad
b(s) = \lim_{z\rightarrow s}((z-s)^2 Q_2(z)),
\end{equation}
and the case (i) occurs when 
${\rm Re}(\int_s^z \sqrt{Q_0(z)}dz) > 0$ on $C$, while 
the case (ii) occurs when 
${\rm Re}(\int_s^z \sqrt{Q_0(z)}dz) < 0$ on $C$.
\end{thm}

\begin{rem} \normalfont
Theorems \ref{thm:Voros-formula} and \ref{thm:Koike-formula}
are proved by \cite{Voros83}, \cite{Aoki91} and \cite{Koike00}, 
respectively, in the case that $\Sigma = \bbP^1$ and the potential 
is a rational function. 
The proofs of \cite{Aoki91} and \cite{Koike00} are based on 
a formal local transformation to a certain canonical equation 
near a turning point and a simple pole. 
(We also need the Borel summability of the formal transformation 
series, which follows from the Borel summability of the WKB solutions
as is shown in \cite{Kamimoto11}.)
Since the transformations are constructed locally, 
the same discussion of \cite{Aoki91} and \cite{Koike00}
is applicable to case that $\Sigma$ 
is a general compact Riemann surface.  
Therefore, together with Theorem \ref{thm:summability}, 
Theorems \ref{thm:Voros-formula} and \ref{thm:Koike-formula} 
are valid when $\Sigma$ is a general compact Riemann surface.
\end{rem}

The connection formulas in Theorems \ref{thm:Voros-formula} 
and \ref{thm:Koike-formula} are quite effective for 
the study of global properties of solutions of the Schr{\"o}dinger equation. 
See \cite[Section 3]{Kawai05} for computation of the monodromy 
of a Fuchsian Schr{\"o}dinger equation via the above connection formulas.

\subsection{Voros symbols} \label{section:Voros-symbols}

Here we introduce an important notion, called {\em Voros symbols}. 
Voros symbols are formal series defined as the integral of 
$S_{\rm odd}(z,\eta)dz$ along a path or a cycle in the Riemann surface
$\hat{\Sigma}$ of $\sqrt{\phi}$. The Riemann surface $\hat{\Sigma}$
is a double cover of $\Sigma$ branching at 
odd order zeros and odd order poles of $\phi$.  

Denote by $\hat{P}_0$, $\hat{P}_{\rm s}$ and $\hat{P}_{\infty}$ 
the lift of $P_0$, $P_{\rm s}$ and $P_{\infty}$ on $\hat{\Sigma}$, 
respectively.  We also set 
$\hat{P} = \hat{P}_0 \cup \hat{P}_{\rm s} \cup \hat{P}_{\infty}$ and
\begin{eqnarray} 
\label{eq:H1-path}
H_1(\hat{\Sigma}\setminus(\hat{P}_0\cup\hat{P}_{\rm s}),\hat{P}_{\infty}) 
& := & 
H_1(\hat{\Sigma}\setminus(\hat{P}_0\cup\hat{P}_{\rm s}),\hat{P}_{\infty};\bbZ), \\
H_1(\hat{\Sigma}\setminus\hat{P}) 
& := &
H_1(\hat{\Sigma}\setminus\hat{P};\bbZ).
\end{eqnarray}
Here the one in \eqref{eq:H1-path} is a relative homology group. 
We call an element 
$\beta \in H_1(\hat{\Sigma}\setminus(\hat{P}_0\cup\hat{P}_{\rm s}),
\hat{P}_{\infty})$
and 
$\gamma \in H_1(\hat{\Sigma}\setminus\hat{P})$
a {\em path} and a {\em cycle}, respectively, to distinguish them. 
There is a natural inclusion 
$H_1(\hat{\Sigma}\setminus\hat{P}) \hookrightarrow 
H_1(\hat{\Sigma}\setminus(\hat{P}_0\cup\hat{P}_{\rm s}),\hat{P}_{\infty})$.

\begin{defn}
[{\cite[Section 2.1]{Delabaere93}}]
\begin{itemize}
\item %
Let $\beta \in H_1(\hat{\Sigma}\setminus(\hat{P}_0\cup\hat{P}_{\rm s}),
\hat{P}_{\infty})$ be a path. The formal power series 
\begin{equation} \label{eq:Voros-path}
W_{\beta}(\eta) = \int_{\beta} S^{\rm reg}_{\rm odd}(z,\eta) dz
\end{equation}
is called the {\em Voros coefficient for the path $\beta$}.
(Note that the integral of $S_{\rm odd}^{\rm reg}(z,\eta) dz$
along a path $\beta$ is well-defined by Proposition \ref{prop:Sodd-reg}.)
The formal series $e^{W_\beta(\eta)}$ is called 
the {\em Voros symbol for the path $\beta$}. 
\item %
Let $\gamma \in H_1(\hat{\Sigma}\setminus\hat{P})$ 
be a cycle. The formal series 
\begin{equation} \label{eq:Voros-cycle}
V_{\gamma}(\eta) = \oint_{\gamma} S_{\rm odd}(z,\eta) dz
\end{equation}
is called the {\em Voros coefficient for the cycle $\gamma$}.
The formal series $e^{V_{\gamma}(\eta)}$ is called 
the {\em Voros symbol for the cycle $\gamma$}. 
\end{itemize}
\end{defn}

Theorem \ref{thm:summability} gives a criterion 
for the Borel summability of the Voros symbols. 
\begin{prop}[{\cite{Koike-Schafke}, 
see also \cite[Corollary 2.21]{Iwaki14a}}]
\label{prop:Voros-Borel-sum} 
If a path $\beta \in H_1(\hat{\Sigma}\setminus
(\hat{P}_0\cup\hat{P}_{\rm s}),\hat{P}_{\infty})$ 
(resp., a cycle $\gamma \in H_1(\hat{\Sigma}\setminus\hat{P})$) 
never intersects with a Stokes segment in the Stokes graph $G$, 
then the Voros symbol for the path $\beta$ 
(resp., for the cycle $\gamma$) is Borel summable. 
In particular, if the Stokes graph $G$ is saddle-free, 
all Voros symbols are Borel summable. 
\end{prop}

The Borel sums of Voros symbols naturally appear in 
the expression of global connection formula for 
the WKB solutions (\cite[Section 3]{Kawai05}; 
see also Section \ref{section:proof-of-Stokes-auto} below).

\subsection{Mutation of Stokes graphs and jump formula for Voros symbols}
\label{section:stokes-auto}

The saddle-free condition is essential in the above results. 
Now we discuss effects of Stokes segments. 
Let us consider the situation that the Stokes graph 
$G(\phi)$ has a unique Stokes segment $\ell_0$. 
We restrict our discussion to the case that $\ell_0$ is of {\em type II} 
in Figure \ref{fig:Stokes-segments}. (The Stokes segment of type I 
(resp., type III) has already been analyzed by \cite{Delabaere93, Aoki09} 
(resp., by \cite{Aoki14, Iwaki14a}).) 

For the purpose, following \cite[Section 3.6]{Iwaki14a}, let us 
consider the 
{\em $S^1$-family} of potentials of the Schr{\"o}dinger equation 
defined by 
\begin{equation} \label{eq:S1-potentials}
Q^{(\theta)}(z,\eta) = e^{2i\theta}Q(z,e^{i\theta}\eta) \quad 
(\theta \in {\mathbb R}).
\end{equation}
Here $Q(z,\eta) = Q^{(0)}(z,\eta)$ is the original potential 
of \eqref{eq:Sch}. Denote by 
\begin{equation}
\label{eq:G-theta}
G_{\theta} = G(e^{2i\theta}\phi) 
\end{equation}
the Stokes graph of the Schr{\"o}dinger equation 
with the potential \eqref{eq:S1-potentials}. 
Here we assume that the potential \eqref{eq:S1-potentials} 
satisfies Assumption \ref{ass:trajectory} for any $\theta$.

The $S^1$-family \eqref{eq:S1-potentials} gives a {\em reduction} 
of the Stokes segment in the following sense. 
Note that, if a Stokes segment appears in the original
Stokes graph $G_0 = G(\phi)$ and it connects two points 
$a, b \in P_0 \cup P_{\rm s}$, then the equality
\begin{equation}
\int_{a}^{b} \sqrt{Q_0(z)} dz \in {\mathbb R}_{\ne 0}
\end{equation}
holds by the definition \eqref{eq:StokesCurves} of a Stokes curve. 
Hence, for any sufficiently small $\delta>0$, the Stokes segment 
$\ell_0$ is reduced (to two Stokes curves) 
in the Stokes graphs $G_{\pm \delta} = G(e^{\pm 2i \delta} \phi)$, 
and they become saddle-free (c.f., \cite[Section 5]{Bridgeland13}). 
Moreover, the topology of $G_{+\delta}$ and that of $G_{-\delta}$
are different as in Figure \ref{fig:saddle-reduction}. 
We call such a discontinuous change of the topology of 
Stokes graphs (caused by a reduction of a Stokes segment)
the {\em mutation of Stokes graphs}. 
This is a key phenomenon that relates exact WKB analysis 
to cluster algebras. 

\begin{figure}
\begin{center}
\begin{pspicture}(-6.2,-3.8)(6.2,-0.5)
%
%
%
%
%
\rput[c]{0}(0,-1.4){$G_0$}\rput[c]{0}(-4,-1.4){$G_{+\delta}$}
\rput[c]{0}(3.8,-1.4){$G_{-\delta}$}\rput[c]{0}(0,-2.85){$\ell_0$}
\rput[c]{0}(1.2,-1.2){$\bullet$}\rput[c]{0}(1.2,-3.6){$\bullet$}
\rput[c]{0}(-2.8,-1.2){$\bullet$}\rput[c]{0}(-2.85,-3.6){$\bullet$}
\rput[c]{0}(5.15,-1.2){$\bullet$}\rput[c]{0}(5.2,-3.6){$\bullet$}
\rput[c]{0}(-0.7,-2.4){$\otimes$}\rput[c]{0}(0.7,-2.4){$\times$}
\rput[c]{0}(-4.7,-2.4){$\otimes$}\rput[c]{0}(-3.3,-2.4){$\times$}
\rput[c]{0}(3.3,-2.4){$\otimes$}\rput[c]{0}(4.7,-2.4){$\times$}
\psset{linewidth=2.5pt}
\psline(-0.58,-2.4)(0.7,-2.4)
\psset{linewidth=0.5pt}
\psline(0.7,-2.4)(1.2,-1.2)\psline(0.7,-2.4)(1.2,-3.6)
\pscurve(-4.58,-2.4)(-3.5,-2.8)(-2.85,-3.6)
\pscurve(-3.3,-2.4)(-4.9,-2.0)(-5.1,-2.6)(-3.8,-3.1)(-2.9,-3.6)
\psline(-3.3,-2.4)(-2.8,-1.2)\psline(-3.3,-2.4)(-2.8,-3.6)
\pscurve(3.42,-2.4)(4.5,-2.0)(5.15,-1.2)
\pscurve(4.7,-2.4)(3.1,-2.8)(2.9,-2.2)(4.2,-1.7)(5.1,-1.2)
\psline(4.7,-2.4)(5.2,-1.2)\psline(4.7,-2.4)(5.2,-3.6)
\end{pspicture}
\end{center}
\caption{Reduction of a Stokes segment of type II and 
the mutation of Stokes graphs. 
The thick line designates the Stokes segment.}
\label{fig:saddle-reduction}
\end{figure}
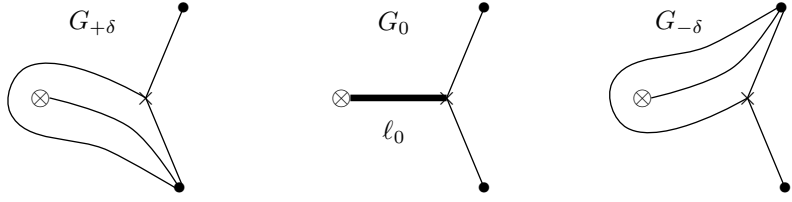

Denote by $e^{W^{(\theta)}_{\beta}} = e^{W^{(\theta)}_{\beta}(\eta)}$ 
and $e^{V^{(\theta)}_{\gamma}} = e^{V^{(\theta)}_{\gamma}(\eta)}$
the Voros symbols for the Schr{\"o}dinger equation with 
the potential \eqref{eq:S1-potentials}. 
As is shown in \cite[Lemma 3.8]{Iwaki14a}, we have
\begin{equation} \label{eq:WV-theta}
e^{W^{(\theta)}_{\beta}(\eta)} = 
e^{W_{\beta}(e^{i \theta}\eta)}, \quad
e^{V^{(\theta)}_{\gamma}(\eta)} = 
e^{V_{\gamma}(e^{i \theta}\eta)},
\end{equation}
where $e^{W_{\beta}(\eta)}$ and $e^{V_{\gamma}(\eta)}$ 
are the Voros symbols of the original equation \eqref{eq:Sch}. 
(Note that the Riemann surfaces $\hat{\Sigma}^{(\theta)}$ 
of $e^{i \theta}\sqrt{\phi}$ are common for all $\theta$. 
Hence we can identify the paths and cycles on $\hat{\Sigma}^{(\theta)}$ 
for different $\theta$ naturally.)
Since the Stokes graphs $G_{\pm \delta}$ are saddle-free,
Proposition \ref{prop:Voros-Borel-sum} implies that 
$e^{W^{(\pm \delta)}_{\beta}}$ and $e^{V^{(\pm \delta)}_{\gamma}}$
are Borel summable for any sufficiently small $\delta>0$. 
As is shown in \cite[Section 3.6]{Iwaki14a}, the limits 
$\delta \rightarrow +0$ of these Borel sums exist, 
and they are given by
\begin{equation} \label{eq:limits-Voros}
\lim_{\delta \rightarrow +0}{\mathcal S}[e^{W^{(\pm \delta)}_{\beta}}](\eta) 
= {\mathcal S}_{\pm}[e^{W_{\beta}}](\eta), 
\quad 
\lim_{\delta \rightarrow +0}{\mathcal S}[e^{V^{(\pm \delta)}_{\gamma}}](\eta)
= {\mathcal S}_{\pm}[e^{V_{\gamma}}](\eta).
\end{equation}
Here ${\mathcal S}_{\pm}[\bullet]$ means the 
{\em Borel sum in the direction $\pm \varepsilon$} 
for a sufficiently small $\varepsilon>0$; 
it is defined as a Laplace integral of the same form as 
\eqref{eq:Borel sum} whose integration path is taken 
along a half line $L_{\pm}$, which has a small angle $\mp \varepsilon$
as indicated in Figure \ref{fig:Borel-plane}. 
The Borel sums ${\mathcal S}_{\pm}[e^{W_{\beta}}]$ and 
${\mathcal S}_{\pm}[e^{V_{\gamma}}]$ give analytic functions 
of $\eta$ defined on $\{\eta \in {\mathbb R} ~|~ \eta \gg 1\}$.

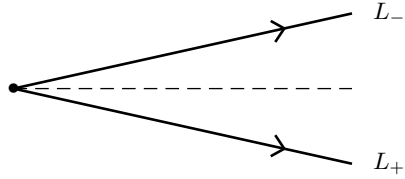
\begin{figure}
\begin{center}
\begin{pspicture}(-0.73,-1.2)(6.7,1.5)
\psset{linewidth = 0.5pt, linestyle = dashed}
\psline(0,0)(4.5,0)
\psset{linewidth=1pt, linestyle = solid}
\psline(0,0)(4.5,1)\psline(3.6,0.8)(3.42,0.9)\psline(3.6,0.8)(3.48,0.62)
\psline(0,0)(4.5,-1)\psline(3.6,-0.8)(3.42,-0.9)\psline(3.6,-0.8)(3.48,-0.62)
\rput[c]{0}(0,0){\small $\bullet$} 
\rput[c]{0}(5.,1){\small $L_{-}$} 
\rput[c]{0}(5.,-1){\small $L_{+}$} 
\end{pspicture}
\end{center}
\caption{Rotated paths of Laplace integral for 
${\mathcal S}_{\pm}[\bullet]$ on $y$-plane. The dashed line 
represents the half line with no angle.} 
\label{fig:Borel-plane}
\end{figure}

What is important here is that, the mutation of Stokes graphs 
yields a discrepancy (or jump) for these Borel sums: 
Namely, ${\mathcal S}_{+}[e^{W_{\beta}}] \ne {\mathcal S}_{-}[e^{W_{\beta}}]$ 
etc.\ may hold as analytic functions on 
$\{\eta \in {\mathbb R} ~|~ \eta \gg 1\}$ 
(c.f., \cite{Voros83, Delabaere93}). 

For the mutation of Stokes graphs relevant to a Stokes segment 
of {\em type I} and {\em type III}, the jump formulas describing 
the relationships between these Borel sums of the Voros symbols 
are known (see \cite{Delabaere93, Aoki14}). 
We find the following formula for a Stokes segment of {\em type II}, 
which is the first main result of this paper.

\begin{thm}
\label{thm:Stokes-auto-II}
Assume that the Stokes graph has a unique 
Stokes segment $\ell_0$, which is of type II, 
connecting a turning point and a simple pole $s$. 
Then, the limits of Borel sums \eqref{eq:limits-Voros} 
of the Voros symbols for any path $\beta$ and any cycle $\gamma$ 
satisfy the following relations as analytic functions of $\eta$ 
on $\{\eta \in {\mathbb R} ~|~ \eta \gg 1\}$: 
\begin{eqnarray} ~ \label{eq:Stokes-auto-II} \\ \nonumber 
\begin{cases}
{\mathcal S}_{-}[e^{{W}_{\beta}}] = {\mathcal S}_{+}[e^{{W}_{\beta}}]
\Bigl( 1 + (t + t^{-1}) {\mathcal S}_{+}[e^{{V}_{\gamma_0}}] + 
{\mathcal S}_{+}[e^{2{V}_{\gamma_0}}] 
\Bigr)^{-\langle\gamma_0,\beta\rangle} \\[+.7em]
{\mathcal S}_{-}[e^{{V}_{\gamma}}] = {\mathcal S}_{+}[e^{{V}_{\gamma}}]
\Bigl( 1 + (t + t^{-1}) {\mathcal S}_{+}[e^{{V}_{\gamma_0}}] + 
{\mathcal S}_{+}[e^{2{V}_{\gamma_0}}] \Bigr)^{-(\gamma_0,\gamma)}.
\end{cases}
\end{eqnarray}
Here 
\begin{itemize}
\item %
$t = t(s)$ is given by \eqref{eq:t}, 
\item %
$\gamma_0 \in H_1(\hat{\Sigma}\setminus\hat{P})$ 
is a cycle which surrounds the Stokes segment $\ell_0$
(see Figure \ref{fig:saddle-classes}) 
whose orientation is given by the condition
\begin{equation} \label{eq:saddle-class}
\oint_{\gamma_0} \sqrt{Q_0(z)} dz \in {\mathbb R}_{<0},
\end{equation}
\item %
$\langle~,~\rangle$ is the intersection form 
on the homology groups:
\begin{equation} \label{eq:intersection-form}
\langle~,~\rangle : H_1(\hat{\Sigma}\setminus\hat{P}) \times 
H_1(\hat{\Sigma}\setminus\hat{P}_{0},\hat{P}_{\infty}) \rightarrow \bbZ,
\end{equation}
normalized as 
\begin{equation}
\langle \text{$x$-axis}, \text{$y$-axis} \rangle = +1,
\end{equation}
\item %
$(~,~)$ is the restriction of $\langle~,~\rangle$ on 
$H_1(\hat{\Sigma}\setminus\hat{P}) \times 
H_1(\hat{\Sigma}\setminus\hat{P})$.
\end{itemize}  
\end{thm}

\begin{figure}
\begin{center}
\begin{pspicture}(2,-1)(5.3,1.2)
\psset{linewidth=0.5pt}
\psset{fillstyle=solid, fillcolor=black}
\psset{fillstyle=none}
%
%
%
%
%
\psset{linewidth=0.5pt}
\psset{fillstyle=solid, fillcolor=black}
\psset{fillstyle=none}
\rput[c]{0}(3,0){$\otimes$}
\rput[c]{0}(4.2,0){$\times$}
\rput[c]{0}(3.6,0.7){$\gamma_0$}
%
\psline(3.12,0)(4.2,0)
\psline(4.2,0)(4.7,1)
\psline(4.2,0)(4.7,-1)
\psset{linewidth=1.5pt}
\pscurve(4.2,0)(4.3,0.1)(4.4,0)(4.5,-0.1)(4.6,0)(4.7,0.1)
(4.8,0)(4.9,-0.1)(5.0,0)(5.1,0.1)
\pscurve(2.9,0.1)(2.8,0)(2.7,-0.1)(2.6,0)(2.5,0.1)
(2.4,0)(2.3,-0.1)(2.2,0)(2.1,0.1)
%
\psset{linewidth=1.pt, linecolor=magenta}
\pscurve(2.8,0)(2.9,0.23)(3.6,0.4)(4.3,0.23)(4.4,0)
\psset{linestyle=dashed, linecolor=magenta}
\pscurve(2.8,0)(2.9,-0.23)(3.6,-0.4)(4.3,-0.23)(4.4,0)
\end{pspicture}
\end{center}
\caption{The cycle $\gamma_0$ associated with a Stokes segment 
of type II. The solid part of $\gamma_0$ lies on 
the first sheet of $\hat{\Sigma}$, while the dotted part of 
$\gamma_0$ lies on the second sheet of $\hat{\Sigma}$.}
\label{fig:saddle-classes}
\end{figure}
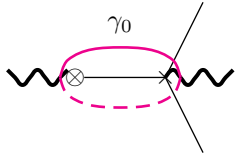

Theorem \ref{thm:Stokes-auto-II} is proved in next subsection. 

\begin{rem} \label{remark:known-results}
Koike and Takei \cite{Koike11} has already derived 
the same formula for the Whittaker equation, 
which is a particular example having a Stokes segment of type II. 
Theorem \ref{thm:Stokes-auto-II} generalizes the result 
of \cite{Koike11} to the general case.
Kamimoto et al. \cite{Kamimoto10} also analyze the singularity 
structure of Borel transform of the WKB solutions 
(via a ``WKB theoretic transformation") when a Stokes segment 
of type II appears. Theorem \ref{thm:Stokes-auto-II}
follows from the result of \cite{Kamimoto10} if their 
formal transformation series is Borel summable (see \cite{Kamimoto11}). 
\end{rem}

\begin{rem} \label{rem:Stokes-Voros} \label{rem:on-Stokes-phenomenon}
The jump \eqref{eq:Stokes-auto-II} for Voros symbols 
can be interpreted as the {\em Stokes phenomenon} 
(in $\eta \rightarrow +\infty$) for the Voros symbols since the both 
${\mathcal S}_{+}[e^{W_{\beta}}]$ and ${\mathcal S}_{-}[e^{W_{\beta}}]$ 
has $e^{W_{\beta}}$ as the asymptotic expansion for 
$\eta \rightarrow +\infty$.
Note that, {\em not} all Voros symbols jump; if a path $\beta$ or 
a cycle $\gamma$ never intersects with the Stokes segment $\ell_0$, 
the jump formula becomes trivial (i.e.,   
${\mathcal S}_{-}[e^{W_{\beta}}] = {\mathcal S}_{+}[e^{W_{\beta}}]$ etc.\/)
since the intersection number of such paths or cycles 
and $\gamma_0$ is equal to $0$.
\end{rem}

\subsection{Proof of Theorem \ref{thm:Stokes-auto-II}}
\label{section:proof-of-Stokes-auto}

The strategy of the proof is similar to the one for 
a type I or type III Stokes segment 
presented in \cite[Appendix A, B]{Iwaki14a}.
Let us consider two connection problems for 
certain WKB solutions indicated in 
Figure \ref{fig:two-connection-problems}.
The first one (A) is the connection problem 
from  the Stokes region $D_1^{+}$ to $D_2^{+}$ in 
the Stokes graph $G_{+\delta}$, while the second one (B)
is the connection problem from  
the Stokes region $D_1^{-}$ to $D_2^{-}$ 
(through the intermediate Stokes regions $D_3^{-}$ and $D_4^{-}$)
in the Stokes graph $G_{-\delta}$ along the thick paths 
depicted in Figure \ref{fig:two-connection-problems}. 
Here, after taking a branch cut as in 
Figure \ref{fig:two-connection-problems},
we have specified the branch of $\sqrt{Q_0(z)}$ so that 
the {\em signs} of Stokes curves are assigned as in 
Figure \ref{fig:two-connection-problems}, where the sign 
is defined as follows. 
To a Stokes curve emanating from $b \in P_{0} \cup P_{\rm s}$ 
in the Stokes graph $G_{\pm \delta}$, 
we assign $\oplus$ (resp., $\ominus$) if the sign of 
${\rm Re}( e^{\pm i\delta} \int_{b}^{z} \sqrt{Q_0(z)} dz)$ 
on the Stokes curve near its end-point is positive (resp., negative). 
We can show the desired formula \eqref{eq:Stokes-auto-II} 
in the same manner as presented here if the signs are assigned in 
the opposite way. We also note that, since the Stokes graph 
$G_{\pm \delta}$ is saddle-free, the point $p$ in Figure 
\ref{fig:two-connection-problems}, which is the end-point of 
a Stokes curve emanating from $a$, is a point in $P_{\infty}$.

\begin{figure}
\begin{center}
\begin{pspicture}(-5.2,-4.3)(2.5,0)
%
\rput[c]{0}(-3.85,-4.3){${\rm (A)} : G_{+\delta}$}
\rput[c]{0}(1.2,-4.3){${\rm (B)} : G_{-\delta}$} 
\rput[c]{0}(-2.75,-1.05){$\oplus$}
\rput[c]{0}(-2.78,-3.73){$\oplus$}
\rput[c]{0}(2.2,-1.05){$\oplus$}
\rput[c]{0}(2.25,-3.7){$\oplus$}
\rput[c]{0}(-4.7,-2.4){$\otimes$}\rput[c]{0}(-3.3,-2.4){$\times$}
\rput[c]{0}(0.3,-2.4){$\otimes$}\rput[c]{0}(1.7,-2.4){$\times$}
\rput[c]{0}(-2.1,-1.9){$D_1^{+}$}\rput[c]{0}(-4.2,-1.0){$D_2^{+}$}
\rput[c]{0}(2.9,-1.9){$D_1^{-}$}\rput[c]{0}(0.9,-0.8){$D_2^{-}$}
\rput[c]{0}(1.2,-2.5){$D_3^{-}$}\rput[c]{0}(1.0,-1.9){$D_4^{-}$}
\rput[c]{0}(-2.95,-2.7){$a$}\rput[c]{0}(2.05,-2.7){$a$}
\rput[c]{0}(-4.7,-2.72){$s$}\rput[c]{0}(0.3,-2.72){$s$}
\rput[c]{0}(-2.4,-0.85){$p$}\rput[c]{0}(2.6,-0.85){$p$}
\psset{linewidth=0.5pt}
\pscurve(-4.58,-2.4)(-3.5,-2.8)(-2.85,-3.6)
\pscurve(-3.3,-2.4)(-4.9,-2.0)(-5.1,-2.6)(-3.8,-3.1)(-2.9,-3.6)
\psline(-3.3,-2.4)(-2.8,-1.2)\psline(-3.3,-2.4)(-2.8,-3.6)
\pscurve(0.42,-2.4)(1.5,-2.0)(2.18,-1.2)
\pscurve(1.7,-2.4)(0.1,-2.8)(-0.1,-2.2)(1.2,-1.7)(2.1,-1.2)
\psline(1.7,-2.4)(2.25,-1.2)\psline(1.7,-2.4)(2.2,-3.6)
\psset{linewidth=1.5pt}
\pscurve(-3.3,-2.4)(-3.2,-2.5)(-3.1,-2.4)(-3.0,-2.3)(-2.9,-2.4)
(-2.8,-2.5)(-2.7,-2.4)(-2.6,-2.3)(-2.5,-2.4)
\pscurve(-4.8,-2.5)(-4.9,-2.4)(-5.0,-2.3)(-5.1,-2.4)
(-5.2,-2.5)(-5.3,-2.4)(-5.4,-2.3)(-5.5,-2.4)(-5.6,-2.5)(-5.7,-2.4)
\pscurve(1.7,-2.4)(1.8,-2.5)(1.9,-2.4)(2.0,-2.3)(2.1,-2.4)
(2.2,-2.5)(2.3,-2.4)(2.4,-2.3)(2.5,-2.4)
\pscurve(0.2,-2.5)(0.1,-2.4)(0,-2.3)(-0.1,-2.4)
(-0.2,-2.5)(-0.3,-2.4)(-0.4,-2.3)(-0.5,-2.4)(-0.6,-2.5)(-0.7,-2.4)
\psset{linewidth=2.5pt}
\psline[arrows=->](-2.6,-1.7)(-3.8,-1.2) 
\psline[arrows=->](2.4,-1.7)(1.1,-1.2) 
\end{pspicture}
\end{center}
\caption{Two connection problems. 
}
\label{fig:two-connection-problems}
\end{figure}
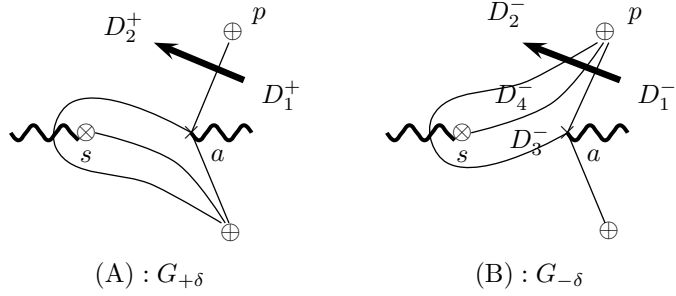

If $\psi_{\pm}(z,\eta)$ are the WKB solutions of \eqref{eq:Sch}, then 
\begin{equation} \label{eq:WKB-theta}
\psi_{\pm}^{(\theta)}(z,\eta) = \psi_{\pm}(z,e^{i\theta}\eta)
\end{equation}
are the WKB solutions for $S^1$-family of Sch{\"o}dinger equations 
with potential \eqref{eq:S1-potentials}. Take the WKB solutions 
\begin{equation} \label{eq:WKB-a}
\psi_{\pm,a}(z,\eta) = 
\frac{1}{\sqrt{S_{\rm odd}(z,\eta)}}
\exp\left(\pm\int_{a}^{z}
S_{\rm odd}(z,\eta) dz\right)
\end{equation}
of \eqref{eq:Sch}, which is normalized at the turning point $a$ 
depicted in Figure \ref{fig:two-connection-problems}.
We will compare the two connection formulas for 
$\psi^{(+\delta)}_{\pm,a}$ 
(in Figure \ref{fig:two-connection-problems} (A))
and for  $\psi^{(-\delta)}_{\pm,a}$ 
(in Figure \ref{fig:two-connection-problems} (B)).

We denote by $\Psi_{\pm,a}^{(\theta), D}$ the Borel sum of 
$\psi^{(\theta)}_{\pm,a}(z,\eta)$ in a Stokes region $D$. 
Using Theorem \ref{thm:Voros-formula}, we have 
the following formula for the first connection problem (A).
\begin{eqnarray}
{\rm (A)} : 
\begin{cases} \label{eq:conn-formula-plus}
\Psi^{(+\delta), D_1^{+}}_{+,a} = \Psi^{(+\delta), D_2^+}_{+,a} 
+ i \Psi^{(+\delta), D_2^+}_{-,a} \\[+.5em] 
\Psi^{(+\delta), D_1^+}_{-,a} = \Psi^{(+\delta), D_2^+}_{-,a}. 
\end{cases}
\end{eqnarray}

On the other hand, in the second connection problem (B)
we have to cross three Stokes curves emanating from 
$a$ and $s$ as in Figure \ref{fig:two-connection-problems}. 
The resulting connection formula is given as follows.
\begin{lem}
\label{lemma:connection2}
In the second connection problem (B) we have
\begin{align} 
~ \label{eq:conn-formula-minus} \\ \nonumber 
{\rm (B)} : 
\begin{cases} 
\Psi^{(-\delta), D_1^-}_{+,a} = \Psi^{(-\delta), D_2^-}_{+,a} 
+ i \left(1 + (t + t^{-1}) {\mathcal S}[e^{V^{(-\delta)}_{\gamma_0}} ]
+ {\mathcal S}[e^{2V^{(-\delta)}_{\gamma_0}} ] \right)
\Psi^{(-\delta), D_2^-}_{-,a} \\[+.5em] 
\Psi^{(-\delta), D_1^-}_{-,a} = \Psi^{(-\delta), D_2^-}_{-,a}.
\end{cases}
\end{align}
\end{lem}
\begin{proof}
Let us describe the connection formulas for WKB solutions 
near the three crossing points. 

\underline{$\bullet$ Connection formula for $D_1^{-} \leadsto D_3^{-}$.}~
Using Theorem \ref{thm:Voros-formula} again, we have 
\begin{eqnarray}
\begin{cases} \label{eq:conn-formula-minus1}
\Psi^{(-\delta), D_1^-}_{+,a} = \Psi^{(-\delta), D_3^-}_{+,a} 
+ i \Psi^{(-\delta), D_3^-}_{-,a} \\ 
\Psi^{(-\delta), D_1^-}_{-,a} = \Psi^{(-\delta), D_3^-}_{-,a}.
\end{cases}
\end{eqnarray}

\underline{$\bullet$ Connection formula for $D_3^{-} \leadsto D_4^{-}$.}~
Here we need to use Theorem \ref{thm:Koike-formula} which describes 
the connection formula on a Stokes curve emanating from the simple pole $s$.
In order to use Theorem \ref{thm:Koike-formula}, 
we need to change the normalization of 
the WKB solutions since Theorem \ref{thm:Koike-formula}
is only valid for the WKB solutions normalized at a simple pole.
For the WKB solution 
\begin{equation} \label{eq:WKB-s}
\psi_{\pm,s}(z,\eta) = 
\frac{1}{\sqrt{S_{\rm odd}(z,\eta)}}
\exp\left(\pm\int_{s}^{z}
S_{\rm odd}(z,\eta)dz\right)
\end{equation}
of \eqref{eq:Sch} normalized at $s$, the following formula 
holds on the Stokes curve in question:
\begin{eqnarray}
\begin{cases} \label{eq:conn-formula-simple-pole}
\Psi^{(-\delta), D_3^-}_{+,s} = \Psi^{(-\delta), D_4^-}_{+,s} 
+ i \hspace{+.1em} (t + t^{-1}) 
\hspace{+.1em} \Psi^{(-\delta), D_4^-}_{-,s} \\[+.5em] 
\Psi^{(-\delta), D_3^-}_{-,s} = \Psi^{(-\delta), D_4^-}_{-,s}.
\end{cases}
\end{eqnarray}
Here $\Psi_{\pm,s}^{(-\delta), D_3^-}$ etc.\ are the Borel sum of 
$\psi^{(-\delta)}_{\pm,s}(z,\eta)$ in the Stokes region $D_3^-$ etc.
Then, using the relation 
\begin{equation}
\psi^{(\theta)}_{\pm, a}(z,\eta)=
\exp\left(\pm\int_{a}^{s}
S_{\rm odd}(z,e^{i\theta}\eta)dz\right)\psi^{(\theta)}_{\pm, s}(z,\eta)  
= e^{\pm \frac{1}{2} V^{(\theta)}_{\gamma_0}(\eta)}
\psi^{(\theta)}_{\pm, s}(z,\eta)
\end{equation}
(here the cycle $\gamma_0$ satisfying \eqref{eq:saddle-class} 
has the orientation as depicted in 
Figure \ref{fig:path-and-cycles} (a), and 
we have used the equality \eqref{eq:WV-theta}), 
the formula \eqref{eq:conn-formula-simple-pole} is translated to 
a formula for the Borel sum of the WKB solutions 
$\psi^{(-\delta)}_{\pm,a}$ as follows:
\begin{eqnarray}
\begin{cases} \label{eq:conn-formula-minus2}
\Psi^{(-\delta), D_3^-}_{+,a} = \Psi^{(-\delta), D_4^-}_{+,a} 
+ i \hspace{+.1em} (t + t^{-1}) {\mathcal S}[e^{V^{(-\delta)}_{\gamma_0}}] 
\hspace{+.1em} \Psi^{(-\delta), D_4^-}_{-,a} \\[+.5em] 
\Psi^{(-\delta), D_3^-}_{-,a} = \Psi^{(-\delta), D_4^-}_{-,a}.
\end{cases}
\end{eqnarray}

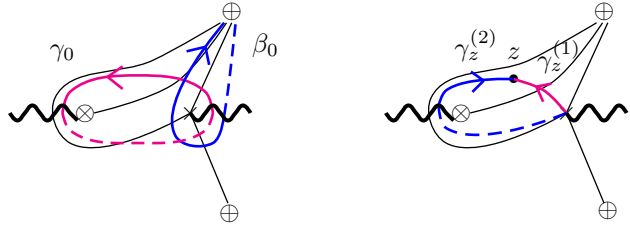
\begin{figure}
\begin{center}
\begin{pspicture}(-5.2,-4.3)(2.5,0)
%
\rput[c]{0}(-4.2,-4.5){(a) The cycle $\gamma_0$ and the path $\beta_0$.}
\rput[c]{0}(1.7,-4.5){(b) The path $\gamma^{(1)}_z$ and $\gamma^{(2)}_z$.}
\rput[c]{0}(-2.75,-1.05){$\oplus$}\rput[c]{0}(-2.78,-3.73){$\oplus$}
\rput[c]{0}(2.2,-1.05){$\oplus$}\rput[c]{0}(2.25,-3.7){$\oplus$}
\rput[c]{0}(-4.7,-2.4){$\otimes$}\rput[c]{0}(-3.3,-2.4){$\times$}
\rput[c]{0}(0.3,-2.4){$\otimes$}\rput[c]{0}(1.7,-2.4){$\times$}
\rput[c]{0}(-5,-1.6){$\gamma_0$}
\rput[c]{0}(-2.3,-1.5){$\beta_0$}
\rput[c]{0}(1,-1.95){\footnotesize $\bullet$}
\rput[c]{0}(1,-1.6){$z$}
\rput[c]{0}(1.6,-1.65){$\gamma^{(1)}_z$}
\rput[c]{0}(0.5,-1.5){$\gamma^{(2)}_z$}
\psset{linewidth=0.5pt}
\pscurve(0.42,-2.4)(1.5,-2.0)(2.18,-1.2)
\pscurve(1.7,-2.4)(0.1,-2.8)(-0.1,-2.2)(1.2,-1.7)(2.1,-1.2)
\psline(1.7,-2.4)(2.25,-1.2)\psline(1.7,-2.4)(2.2,-3.6)
\pscurve(-4.58,-2.4)(-3.5,-2.0)(-2.82,-1.2)
\pscurve(-3.3,-2.4)(-4.9,-2.8)(-5.1,-2.2)(-3.8,-1.7)(-2.9,-1.2)
\psline(-3.3,-2.4)(-2.75,-1.2)\psline(-3.3,-2.4)(-2.8,-3.6)
\psset{linewidth=1.pt, linecolor=magenta}
\psline(-4.4,-1.92)(-4.2,-2.08) \psline(-4.4,-1.92)(-4.2,-1.78)
\pscurve(-5.0,-2.4)(-4.9,-2.12)(-4,-1.9)(-3.1,-2.12)(-3.0,-2.4)
\psset{linewidth=1.pt, linecolor=blue}
\psline(-3.1,-1.52)(-3.05,-1.76) \psline(-3.1,-1.52)(-3.35,-1.58)
\pscurve(-2.85,-2.43)(-3,-2.8)(-3.5,-2.7)(-2.85,-1.2)
\psset{linewidth=1.pt, linecolor=magenta}
\psline(1.3,-2.05)(1.39,-2.25) \psline(1.3,-2.05)(1.5,-2.01)
\pscurve(1.7,-2.4)(1.4,-2.1)(1,-1.95)
\psset{linewidth=1.pt, linecolor=blue}
\psline(0.6,-1.97)(0.45,-2.12) \psline(0.6,-1.97)(0.43,-1.85)
\pscurve(0,-2.3)(0.1,-2.1)(1,-1.95)
\psset{linestyle=dashed}
\psset{linewidth=1.pt, linecolor=magenta}
\pscurve(-5.0,-2.4)(-4.9,-2.63)(-4,-2.8)(-3.1,-2.63)(-3.0,-2.4)
\psset{linewidth=1.pt, linecolor=blue}
\pscurve(-2.7,-1.22)(-2.8,-2.1)(-2.85,-2.43)
\psset{linewidth=1.pt, linecolor=blue}
\pscurve(1.7,-2.4)(0.1,-2.6)(0,-2.3)
\psset{linestyle=solid}
\psset{linewidth=1.5pt, linecolor=black}
\pscurve(-3.3,-2.4)(-3.2,-2.5)(-3.1,-2.4)(-3.0,-2.3)(-2.9,-2.4)
(-2.8,-2.5)(-2.7,-2.4)(-2.6,-2.3)(-2.5,-2.4)
\pscurve(-4.8,-2.5)(-4.9,-2.4)(-5.0,-2.3)(-5.1,-2.4)
(-5.2,-2.5)(-5.3,-2.4)(-5.4,-2.3)(-5.5,-2.4)(-5.6,-2.5)(-5.7,-2.4)
\pscurve(1.7,-2.4)(1.8,-2.5)(1.9,-2.4)(2.0,-2.3)(2.1,-2.4)
(2.2,-2.5)(2.3,-2.4)(2.4,-2.3)(2.5,-2.4)
\pscurve(0.2,-2.5)(0.1,-2.4)(0,-2.3)(-0.1,-2.4)
(-0.2,-2.5)(-0.3,-2.4)(-0.4,-2.3)(-0.5,-2.4)(-0.6,-2.5)(-0.7,-2.4)
\end{pspicture}
\end{center}
\caption{The paths and cycles.}
\label{fig:path-and-cycles}
\end{figure}

\underline{$\bullet$ Connection formula for $D_4^{-} \leadsto D_2^{-}$.}~
Let us introduce another WKB solutions 
\begin{equation}\label{eq:WKB-b}
\varphi_{\pm,a}(z,\eta) = \frac{1}{\sqrt{S_{\rm odd}(z,\eta)}}
\exp\left(\pm\int_{a}^{z}
S_{\rm odd}(z,\eta) dz\right),
\end{equation}
where the path of integration is taken along 
the path $\gamma^{(2)}_z$ depicted in Figure \ref{fig:path-and-cycles} (b). 
(The path for the original WKB solutions $\psi_{\pm,a}(z,\eta)$ is taken 
along the path $\gamma^{(1)}_z$ in Figure \ref{fig:path-and-cycles} (b).)
Then, Theorem \ref{thm:Voros-formula} is valid for the Borel sum 
of $\varphi^{(-\delta)}_{\pm,a}$ on the Stokes curve in question 
(see Section \ref{section:connection-formula}):
\begin{eqnarray}
\begin{cases} \label{eq:conn-formula-minus3-pre}
\Phi^{(-\delta), D_4^-}_{+,a} = \Phi^{(-\delta), D_2^-}_{+,a} 
+ i \hspace{+.1em} \Phi^{(-\delta), D_2^-}_{-,a} \\[+.5em] 
\Phi^{(-\delta), D_4^-}_{-,a} = \Phi^{(-\delta), D_2^-}_{-,a}.
\end{cases}
\end{eqnarray}
Here $\Phi_{\pm,a}^{(-\delta), D_4^-}$ etc.\ are the Borel sum of 
$\varphi^{(-\delta)}_{\pm,a}(z,\eta)$ in the Stokes region $D_4^-$ etc.
Since $\psi^{(\theta)}_{\pm,a}$ and $\varphi^{(\theta)}_{\pm,a}$ 
are related as
\begin{equation}
\psi^{(\theta)}_{\pm, a}(z,\eta) 
= e^{\pm V^{(\theta)}_{\gamma_0}(\eta)}
\varphi^{(\theta)}_{\pm, a}(z,\eta),
\end{equation}
we obtain the following formula for the Borel sum of 
the original WKB solutions $\psi^{(-\delta)}_{\pm, a}$ from 
\eqref{eq:conn-formula-minus3-pre}:
\begin{eqnarray}
\begin{cases} \label{eq:conn-formula-minus3}
\Psi^{(-\delta), D_4^-}_{+,a} = \Psi^{(-\delta), D_2^-}_{+,a} 
+ i {\mathcal S}[e^{2V^{(-\delta)}_{\gamma_0}}] 
\hspace{+.1em} \Psi^{(-\delta), D_2^-}_{-,a} \\[+.5em] 
\Psi^{(-\delta), D_4^-}_{-,a} = \Psi^{(-\delta), D_2^-}_{-,a}.
\end{cases}
\end{eqnarray}

Therefore, combining \eqref{eq:conn-formula-minus1}, 
\eqref{eq:conn-formula-minus2} and \eqref{eq:conn-formula-minus3}, 
we have the connection formula \eqref{eq:conn-formula-minus} 
for the second connection problem (B).
\end{proof}

Next, let us rewrite the formulas 
\eqref{eq:conn-formula-plus} and 
\eqref{eq:conn-formula-minus} to the formulas 
for the WKB solutions 
\begin{equation} \label{eq:WKB-p}
\psi_{\pm,p}(z,\eta) = 
\frac{1}{\sqrt{S_{\rm odd}(z,\eta)}}
\exp\left\{ \pm\left(\eta\int_{a}^{z}
\sqrt{Q_{0}(z)}dz+\int_{p}^z
S_{\rm odd}^{\rm reg}(z,\eta)dz\right) \right\}
\end{equation}
of \eqref{eq:Sch} normalized at $p \in P_{\infty}$, 
which is an end-point of a Stokes curve emanating from $a$ 
as depicted in Figure \ref{fig:two-connection-problems}.
The WKB solutions $\psi^{(\theta)}_{\pm, a}$ and 
$\psi^{(\theta)}_{\pm,p}$ are related as 
\begin{equation} \label{eq:change-of-normalization-app}
\psi^{(\theta)}_{\pm, a}(z,\eta) = 
\exp\left(\pm \int_{a}^{p}
S^{\rm reg}_{\rm odd}(z,e^{i \theta}\eta)dz \right)\psi^{(\theta)}_{\pm,p}
= e^{\pm \frac{1}{2} W^{(\theta)}_{\beta_0}(\eta)}
\psi^{(\theta)}_{\pm,p}(z,\eta),
\end{equation}
where $\beta_0$ is the path designated in Figure 
\ref{fig:path-and-cycles} (a). Therefore, thanks to 
\eqref{eq:change-of-normalization-app},
\eqref{eq:conn-formula-plus} and 
\eqref{eq:conn-formula-minus}, we obtain the following equalities:
\begin{eqnarray} 
\label{eq:conn-formula-p1-plus} ~ \\ \nonumber
& {\rm (A)}: & 
\begin{cases} 
\Psi^{(+\delta), D_1^+}_{+,p} = \Psi^{(+\delta), D_2^+}_{+,p} + 
i {\mathcal S}[e^{-W^{(+\delta)}_{\beta_0}}]
\Psi^{(+\delta), D_2^+}_{-,p} \\[+.5em] 
\Psi^{(+\delta), D_1^{+}}_{-,p} = \Psi^{(+\delta), D_2^{+}}_{-,p}.
\end{cases} \\[+.5em]
\label{eq:conn-formula-p1-minus} ~ \\ \nonumber
& {\rm (B)}: &
\begin{cases} 
\Psi^{(-\delta), D_1^-}_{+,p} = \Psi^{(-\delta), D_2^-}_{+,p}  \\[+.3em] 
\hspace{+2.em}
+ i \left(1 + (t + t^{-1}) {\mathcal S}[e^{V^{(-\delta)}_{\gamma_0}} ]
+ {\mathcal S}[e^{2V^{(-\delta)}_{\gamma_0}} ] \right)
{\mathcal S}[e^{-W^{(-\delta)}_{\beta_0}}]  ~
\Psi^{(-\delta), D_2^-}_{-,p} \\[+.7em] 
\Psi^{(-\delta), D_1^{-}}_{-,p} = \Psi^{(-\delta), D_2^{-}}_{-,p}.
\end{cases}
\end{eqnarray}
Taking $\delta>0$ sufficiently small, 
we may assume that, for a fixed $z_1 \in D_{1}^{+} \cap D_1^{-}$ 
(resp., $z_2 \in D_{2}^{+} \cap D_2^{-}$) the path from $p$ to $z$,
which normalizes the WKB solution \eqref{eq:WKB-p} 
when $z$ lies in a neighborhood of $z_1$ (resp., $z_2$), 
never intersects with the Stokes curves in $G_{\theta}$ 
for any $\theta$ satisfying $-\delta \le \theta \le + \delta$. 
Therefore, Theorem \ref{thm:summability} implies that 
$\psi^{(\theta)}_{\pm, p}$ is Borel summable 
for any $\theta$ satisfying $-\delta \le \theta \le + \delta$. 
Then, the same discussion in the proof of 
\cite[Proposition 2.23, Lemma 3.9]{Iwaki14a} 
shows that  
\begin{equation} \label{eq:no-PSP-formula-app1}
\lim_{\delta \rightarrow +0}
\Psi^{(+\delta), D_1^{+}}_{\pm,p} = 
\lim_{\delta \rightarrow +0}
\Psi^{(-\delta), D_1^{-}}_{\pm,p}
\end{equation}
holds as analytic functions of both $z$ (near $z_1$)
and $\eta$ (for a sufficiently large $\eta > 0$). 
Similarly, 
\begin{equation} \label{eq:no-PSP-formula-app2}
\lim_{\delta \rightarrow +0}
\Psi^{(+\delta), D_2^{+}}_{\pm,p} = 
\lim_{\delta \rightarrow +0}
\Psi^{(-\delta), D_2^{-}}_{\pm,p}
\end{equation}
also holds when $z$ lies near $z_2$. 
Therefore, comparing the connection multipliers 
in \eqref{eq:conn-formula-p1-plus} and 
\eqref{eq:conn-formula-p1-minus} after taking the limit 
$\delta \rightarrow +0$, we obtain 
\begin{equation}\label{eq:Stokes-auto-II-beta0}
{\mathcal S}_{-}[e^{{W}_{\beta_0}}] = 
{\mathcal S}_{+}[e^{{W}_{\beta_0}}]
\Bigl( 1 + (t + t^{-1}) {\mathcal S}_{+}[e^{{V}_{\gamma_0}}] 
+ {\mathcal S}_{+}[e^{2{V}_{\gamma_0}}] \Bigr).
\end{equation}
Here we have used the equality
\begin{equation} \label{eq:no-jump-gamma-0}
{\mathcal S}_{-}[e^{{V}_{\gamma_0}}] = 
{\mathcal S}_{+}[e^{{V}_{\gamma_0}}],
\end{equation}
which follows from the fact that $\gamma_0$ does not intersect with 
the Stokes segment $\ell_0$ (see \cite[Proposition 2.23]{Iwaki14a} 
and Remark \ref{rem:on-Stokes-phenomenon}).
The equality \eqref{eq:Stokes-auto-II-beta0} is the desired 
formula \eqref{eq:Stokes-auto-II} for the path $\beta_0$. 
(Note that $\langle \gamma_0, \beta_0 \rangle = -1$.) 
The formula \eqref{eq:Stokes-auto-II} 
for a general path $\beta$ and a general cycle $\gamma$
can be derived from \eqref{eq:Stokes-auto-II-beta0} by the same manner 
as in the proof of \cite[Theorem 3.4 (Appendix A)]{Iwaki14a}.
This completes the proof of Theorem \ref{thm:Stokes-auto-II}.

\begin{rem}
Since the WKB solutions $\psi_{\pm, a}$ defined by \eqref{eq:WKB-a} 
are normalized along a path which intersects with the Stokes segment 
$\ell_0$, we {\em cannot} expect similar equalities as 
\eqref{eq:no-PSP-formula-app1} and \eqref{eq:no-PSP-formula-app2} 
hold for the Borel sums of $\psi_{\pm, a}$. 
\end{rem}

In our previous paper \cite{Iwaki14a}, we found that 
the jump formulas for Stokes segments of type I and type III
obtained in \cite{Delabaere93, Aoki14} can be formulated as 
a {\em (signed) mutation} in cluster algebras. 
We will see that the jump formula for Voros symbols 
for a type II Stokes segment shown in Theorem \ref{thm:Stokes-auto-II}
is recognized as a {\em (signed) mutation} in 
{\em generalized cluster algebras} in the sense of \cite{Chekhov11} 
in Section \ref{section:main-result}. 


\section{Orbifold triangulations and Stokes graphs}

\subsection{Ideal triangulations of orbifolds}
In this subsection we recall some basic notions
and properties of ideal triangulations of orbifolds 
introduced in \cite{Felikson11}.
Here we concentrate on the orbifold points 
of {\em weight 2\/} therein.

\begin{defn}
Let $\bfS$ be any compact connected oriented surface possibly with boundary.
Let $\bfM$ be a finite set of {\em marked points\/} on $\bfS$ that is
nonempty and includes at least one marked point on each boundary component 
of $\bfS$ and possibly some interior points of $\bfS$.
Let $\bfQ$ be a finite nonempty set of interior points of $\bfS$ such that
$\bfM\cap \bfQ=\emptyset$. The triplet  $\bfO=(\bfS,\bfM,\bfQ)$ 
is called a  {\em bordered orbifold with marked points},
or simply, an {\em orbifold}.
An interior point of $\bfM$ is called a {\em puncture}.
A point of $\bfQ$ is called an {\em orbifold point}.
In figures, a marked point and an orbifold point
is shown by $\bullet$ and  $\otimes$, respectively.
See Figure \ref{fig:annulus1} for an example.
\end{defn}

\begin{figure}
\begin{center}
\begin{pspicture}(0,0)(5,5.2)
%
\psset{linewidth=0.5pt}
\pscircle(2.5,2.5){2.5}
\pscircle[fillstyle=solid, fillcolor=lightgray](2.5,2.5){0.5}
\pscircle[fillstyle=solid, fillcolor=black](0,2.5){0.08} 
\pscircle[fillstyle=solid, fillcolor=black](5,2.5){0.08} 
\pscircle[fillstyle=solid, fillcolor=black](2.5,0){0.08} 
\pscircle[fillstyle=solid, fillcolor=black](2.5,5){0.08}
\pscircle[fillstyle=solid, fillcolor=black](2.5,3){0.08}
\pscircle[fillstyle=solid, fillcolor=black](2.1,2.22){0.08}
\pscircle[fillstyle=solid, fillcolor=black](2.9,2.22){0.08}
\pscircle[fillstyle=solid, fillcolor=black](1.2,2){0.08}
\pscircle[fillstyle=solid, fillcolor=black](3.6,4){0.08}
\rput[c]{0}(0.9,2.8){$\otimes$}
\rput[c]{0}(1.5,2.8){$\otimes$}
\rput[c]{0}(3.6,2.5){$\otimes$}
\psline(1.2,2)(1.47,2.7)
\psline(1.2,2)(0.93,2.7)
\psline(2.9,2.22)(3.5,2.46)
\pscurve(1.2,2)(0.6,2.9)(1.2,3.4)(1.8,2.9)(1.2,2)
\pscurve(1.2,2)(0.6,2.2)(0,2.5)
\pscurve(1.2,2)(1.67,1)(2.5,0)
\pscurve(1.2,2)(1.6,2.1)(2.1,2.22)
\pscurve(2.5,0)(2.2,1.5)(2.1,2.22)
\pscurve(2.5,0)(2.8,1.5)(2.9,2.22)
\psline(2.5,5)(3.6,4)
\pscurve(2.5,5)(3.8,4.3)(4,3.65)(3.4,3.6)(2.5,5)
\pscurve(2.9,2.22)(3.8,2.05)(5,2.5)
\pscurve(2.9,2.22)(3.7,2.8)(5,2.5)
\pscurve(2.5,5)(3.2,3.6)(4.5,2.9)(5,2.5)
\pscurve(2.5,3)(3.4,3.1)(4.4,2.85)(5,2.5)
\psline(2.5,3)(2.5,5)
\pscurve(1.2,2)(0.5,2.55)(0.8,3.6)(1.75,3.35)(1.95,2.6)(2,2.3)(2.1,2.22)
\pscurve(1.2,2)(0.45,2.4)(0.8,3.8)(1.8,3.6)(2.5,3)
\pscurve(0,2.5)(0.95,4.1)(1.75,3.8)(2.5,3)
\end{pspicture}
\end{center}
\caption{Example of an ideal triangulation of
an annulus with two punctures and three orbifold points.} 
\label{fig:annulus1}
\end{figure}
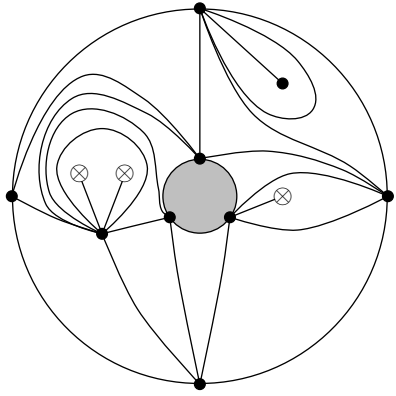

\begin{defn}
\label{defn:arc1}
 An {\em arc\/} $\alpha$ in an orbifold $\bfO$ is a curve
in $\bfS$ such that
\begin{itemize}
\item one of the following holds:
\begin{itemize}
\item
both endpoints of $\alpha$ are marked points  (then $\alpha$ is called an
{\em ordinary arc\/}),
\item
one end point of $\alpha$ is a marked point and the other end point is
an orbifold point (then $\alpha$ is called a
{\em pending arc\/}),
\end{itemize}
\item $\alpha$ does not intersect itself except for the endpoints,
\item $\alpha$ is away from punctures, orbifold points, and the boundary except 
for the endpoints,
\item $\alpha$ is not contractible into a marked point or onto a boundary of $\bfS$,
\item $\alpha$ does not cut out a monogon containing exactly one 
orbifold point.
\end{itemize}
Furthermore, each arc is considered 
{\em up to isotopy\/}  in the class of such curves.
\end{defn}

\begin{defn}
\label{defn:compat1}
Two arcs are said to be {\em compatible\/} if
\begin{itemize}
\item
 there are representatives in their respective isotopy classes such that
 they do not intersect each other in the interior of $\bfS$,
 \item
they do not share an orbifold point as a common end point.
 \end{itemize}
\end{defn}

\begin{defn}
A {\em labeled ideal triangulation\/} $T$ 
of an orbifold $\bf O$ is an $n$-tuple of arcs $(\alpha_i)_{i=1}^{n}$ 
in ${\bf O}$ which gives a maximal set of
distinct pairwise compatible arcs.
\end{defn}

\begin{caut}
This should not be confused with a {\em tagged triangulation\/} 
of an orbifold also studied in \cite{Felikson11}.
We do not use tagged triangulations in this paper. 
\end{caut}

Figure \ref{fig:annulus1} shows an example of a labeled ideal triangulation
of an annulus with two punctures and three orbifold points, where the labels
of arcs are omitted for simplicity.

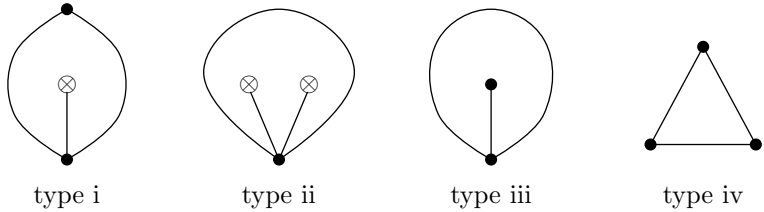
\begin{figure}
\begin{center}
\begin{pspicture}(0,-0.5)(2,2.2)
%
\psset{linewidth=0.5pt}
\pscircle[fillstyle=solid, fillcolor=black](1,0){0.08} 
\pscircle[fillstyle=solid, fillcolor=black](1,2){0.08} 
\rput[c]{0}(1,1){$\otimes$}
\rput[c]{0}(1,-0.5){type i}
\psline(1,0.9)(1,0)
\pscurve(1,0)(0.3,0.6)(0.3,1.4)(1,2)
\pscurve(1,0)(1.7,0.6)(1.7,1.4)(1,2)
\end{pspicture}
\hskip20pt
\begin{pspicture}(0,-0.5)(2,2.2)
%
\psset{linewidth=0.5pt}
\pscircle[fillstyle=solid, fillcolor=black](1,0){0.08} 
\rput[c]{0}(0.6,1){$\otimes$}
\rput[c]{0}(1.4,1){$\otimes$}
\rput[c]{0}(1,-0.5){type ii}

\psline(0.62,0.9)(1,0)
\psline(1.38,0.9)(1,0)
\pscurve(1,0)(0,1.2)(1,2)(2,1.2)(1,0)
\end{pspicture}
\hskip20pt
\begin{pspicture}(0,-0.5)(2,2.2)
%
\psset{linewidth=0.5pt}
\pscircle[fillstyle=solid, fillcolor=black](1,0){0.08} 
\pscircle[fillstyle=solid, fillcolor=black](1,1){0.08} 
\rput[c]{0}(1,-0.5){type iii}

\psline(1,1)(1,0)
\pscurve(1,0)(0.3,0.6)(0.3,1.6)(1,2)(1.7,1.6)(1.7,0.6)(1,0)
\end{pspicture}
\hskip20pt
\begin{pspicture}(0,-0.5)(2,2.2)
%
\psset{linewidth=0.5pt}
\pscircle[fillstyle=solid, fillcolor=black](0.3,0.2){0.08} 
\pscircle[fillstyle=solid, fillcolor=black](1.7,0.2){0.08} 
\pscircle[fillstyle=solid, fillcolor=black](1,1.5){0.08} 
\rput[c]{0}(1,-0.5){type iv}

\psline(0.3,0.2)(1.7,0.2)
\psline(0.3,0.2)(1,1.5)
\psline(1.7,0.2)(1,1.5)
\end{pspicture}
\end{center}
\caption{Pieces of labeled ideal triangulations of orbifolds.} 
\label{fig:piece1}
\end{figure}

It  was explained in \cite{Felikson11} that any labeled 
ideal triangulation of an orbifold is built by glueing the edges 
of four kinds of ``puzzle pieces'' depicted in Figure \ref{fig:piece1}.
We call them
{\em type i, type ii, type iii, and type iv pieces}, respectively,
as specified in Figure \ref{fig:piece1}. We remark that, there are 
actually infinitely many  distinct type ii pieces of given 
two orbifold points and  a puncture. See Figure \ref{fig:pending2}.
Note that orbifold points and pending arcs are always inside 
type i or type ii pieces.

\begin{defn} For a given labeled ideal triangulation 
$T=(\alpha_i)_{i=1}^n$ of an orbifold $\bfO$, 
a {\em flip of $T$ at $k$} ($k=1,\dots,n$) is another labeled
ideal triangulation $T'=(\alpha'_i)_{i=1}^n$ of $\bfO$ 
such that $\alpha'_i=\alpha_i$ for $i\neq k$
and $\alpha'_k \neq \alpha_k$. 
\end{defn}

\begin{lem}[{\cite{Fomin08,Felikson11}}]
\label{lem:flip1}
For a given triangulation $T=(\alpha_i)_{i=1}^n$ of an 
orbifold $\bfO$, the flip $T'=\mu_k(T)$ exists if and only if
the arc $\alpha_k$ is not inside a type iii piece of $T$;
moreover, under the same condition $T'$ exists uniquely.
In particular, any pending arcs can be  flipped.
\end{lem}
Note that $T'=\mu_k(T)$ if and only if $T=\mu_k(T')$. Namely,
$\mu_k$ is an involution.

\begin{ex} Flips of pending arcs are shown in Figures \ref{fig:pending1}
and \ref{fig:pending2}.
\end{ex}

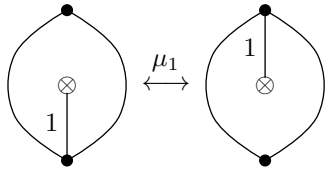
\begin{figure}
\begin{center}
\begin{pspicture}(0,0)(2,2.2)
%
\psset{linewidth=0.5pt}
\pscircle[fillstyle=solid, fillcolor=black](1,0){0.08} 
\pscircle[fillstyle=solid, fillcolor=black](1,2){0.08} 
\rput[c]{0}(1,1){$\otimes$}
\rput[c]{0}(0.8,0.5){$1$}
\psline(1,0.9)(1,0)
\pscurve(1,0)(0.3,0.6)(0.3,1.4)(1,2)
\pscurve(1,0)(1.7,0.6)(1.7,1.4)(1,2)
\end{pspicture}
\begin{pspicture}(-0.2,0)(0.2,2.2)
\rput[c]{0}(0,1.3){$\mu_1$}
\rput[c]{0}(0,1){$\longleftrightarrow$}
\end{pspicture}
\begin{pspicture}(0,0)(2,2.2)
%
\psset{linewidth=0.5pt}
\pscircle[fillstyle=solid, fillcolor=black](1,0){0.08} 
\pscircle[fillstyle=solid, fillcolor=black](1,2){0.08} 
\rput[c]{0}(1,1){$\otimes$}
\rput[c]{0}(0.8,1.5){$1$}
\psline(1,1.1)(1,2)
\pscurve(1,0)(0.3,0.6)(0.3,1.4)(1,2)
\pscurve(1,0)(1.7,0.6)(1.7,1.4)(1,2)
\end{pspicture}
\end{center}
\caption{Flips of pending arcs in type i pieces.} 
\label{fig:pending1}
\end{figure}
\begin{figure}
\begin{center}
%
%
\begin{pspicture}(0,0)(2,2.2)
%
\psset{linewidth=0.5pt}
\pscircle[fillstyle=solid, fillcolor=black](1,0){0.08} 
\rput[c]{0}(0.6,1){$\otimes$}
\rput[c]{0}(1.4,1){$\otimes$}
\rput[c]{0}(0.85,0.8){1}
\rput[c]{0}(1,1.2){2}
\pscurve(1.34,1.07)(1,1.4)(0.4,1.1)(1,0)
\psline(0.62,0.9)(1,0)
\pscurve(1,0)(0,1.2)(1,2)(2,1.2)(1,0)
\end{pspicture}
\begin{pspicture}(-0.3,0)(0.3,2.2)
\rput[c]{0}(0,1.3){$\mu_2$}
\rput[c]{0}(0,1){$\longleftrightarrow$}
\end{pspicture}
%
\begin{pspicture}(0,0)(2,2.2)
%
\psset{linewidth=0.5pt}
\pscircle[fillstyle=solid, fillcolor=black](1,0){0.08} 
\rput[c]{0}(0.6,1){$\otimes$}
\rput[c]{0}(1.4,1){$\otimes$}
\rput[c]{0}(0.85,0.8){1}
\rput[c]{0}(1.15,0.8){2}
\psline(0.62,0.9)(1,0)
\psline(1.38,0.9)(1,0)
\pscurve(1,0)(0,1.2)(1,2)(2,1.2)(1,0)
\end{pspicture}
\begin{pspicture}(-0.3,0)(0.3,2.2)
\rput[c]{0}(0,1.3){$\mu_1$}
\rput[c]{0}(0,1){$\longleftrightarrow$}
\end{pspicture}
%
\begin{pspicture}(0,0)(2,2.2)
%
\psset{linewidth=0.5pt}
\pscircle[fillstyle=solid, fillcolor=black](1,0){0.08} 
\rput[c]{0}(0.6,1){$\otimes$}
\rput[c]{0}(1.4,1){$\otimes$}
\rput[c]{0}(1.15,0.8){2}
\rput[c]{0}(1,1.2){1}
\pscurve(0.66,1.07)(1,1.4)(1.6,1.1)(1,0)
\psline(1.38,0.9)(1,0)
\pscurve(1,0)(0,1.2)(1,2)(2,1.2)(1,0)
\end{pspicture}
\begin{pspicture}(-0.3,0)(0.3,2.2)
\rput[c]{0}(0,1.3){$\mu_2$}
\rput[c]{0}(0,1){$\longleftrightarrow$}
\end{pspicture}
%
\begin{pspicture}(0,0)(2,2.2)
%
\psset{linewidth=0.5pt}
\pscircle[fillstyle=solid, fillcolor=black](1,0){0.08} 
\rput[c]{0}(0.6,1){$\otimes$}
\rput[c]{0}(1.4,1){$\otimes$}
\rput[c]{0}(1.2,0.7){2}
\rput[c]{0}(1,1.2){1}
\pscurve(0.66,1.07)(1,1.4)(1.6,1.1)(1,0)
\pscurve(1.31,0.95)(0.5,0.7)(0.5,1.3)(1.75,1.3)(1,0)
\pscurve(1,0)(0,1.2)(1,2)(2,1.2)(1,0)
\end{pspicture}
%
\end{center}
\caption{Flips of pending arcs in type ii pieces.} 
\label{fig:pending2}
\end{figure}
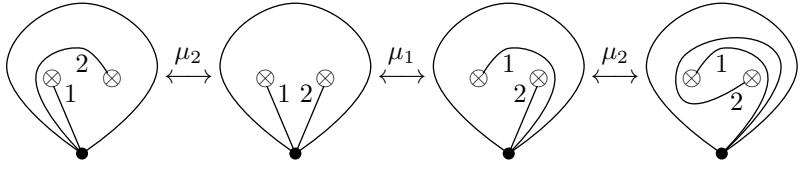

The following property is a  consequence of the transitivity  of flips 
on (unlabeled) ideal triangulations of an orbifold 
given by \cite[Theorem 4.2]{Felikson11}.

\begin{prop}[{\cite[Theorem 4.2]{Felikson11}}]
\label{prop:total1}
Let $\bfO$ be an orbifold, and let
$T=(\alpha_i)_{i=1}^n$ be a labeled ideal triangulation of $\bfO$.
Then, the total number ($=n$) of arcs of  $T$
 and the total number of pieces in $T$ do not depend on $T$.
\end{prop}

\subsection{Signed adjacency matrix}
To each labeled ideal triangulation $T=(\alpha_i)_{i=1}^n$ 
of an orbifold, we attach a skew-symmetrizable integer matrix
$B=B(T)=(b_{ij})_{i,j=1}^n$ called the {\em signed adjacency matrix of $T$} 
(see \cite[Section 4]{Fomin08} and \cite[Table 3.4]{Felikson11}). 
In this paper we do not need its full definition, which is 
somewhat cumbersome to describe, and only the submatrices 
of $B$ corresponding to type i and type ii pieces of $T$ are essential.
They are explicitly given as follows:
 \par
{\em Type i piece:}
\begin{align}
\label{eq:b1}
\raisebox{-30pt}{
\mbox{
\begin{pspicture}(0,-0.2)(2,2.2)
%
\psset{linewidth=0.5pt}
\pscircle[fillstyle=solid, fillcolor=black](1,0){0.08} 
\pscircle[fillstyle=solid, fillcolor=black](1,2){0.08} 
\rput[c]{0}(1,1){$\otimes$}
\rput[c]{0}(0.8,0.5){1}
\rput[c]{0}(-0,1){2}
\rput[c]{0}(2,1){3}
\psline(1,0.9)(1,0)
\pscurve(1,0)(0.3,0.6)(0.3,1.4)(1,2)
\pscurve(1,0)(1.7,0.6)(1.7,1.4)(1,2)
\end{pspicture}
}
}
\hskip30pt
B=
\left(
\begin{matrix}
0&2&-2\\
-1&0&1\\
1&-1&0\\
\end{matrix}
\right)
\end{align}

\par
{\em Type ii piece:} The matrix $B$ depends only on the relative 
position of arcs near the marked point, and it is independent of
how they reach to orbifold points.
\begin{align}
\label{eq:b2}
\raisebox{-30pt}{
\mbox{
\begin{pspicture}(0,-0.2)(2,2.2)
%
\psset{linewidth=0.5pt}
\pscircle[fillstyle=solid, fillcolor=black](1,0){0.08} 
\rput[c]{0}(0.6,1){$\otimes$}
\rput[c]{0}(1.4,1){$\otimes$}
\rput[c]{0}(0.85,0.8){1}
\rput[c]{0}(1.15,0.8){2}
\rput[c]{0}(-0.2,1){3}
\psline(0.62,0.9)(1,0)
\psline(1.38,0.9)(1,0)
\pscurve(1,0)(0,1.2)(1,2)(2,1.2)(1,0)
\end{pspicture}
}
}
\hskip30pt
B=
\left(
\begin{matrix}
0&-2&2\\
2&0&-2\\
-1&1&0\\
\end{matrix}
\right)
\end{align}

We set the {\em weight $d_i$} of the arc $\alpha_i$ in $T$ as
\begin{align}
\label{eq:weight1}
d_i=
\begin{cases}
1 & \mbox{$\alpha_i$ is an ordinary arc}\\
2 & \mbox{$\alpha_i$ is a pending arc.}
\end{cases}
\end{align}
Introduce the diagonal matrix $D=(d_i\delta_{ij})_{i,j=1}^n$.
Then, $BD$ and $\tilde{B}=D^{-1} B$ are skew-symmetric and integer matrices.
For example, in the above two cases, we have respectively
\begin{align}
\label{eq:bd1}
BD=\left(
\begin{matrix}
0&2&-2\\
-2&0&1\\
2&-1&0\\
\end{matrix}
\right),
\quad
\tilde{B}=D^{-1}B=\left(
\begin{matrix}
0&1&-1\\
-1&0&1\\
1&-1&0\\
\end{matrix}
\right),\\
\label{eq:bd2}
BD=\left(
\begin{matrix}
0&-4&2\\
4&0&-2\\
-2&2&0\\
\end{matrix}
\right),
\quad
\tilde{B}=D^{-1}B=\left(
\begin{matrix}
0&-1&1\\
1&0&-1\\
-1&1&0\\
\end{matrix}
\right).
\end{align}

The following is  the most basic notion in cluster algebra theory \cite{Fomin03a}.
\begin{defn}
For a skew-symmetrizable matrix $B=(b_{ij})_{i,j=1}^n$,
the {\em mutation of $B$ at $k$}
($k=1,\dots,n$) is another skew-symmetrizable matrix 
$B'=\mu_k(B)=(b'_{ij})_{i,j=1}^n$
defined by the following relation:
 \begin{align}
\label{eq:bmut}
b'_{ij}&=
\begin{cases}
-b_{ij}& \mbox{$i=k$ or $j=k$}\\
b_{ij}+[- b_{ik}]_+ b_{kj} + b_{ik}[ b_{kj}]_+
& \mbox{$i,j\neq k$,}\\
\end{cases}
\end{align}
where $[a]_+=\max(a,0)$.
\end{defn}

The flips of labeled ideal triangulations of an orbifold and 
the mutations of their signed adjacency matrices are compatible 
in the following sense.

\begin{prop}[{\cite[Theorem 4.19]{Felikson11}}]
\label{prop:bT1}
For any labeled ideal triangulation $T=(\alpha_i)_{i=1}^n$ of an 
orbifold $\bfO$ and any $k$ such that the arc $\alpha_k$ is not 
inside a type iii piece of $T$, we have $B(\mu_k(T))=\mu_k(B(T))$.
\end{prop}

\begin{ex} For the matrix $B$ in \eqref{eq:b1}, $\mu_1(B)=-B$.
For the matrix $B$ in \eqref{eq:b2}, $\mu_1(B)=\mu_2(B)=-B$.
They certainly agree with Proposition \ref{prop:bT1}.
\end{ex}

\subsection{Stokes triangulations of orbifolds 
and signed flips of pending arcs}

In \cite{Iwaki14a} the notion of a {\em Stokes triangulation\/} 
of a bordered surface was introduced. It is a refinement of
 an ideal triangulation  by \cite{Fomin08}
so as to capture the characteristics of the geometry of Stokes graphs
especially under the mutation.
See also \cite{Qiu14}.
Here we  extend it to the orbifold case straightforwardly.

\begin{defn} For an orbifold $\bfO=(\bfS,\bfM,\bfQ)$,
let $m$ be the total number of pieces in any labeled ideal triangulation
$T$ of $\bfO$, where $m$ is independent of the choice of $T$
by Proposition \ref{prop:total1}.
Accordingly, we introduce a set $\bfA$ consisting of $m$ 
internal points of $\bfO$ such that $\bfA\cap \bfM=\bfA \cap \bfQ=\emptyset$.
We call $a\in \bfA$ a {\em midpoint (of a piece)\/},
and in figures it will be shown by $\times$. 
For brevity, we still call
$(\bfO,\bfA)=(\bfS,\bfM,\bfQ,\bfA)$ an {\em orbifold}.
\end{defn}

\begin{defn} An {\em arc $\alpha$ in an orbifold $(\bfO,\bfA)$} is a curve
in $\bfS\setminus \bfA$ satisfying all conditions in
Definition \ref{defn:arc1}.
Each arc $\alpha$ is considered up to isotopy in the class of such curves. 
\end{defn}

When we consider an arc $\alpha$ in $(\bfO,\bfA)$,
it is sometimes convenient to regard it as an arc in $\bfO$
by forgetting the midpoints.
In that case we write the latter arc as $\tilde{\alpha}$ to
avoid confusion.

\begin{defn}
An $n$-tuple $T=(\alpha_i)_{i=1}^n$ of arcs
in $(\bfO,\bfA)$
is called a {\em labeled Stokes triangulation of $(\bfO,\bfA)$\/} 
if the following conditions are satisfied:
\begin{itemize}
\item
The arcs $\alpha_1$, \dots, $\alpha_n$
in $(\bfO,\bfA)$
 are pairwise compatible
(in the sense of Definition
\ref{defn:compat1} but considered in the isotopy classes for
arcs in $(\bfO,\bfA)$).
\item
The $n$-tuple $\tilde{T}=(\tilde{\alpha}_i)_{i=1}^n$
of arcs in  $\bfO$ yields
a labeled ideal triangulation of $\bfO$.
\item
Every piece of $T$ contains exactly one midpoint in its interior.
Here, a piece of $T$ is
a collection of arcs of $T$ such that it yields a piece of 
$\tilde{T}$ as arcs in $\bfO$.
\end{itemize}
\end{defn}

For a labeled Stokes triangulation $T$ of $(\bfO,\bfA)$,
its signed adjacency matrix $B(T)$ and 
weights of arcs in $T$ are defined by those 
of the underlying labeled ideal triangulation $\tilde{T}$.

Again,  pieces of a labeled Stokes triangulation $T$
fall into  four types (from type i to type iv) depending
on their types as pieces  in $\tilde{T}$.

Next we  introduce the {\em signed flips\/}
for labeled Stokes triangulations.
Like the ordinary flip of an arc of a labeled ideal triangulation, 
they are defined if an arc  is not  
 inside a type iii piece of a triangulation.
Here we concentrate on the signed flips of pending arcs.
 See \cite[Definition 6.4]{Iwaki14a} for the signed flips of
 ordinary arcs.

\begin{figure}
\begin{center}
\begin{pspicture}(0,0)(11,2.2)
\psset{linewidth=0.5pt}
\psset{fillstyle=solid, fillcolor=black}
\pscircle(1,0){0.08} 
\pscircle(1,2){0.08} 
\psset{fillstyle=none}
\psline(1,0)(1,0.9)
\pscurve(1,0)(0.3,0.6)(0.3,1.4)(1,2)
\pscurve(1,0)(1.7,0.6)(1.7,1.4)(1,2)
\rput[c]{0}(1,1.4){$\times$}
\rput[c]{0}(1,1){$\otimes$}
\rput[c]{0}(0.8,0.5){1}
\psline[arrows=->](2.2,1.2)(2.8,1.2)
\psline[arrows=<-](2.2,0.8)(2.8,0.8)
\rput[c]{0}(2.5,1.6){$\mu_{1}^{(+)}$}
\rput[c]{0}(2.5,0.5){$\mu_{1}^{(-)}$}
%
%
\rput[c]{0}(3,0)
{
\psset{linewidth=0.5pt}
\psset{fillstyle=solid, fillcolor=black}
\pscircle(1,0){0.08} 
\pscircle(1,2){0.08} 
\psset{fillstyle=none}
\pscurve(1,0)(0.3,0.6)(0.3,1.4)(1,2)
\pscurve(1,0)(1.7,0.6)(1.7,1.4)(1,2)
\pscurve(0.93,1.07)(0.7,1.5)(1,2)
%
\rput[c]{0}(1,1.4){$\times$}
\rput[c]{0}(1,1){$\otimes$}
\rput[c]{0}(0.6,1.5){1}
\psline[arrows=->](2.2,1.2)(2.8,1.2)
\psline[arrows=<-](2.2,0.8)(2.8,0.8)
\rput[c]{0}(2.5,1.6){$\mu_{1}^{(+)}$}
\rput[c]{0}(2.5,0.5){$\mu_{1}^{(-)}$}
}
%
%
\rput[c]{0}(6,0)
{
\psset{linewidth=0.5pt}
\psset{fillstyle=solid, fillcolor=black}
\pscircle(1,0){0.08} 
\pscircle(1,2){0.08} 
\psset{fillstyle=none}
\pscurve(1,0)(0.3,0.6)(0.3,1.4)(1,2)
\pscurve(1,0)(1.7,0.6)(1.7,1.4)(1,2)
\pscurve(1,0)(1.3,1.35)(1,1.7)(0.7,1.35)(0.93,1.07)
%
\rput[c]{0}(1,1.4){$\times$}
\rput[c]{0}(1,1){$\otimes$}
\rput[c]{0}(1,0.5){1}
\psline[arrows=->](2.2,1.2)(2.8,1.2)
\psline[arrows=<-](2.2,0.8)(2.8,0.8)
\rput[c]{0}(2.5,1.6){$\mu_{1}^{(+)}$}
\rput[c]{0}(2.5,0.5){$\mu_{1}^{(-)}$}
}
%
%
\rput[c]{0}(9,0)
{
\psset{linewidth=0.5pt}
\psset{fillstyle=solid, fillcolor=black}
\pscircle(1,0){0.08} 
\pscircle(1,2){0.08} 
\psset{fillstyle=none}
\pscurve(1,0)(0.3,0.6)(0.3,1.4)(1,2)
\pscurve(1,0)(1.7,0.6)(1.7,1.4)(1,2)
\pscurve(1,2)(0.55,1.5)(0.55,0.8)(1,0.5)(1.3,1.35)(1,1.7)(0.7,1.35)(0.93,1.07)
\rput[c]{0}(1,1.4){$\times$}
\rput[c]{0}(1,1){$\otimes$}
\rput[c]{0}(1,0.7){1}
%
%
}
\end{pspicture}
\end{center}
\caption{Signed flips of pending arcs in type i pieces.}
\label{fig:flip1}
\end{figure}
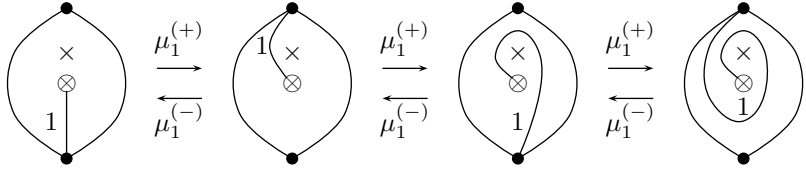

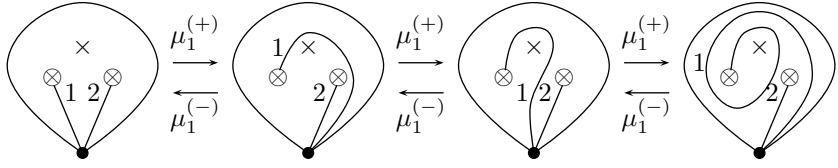
\begin{figure}
\begin{center}
\begin{pspicture}(0,0)(11,2.2)
%
\psset{linewidth=0.5pt}
\pscircle[fillstyle=solid, fillcolor=black](1,0){0.08} 
\rput[c]{0}(0.6,1){$\otimes$}
\rput[c]{0}(1.4,1){$\otimes$}
\rput[c]{0}(1,1.4){$\times$}
\rput[c]{0}(0.85,0.8){1}
\rput[c]{0}(1.15,0.8){2}
\psline(0.62,0.9)(1,0)
\psline(1.38,0.9)(1,0)
\pscurve(1,0)(0,1.2)(1,2)(2,1.2)(1,0)
\psline[arrows=->](2.2,1.2)(2.8,1.2)
\psline[arrows=<-](2.2,0.8)(2.8,0.8)
\rput[c]{0}(2.5,1.6){$\mu_{1}^{(+)}$}
\rput[c]{0}(2.5,0.5){$\mu_{1}^{(-)}$}
%
%
\rput[c]{0}(3,0)
{
%
\psset{linewidth=0.5pt}
\pscircle[fillstyle=solid, fillcolor=black](1,0){0.08} 
\rput[c]{0}(0.6,1){$\otimes$}
\rput[c]{0}(1.4,1){$\otimes$}
\rput[c]{0}(1,1.4){$\times$}
\rput[c]{0}(0.6,1.4){1}
\rput[c]{0}(1.15,0.8){2}
%
\pscurve(0.62,1.1)(1,1.6)(1.6,1)(1,0)
\psline(1.38,0.9)(1,0)
\pscurve(1,0)(0,1.2)(1,2)(2,1.2)(1,0)
\psline[arrows=->](2.2,1.2)(2.8,1.2)
\psline[arrows=<-](2.2,0.8)(2.8,0.8)
\rput[c]{0}(2.5,1.6){$\mu_{1}^{(+)}$}
\rput[c]{0}(2.5,0.5){$\mu_{1}^{(-)}$}
}
%
%
\rput[c]{0}(6,0)
{
%
\psset{linewidth=0.5pt}
\pscircle[fillstyle=solid, fillcolor=black](1,0){0.08} 
\rput[c]{0}(0.6,1){$\otimes$}
\rput[c]{0}(1.4,1){$\otimes$}
\rput[c]{0}(1,1.4){$\times$}
\rput[c]{0}(0.85,0.8){1}
\rput[c]{0}(1.15,0.8){2}
\pscurve(0.62,1.1)(0.8,1.6)(1.2,1.6)(0.95,0.6)(1,0)
\psline(1.38,0.9)(1,0)
\pscurve(1,0)(0,1.2)(1,2)(2,1.2)(1,0)
\psline[arrows=->](2.2,1.2)(2.8,1.2)
\psline[arrows=<-](2.2,0.8)(2.8,0.8)
\rput[c]{0}(2.5,1.6){$\mu_{1}^{(+)}$}
\rput[c]{0}(2.5,0.5){$\mu_{1}^{(-)}$}}
%
%
\rput[c]{0}(9,0)
{
%
\psset{linewidth=0.5pt}
\pscircle[fillstyle=solid, fillcolor=black](1,0){0.08} 
\rput[c]{0}(0.6,1){$\otimes$}
\rput[c]{0}(1.4,1){$\otimes$}
\rput[c]{0}(1,1.4){$\times$}
\rput[c]{0}(0.2,1.2){1}
\rput[c]{0}(1.17,0.8){2}
\pscurve(0.62,1.1)(0.8,1.6)(1.2,1.6)(0.75,0.6)(0.3,1)(0.7,1.8)(1.3,1.8)(1.7,1)(1,0)
\psline(1.38,0.9)(1,0)
\pscurve(1,0)(0,1.2)(1,2)(2,1.2)(1,0)
}
\end{pspicture}
\end{center}
\caption{Signed flips of pending arcs in type ii pieces.}
\label{fig:flip2}
\end{figure}

\begin{defn} For a labeled Stokes triangulation $T=(\alpha_i)_{i=1}^n$
of an orbifold $(\bfO,\bfA)$,  a pending arc $\alpha_k$,
and a sign $\ve\in \{+,-\}$,
the {\em signed flip  $T'=\mu_k^{(\ve)}(T)$ at $k$
with sign $\ve$} is another 
labeled Stokes triangulation of $(\bfO,\bfA)$ obtained from $T$ by replacing
the arc $\alpha_k$ with the one in Figures \ref{fig:flip1} and \ref{fig:flip2}.
Namely,
\begin{itemize}
\item %
if the pending arc $\alpha_k$ is in a {\em type i\/} piece,
$\alpha'_k$ is obtained from $\alpha_k$ by sliding
 its outside end point from a marked point to the other marked point 
 along the boundary of the piece clockwise 
 for $\ve=+$ and anticlockwise for $\ve=-$,
\item %
if the pending arc $\alpha_k$ is in a {\em type ii\/} piece,
$\alpha'_k$ is obtained from $\alpha_k$ by sliding
 its outside end point from a marked point to itself along the 
 boundary {\em or\/} along the other pending arc 
 in the  piece  clockwise for $\ve=+$ and anticlockwise for $\ve=-$.
 \end{itemize}
\end{defn}

The signed flips are compatible with the flip of the underlying
labeled ideal triangulation; i.e.,
for $T'=\mu^{(\ve)}_k(T)$ with any $\ve=\pm$, $\tilde{T}=\mu_k(\tilde{T})$ holds.
In particular, they are also compatible with the mutation of their
signed adjacency matrices, namely, $B(T')=\mu_k(B(T))$.

\subsection{Stokes graphs and Stokes triangulations}
\label{subsection:SG-to-ST}

Here we explain how a  Stokes triangulation of an orbifold
is obtained from a saddle-free Stokes graph of the Schr\"odinger 
equation \eqref{eq:Sch} with simple poles.

Recall that the principal term of the potential $Q(z,\eta)$ determines
a quadratic differential $\phi$ on the Riemann surface $\Sigma$ 
(see Definition \ref{def:quad-diff}), and we have assumed that 
all zeros of $\phi$ are simple 
(see Assumption \ref{ass:zeros-and-poles}), 
and there is no recurrent trajectory 
(see Assumption \ref{ass:trajectory}). 
In this subsection we also assume that 
\begin{itemize} 
\item[(A)]
the Stokes graph $G(\phi)$ in Definition \ref{def:Stokes-curve}
is saddle-free.
\end{itemize}
Furthermore, let us temporarily assume that 
\begin{itemize}
\item[(B)]
$\phi$ has no simple poles.
 \end{itemize}
Under such assumptions,
any Stokes region of $G(\phi)$ falls into one of  three types,
i.e.,
{\em regular horizontal strip, degenerate horizontal strip, or half plane},
depicted in Figure \ref{fig:Stokes1}  (a)--(c) 
(see \cite[Section 3.5]{Bridgeland13}).
\begin{figure}
\begin{center}
\begin{pspicture}(-2.2,-1.8)(9,1.5)
\psset{linewidth=0.5pt}
\psset{fillstyle=solid, fillcolor=black}
\pscircle(-1.5,0){0.08} 
\pscircle(1.5,0){0.08} 
\psset{fillstyle=none}
\psline(0,1.2)(0,0.6)
\psline(0,-1.2)(0,-0.6)
\psline[linestyle=dotted,linewidth=0.8pt](1.5,0)(-1.5,0)
%
\pscurve(1.5,0)(1.4,0)(0.6,0.25)(0,0.6)
\pscurve(-1.5,0)(-1.4,0)(-0.6,0.25)(0,0.6)
\pscurve(1.5,0)(1.4,0)(0.6,-0.25)(0,-0.6)
\pscurve(-1.5,0)(-1.4,0)(-0.6,-0.25)(0,-0.6)
\rput[c]{0}(0,-0.6){$\times$}
\rput[c]{0}(0,0.6){$\times$}
\rput[c]{0}(0,-1.6){(a)}
\rput[c]{0}(-2,0){
\psset{fillstyle=solid, fillcolor=black}
\pscircle(5,-1){0.08} 
\pscircle(5,0.6){0.08} 
\psset{fillstyle=none}
\pscurve(5,-1)(5.3,-0.6)(5.5,-0.2)(5.6,0.8)(5,1.2)
(4.4,0.8)(4.4,0.2)(4.6,-0.2)
\pscurve(5,-1)(4.8,-0.8)(4.68,-0.6)(4.6,-0.2)
\pscurve(4.6,-0.2)(5,0)(5.25,0.6)(5,0.78)(4.85,0.6)
(5,0.48)(5.1,0.6)(5,0.67)
\pscurve[linestyle=dotted,linewidth=0.8pt](5,-1)(5.4,0)
(5.4,0.6)(5,0.85)(4.78,0.6)(5,0.42)%
(5.15,0.6)(5,0.73)(4.9,0.6)(5,0.52)
\rput[c]{0}(4.6,-0.2){$\times$}
\rput[c]{0}(5,-1.6){(b)}
}
\rput[c]{0}(-4.5,0){
\psset{fillstyle=solid, fillcolor=black}
\pscircle(10,-1){0.08} 
\psset{fillstyle=none}
\psline(10,1)(10,0.5)
\pscurve(10,-1)(10.5,-0.5)(10.5,0)(10,0.5)
\pscurve(10,-1)(9.5,-0.5)(9.5,0)(10,0.5)
\pscurve[linestyle=dotted,linewidth=0.8pt](10,-1)(10.2,-0.78)
(10.3,-0.25)(10,-0.05)(9.7,-0.25)(9.8,-0.78)(10,-1)
\rput[c]{0}(10,0.5){$\times$}
\rput[c]{0}(10,-1.6){(c)}
}
%
%
%
\rput[c]{0}(7,-1)
{
\psset{linewidth=0.5pt}
\pscircle[fillstyle=solid, fillcolor=black](1,0){0.08} 
%
\rput[c]{0}(1,1.5){$\times$}
\rput[c]{0}(1,1){$\otimes$}
\rput[c]{0}(1,-0.6){(d)}
\psline(1,0.9)(1,0)
\psline(1,1.5)(1,2)
\pscurve(1,0)(0.5,0.5)(0.5,1)(1,1.5)
\pscurve(1,0)(1.5,0.5)(1.5,1)(1,1.5)
\pscurve[linestyle=dotted,linewidth=0.8pt](1,0)(0.8,1.1)(1,1.25)(1.2,1.1)(1,0)
%
}
\end{pspicture}
\end{center}
\caption{Patterns of Stokes regions:
(a) regular horizontal strip, 
(b) degenerate horizontal strip,
(c) half plane,
and (d) horizontal strip of simple-pole type.
The dotted lines are representatives of trajectories.}
\label{fig:Stokes1}
\end{figure}
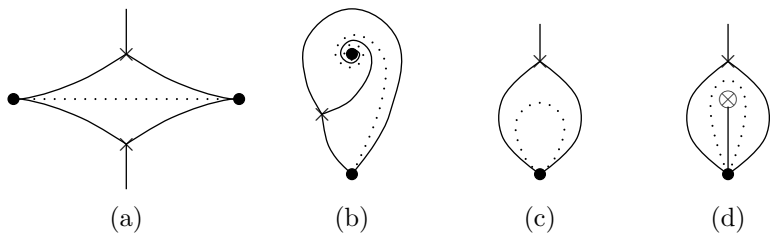
The {\em Stokes triangulation $T(\phi)$ of a bordered surface 
$(\bfS(\phi),\bfM(\phi),\bfA(\phi))$ associated with $G(\phi)$} 
was introduced in \cite[Section 6]{Iwaki14a}, refining
  the ideas of \cite{Kawai05,Gaiotto09,Bridgeland13}.
Let us roughly summarize the construction as follows.
 \begin{itemize}
 \item $\bfS(\phi)$: This is obtained from $\Sigma$ 
 (as a topological surface) by cutting out small holes around  
 poles of $\phi$ of order greater than 2.
 \item $\bfM(\phi)$: At each boundary component,
 $r-2$ marked points are attached, where $r\geq 3$ is the order
 of the corresponding pole. The poles of order 2 of $\phi$
 are the punctures.
 \item $\bfA(\phi)$: The turning points are the midpoints.
 \item Each arc of $T(\phi)$ corresponds to 
a representative of trajectories in a horizontal strip.
 To be more specific, inner arcs in
type iii pieces are the ones for degenerate horizontal strips, while
the other arcs are for regular horizontal strips.
 \item Each edge on the boundary of $\bfS(\phi)$
 corresponds to a representative of trajectories
 in a half plane. 
\item
A label of horizontal strips of $G(\phi)$ gives a label of $T(\phi)$.
 \end{itemize}

Now we remove the condition (B) and allow $\phi$ to have simple poles.
Then, we have an additional type of Stokes regions
given in Figure \ref{fig:Stokes1} (d),
which we call a {\em horizontal strip of simple-pole type}.
Then, in addition to the above, we have
the following correspondence shown in Figure \ref{fig:orbi1}. 
\begin{itemize}
\item ${\bf Q}(\phi)$: The simple poles of $\phi$ 
are the orbifold points.
\item Each pending arc of $T(\phi)$ corresponds to a representative 
of trajectories around a simple pole.
\end{itemize}
As seen in Figure \ref{fig:orbi1}, there are  
two possible arrangements of horizontal strips of simple-pole type
in Stokes graphs which correspond to type i and type ii pieces,
respectively.

In summary, we obtain the labeled Stokes triangulation $T(\phi)$ 
of the orbifold 
$(\bfS(\phi),\allowbreak\bfM(\phi),\allowbreak\bfQ(\phi),\bfA(\phi))$
which is associated with a saddle-free Stokes graph $G(\phi)$.

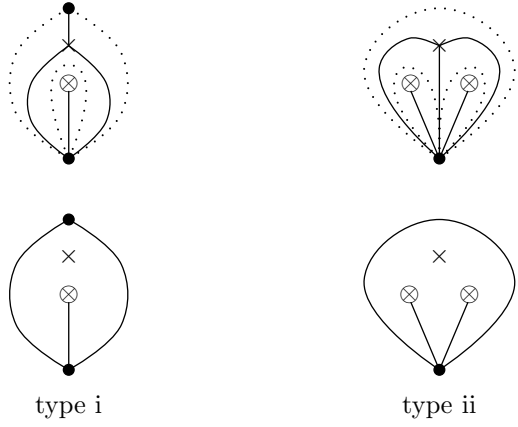
\begin{figure}
\begin{center}
\begin{pspicture}(0,-0.5)(2,2.2)
%
\psset{linewidth=0.5pt}
\pscircle[fillstyle=solid, fillcolor=black](1,0){0.08} 
\pscircle[fillstyle=solid, fillcolor=black](1,2){0.08} 
\rput[c]{0}(1,1.5){$\times$}
\rput[c]{0}(1,1){$\otimes$}
\psline(1,0.9)(1,0)
\psline(1,1.5)(1,2)
\pscurve(1,0)(0.5,0.5)(0.5,1)(1,1.5)
\pscurve(1,0)(1.5,0.5)(1.5,1)(1,1.5)
\pscurve[linestyle=dotted,linewidth=0.8pt](1,0)(0.8,1.1)(1,1.25)(1.2,1.1)(1,0)
\pscurve[linestyle=dotted,linewidth=0.8pt](1,0)(0.3,0.6)(0.3,1.4)(1,2)
\pscurve[linestyle=dotted,linewidth=0.8pt](1,0)(1.7,0.6)(1.7,1.4)(1,2)
\end{pspicture}
\hskip80pt
\begin{pspicture}(0,-0.5)(2,2.2)
%
\psset{linewidth=0.5pt}
\pscircle[fillstyle=solid, fillcolor=black](1,0){0.08} 
\rput[c]{0}(1,1.5){$\times$}
\rput[c]{0}(0.62,1){$\otimes$}
\rput[c]{0}(1.4,1){$\otimes$}

\psline(0.62,0.9)(1,0)
\psline(1.38,0.9)(1,0)
\psline(1,0)(1,1.5)
\pscurve(1,0)(0.2,1.2)(0.6,1.6)(1,1.5)
\pscurve(1,0)(1.8,1.2)(1.4,1.6)(1,1.5)
\pscurve[linestyle=dotted,linewidth=0.8pt](1,0)(0.4,0.9)(0.5,1.2)(0.8,1.1)(1,0)
\pscurve[linestyle=dotted,linewidth=0.8pt](1,0)(1.6,0.9)(1.5,1.2)(1.2,1.1)(1,0)
\pscurve[linestyle=dotted,linewidth=0.8pt](1,0)(0,1.2)(1,2)(2,1.2)(1,0)
\end{pspicture}
\end{center}
\begin{center}
\begin{pspicture}(0,-0.5)(2,2.2)
%
\psset{linewidth=0.5pt}
\pscircle[fillstyle=solid, fillcolor=black](1,0){0.08} 
\pscircle[fillstyle=solid, fillcolor=black](1,2){0.08} 
\rput[c]{0}(1,1.5){$\times$}
\rput[c]{0}(1,1){$\otimes$}
\rput[c]{0}(1,-0.5){type i}
\psline(1,0.9)(1,0)
\pscurve(1,0)(0.3,0.6)(0.3,1.4)(1,2)
\pscurve(1,0)(1.7,0.6)(1.7,1.4)(1,2)
\end{pspicture}
\hskip80pt
\begin{pspicture}(0,-0.5)(2,2.2)
%
\psset{linewidth=0.5pt}
\pscircle[fillstyle=solid, fillcolor=black](1,0){0.08} 
\rput[c]{0}(1,1.5){$\times$}
\rput[c]{0}(0.6,1){$\otimes$}
\rput[c]{0}(1.4,1){$\otimes$}
\rput[c]{0}(1,-0.5){type ii}

\psline(0.62,0.9)(1,0)
\psline(1.38,0.9)(1,0)
\pscurve(1,0)(0,1.2)(1,2)(2,1.2)(1,0)
\end{pspicture}
\end{center}
\caption{Arrangements of 
horizontal strips of simple-pole type in
Stokes graphs (above)
and corresponding Stokes triangulations.} 
\label{fig:orbi1}
\end{figure}

\begin{figure}
\begin{center}
\begin{pspicture}(0,-0.5)(2,2.2)
%
\psset{linewidth=0.5pt}
\pscircle[fillstyle=solid, fillcolor=black](0.4,0.4){0.08} 
\pscircle[fillstyle=solid, fillcolor=black](1.6,0.4){0.08} 
%
\rput[c]{0}(1,1.6){$\otimes$}
\rput[c]{0}(1,1){$\times$}
\rput[c]{0}(1,-0.5){type i-a}
\psline[linewidth=2pt](1,1)(1,1.5)
\psline(1,1)(0.4,0.4)
\psline(1,1)(1.6,0.4)
\end{pspicture}
\begin{pspicture}(0,-0.5)(2,2.2)
%
\psset{linewidth=0.5pt}
\pscircle[fillstyle=solid, fillcolor=black](1,0){0.08} 
%
\rput[c]{0}(1,2){$\otimes$}
\rput[c]{0}(1,1.4){$\times$}
\rput[c]{0}(1,-0.5){type i-b}
\psline[linewidth=2pt](1,1.9)(1,1.4)
\pscurve(1,0)(0.5,0.5)(0.5,0.9)(1,1.4)
\pscurve(1,0)(1.5,0.5)(1.5,0.9)(1,1.4)
\end{pspicture}
\begin{pspicture}(0,-0.5)(2,2.2)
%
\psset{linewidth=0.5pt}
\pscircle[fillstyle=solid, fillcolor=black](1,2){0.08} 
%
\rput[c]{0}(1,0){$\times$}
\rput[c]{0}(1,0.6){$\otimes$}
\rput[c]{0}(1,-0.5){type i-c}
\psline[linewidth=2pt](1,0.5)(1,0)
\pscurve(1,0)(0.3,0.6)(0.3,1.4)(1,2)
\pscurve(1,0)(1.7,0.6)(1.7,1.4)(1,2)
%
%
\end{pspicture}
\begin{pspicture}(0,-0.5)(2,2.2)
%
\psset{linewidth=0.5pt}
\pscircle[fillstyle=solid, fillcolor=black](1,0){0.08} 
\pscircle[fillstyle=solid, fillcolor=black](1,0.9){0.08} 
%
\rput[c]{0}(1,2){$\otimes$}
\rput[c]{0}(1,1.4){$\times$}
\rput[c]{0}(1,0.6){$\times$}
\rput[c]{0}(1,-0.5){type i-d}
\psline[linewidth=2pt](1,1.9)(1,1.4)
\psline(1,0.6)(1,0)
\pscurve(1,0)(0.5,0.5)(0.5,0.9)(1,1.4)
\pscurve(1,0)(1.5,0.5)(1.5,0.9)(1,1.4)
\pscurve(1,0.6)(1.15,0.9)(1,1.05)(0.85,0.9)(1,0.78)(1.1,0.9)(1,0.95)
\pscurve(1,0.6)(0.73,0.9)(1,1.2)(1.3,0.9)(1,0)
\end{pspicture}
\begin{pspicture}(0,-0.5)(2,2.2)
%
\psset{linewidth=0.5pt}
\pscircle[fillstyle=solid, fillcolor=black](1,2){0.08} 
\pscircle[fillstyle=solid, fillcolor=black](1,1.1){0.08} 
\rput[c]{0}(1,0){$\times$}
\rput[c]{0}(1,0.6){$\otimes$}
\rput[c]{0}(1,1.4){$\times$}
\rput[c]{0}(1,-0.5){type i-e}
\psline[linewidth=2pt](1,0.5)(1,0)
\psline(1,2)(1,1.4)
\pscurve(1,0)(0.3,0.6)(0.3,1.4)(1,2)
\pscurve(1,0)(1.7,0.6)(1.7,1.4)(1,2)
\pscurve(1,1.4)(1.15,1.1)(1,0.95)(0.85,1.1)(1,1.22)(1.1,1.1)(1,1.05)
\pscurve(1,1.4)(0.73,1.1)(1,0.8)(1.3,1.1)(1,2)
\end{pspicture}
\end{center}
\vskip10pt
\begin{center}
\begin{pspicture}(0,-0.5)(2,2.2)
%
\psset{linewidth=0.5pt}
\pscircle[fillstyle=solid, fillcolor=black](1,0){0.08} 
%
\rput[c]{0}(1,2){$\otimes$}
\rput[c]{0}(1,1.4){$\times$}
\rput[c]{0}(1,0.6){$\otimes$}
\rput[c]{0}(1,-0.5){type ii-a}
\psline[linewidth=2pt](1,1.9)(1,1.4)
\psline(1,0.5)(1,0)
\pscurve(1,0)(0.5,0.5)(0.5,0.9)(1,1.4)
\pscurve(1,0)(1.5,0.5)(1.5,0.9)(1,1.4)
\end{pspicture}
\begin{pspicture}(0,-0.5)(2,2.2)
%
\psset{linewidth=0.5pt}
\pscircle[fillstyle=solid, fillcolor=black](1,2){0.08} 
\rput[c]{0}(1,0){$\times$}
\rput[c]{0}(1,0.6){$\otimes$}
\rput[c]{0}(1,1.4){$\otimes$}
\rput[c]{0}(1,-0.5){type ii-b}
\psline[linewidth=2pt](1,0.5)(1,0)
\psline(1,2)(1,1.5)
\pscurve(1,0)(0.3,0.6)(0.3,1.4)(1,2)
\pscurve(1,0)(1.7,0.6)(1.7,1.4)(1,2)
\end{pspicture}
\end{center}
\caption{Local configurations near
type II Stokes segment.
Type II Stokes segments are depicted in bold lines.} 
\label{fig:orbi2}
\end{figure}
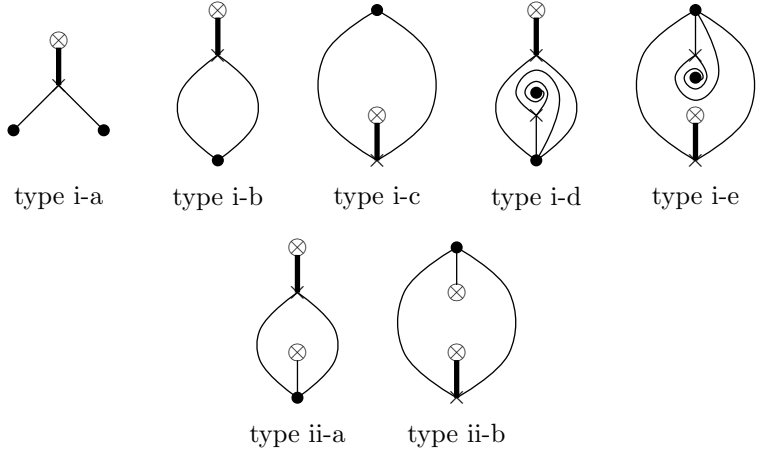

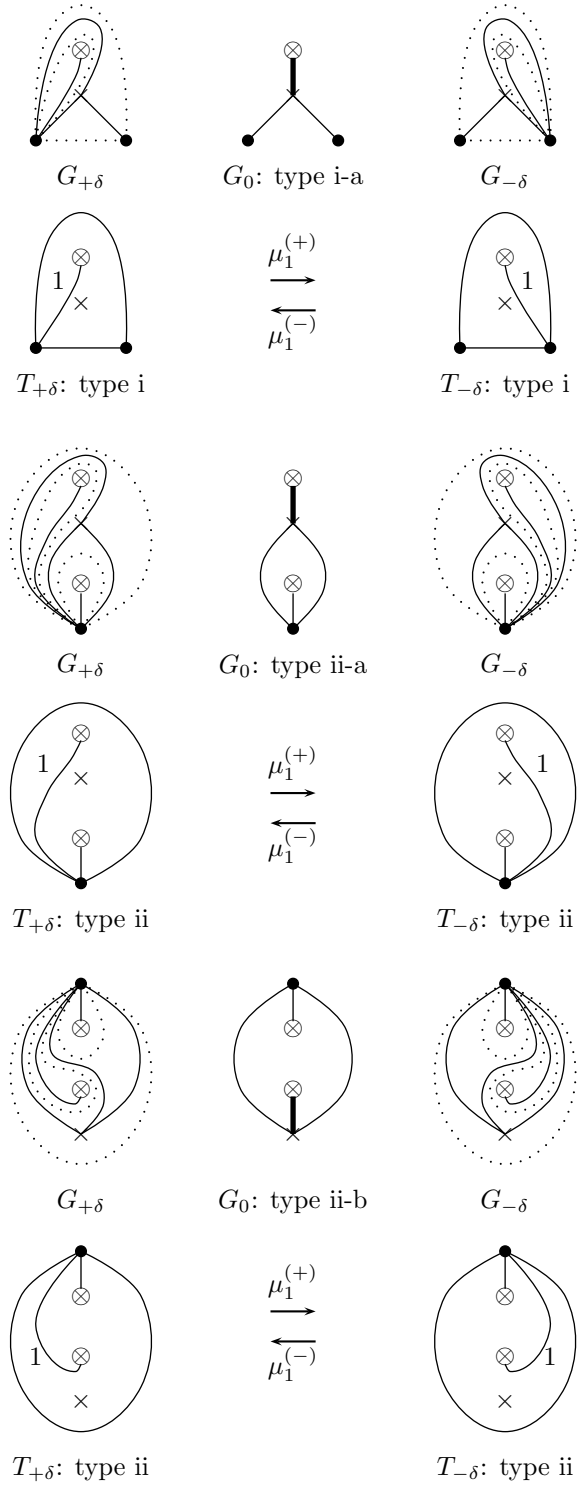
\begin{figure}
\begin{center}
\begin{pspicture}(0,-0.1)(2,2.2)
%
\psset{linewidth=0.5pt}
\pscircle[fillstyle=solid, fillcolor=black](0.4,0.4){0.08} 
\pscircle[fillstyle=solid, fillcolor=black](1.6,0.4){0.08} 
%
\rput[c]{0}(1,1.6){$\otimes$}
\rput[c]{0}(1,1){$\times$}
\rput[c]{0}(1,-0.1){$G_{+\delta}$}
%
\psline(1,1)(0.4,0.4)
\psline(1,1)(1.6,0.4)
\pscurve(0.4,0.4)(0.95,1.3)(1,1.5)
\pscurve(0.4,0.4)(0.95,2)(1.25,1.9)(1,1)
\pscurve[linestyle=dotted,linewidth=0.8pt](0.4,0.4)(0.8,1.6)
(1.1,1.8)(1.15,1.5)(0.4,0.4)
\pscurve[linestyle=dotted,linewidth=0.8pt](0.4,0.4)(0.6,1.9)
(1,2.2)(1.4,1.9)(1.6,0.4)
\psline[linestyle=dotted,linewidth=0.8pt](0.4,0.4)(1.6,0.4)
\end{pspicture}
\hskip20pt
\begin{pspicture}(0,-0.1)(2,2.2)
%
\psset{linewidth=0.5pt}
\pscircle[fillstyle=solid, fillcolor=black](0.4,0.4){0.08} 
\pscircle[fillstyle=solid, fillcolor=black](1.6,0.4){0.08} 
%
\rput[c]{0}(1,1.6){$\otimes$}
\rput[c]{0}(1,1){$\times$}
\rput[c]{0}(1,-0.1){$G_{0}$: type i-a}
\psline[linewidth=2pt](1,1)(1,1.5)
\psline(1,1)(0.4,0.4)
\psline(1,1)(1.6,0.4)
\end{pspicture}
\hskip20pt
\begin{pspicture}(0,-0.1)(2,2.2)
%
\psset{linewidth=0.5pt}
\pscircle[fillstyle=solid, fillcolor=black](0.4,0.4){0.08} 
\pscircle[fillstyle=solid, fillcolor=black](1.6,0.4){0.08} 
%
\rput[c]{0}(1,1.6){$\otimes$}
\rput[c]{0}(1,1){$\times$}
\rput[c]{0}(1,-0.1){$G_{-\delta}$}
%
\psline(1,1)(0.4,0.4)
\psline(1,1)(1.6,0.4)
\pscurve(1.6,0.4)(1.05,1.3)(1,1.5)
\pscurve(1.6,0.4)(1.05,2)(0.75,1.9)(1,1)
\pscurve[linestyle=dotted,linewidth=0.8pt](1.6,0.4)
(1.2,1.6)(0.9,1.8)(0.85,1.5)(1.6,0.4)
\pscurve[linestyle=dotted,linewidth=0.8pt](0.4,0.4)
(0.6,1.9)(1,2.2)(1.4,1.9)(1.6,0.4)
\psline[linestyle=dotted,linewidth=0.8pt](0.4,0.4)(1.6,0.4)
\end{pspicture}
\end{center}
\vskip10pt
\begin{center}
\begin{pspicture}(0,-0.1)(2,2.2)
%
\psset{linewidth=0.5pt}
\pscircle[fillstyle=solid, fillcolor=black](0.4,0.4){0.08} 
\pscircle[fillstyle=solid, fillcolor=black](1.6,0.4){0.08} 
%
\rput[c]{0}(1,1.6){$\otimes$}
\rput[c]{0}(1,1){$\times$}
\rput[c]{0}(1,-0.1){$T_{+\delta}$: type i}
\rput[c]{0}(0.7,1.3){1}
%
%
\pscurve(0.4,0.4)(0.95,1.3)(1,1.5)
\pscurve(0.4,0.4)(0.6,1.9)(1,2.2)(1.4,1.9)(1.6,0.4)
\psline(0.4,0.4)(1.6,0.4)
\end{pspicture}
\hskip20pt
\begin{pspicture}(0,-0.1)(2,2.2)
%
\psline[arrows=->](0.7,1.3)(1.3,1.3)
\psline[arrows=<-](0.7,0.9)(1.3,0.9)
\rput[c]{0}(1,1.7){$\mu_{1}^{(+)}$}
\rput[c]{0}(1,0.6){$\mu_{1}^{(-)}$}
\end{pspicture}
\hskip20pt
\begin{pspicture}(0,-0.1)(2,2.2)
%
\psset{linewidth=0.5pt}
\pscircle[fillstyle=solid, fillcolor=black](0.4,0.4){0.08} 
\pscircle[fillstyle=solid, fillcolor=black](1.6,0.4){0.08} 
%
\rput[c]{0}(1,1.6){$\otimes$}
\rput[c]{0}(1,1){$\times$}
\rput[c]{0}(1,-0.1){$T_{-\delta}$: type i}
\rput[c]{0}(1.3,1.3){1}
%
%
\pscurve(1.6,0.4)(1.05,1.3)(1,1.5)
\pscurve(0.4,0.4)(0.6,1.9)(1,2.2)(1.4,1.9)(1.6,0.4)
\psline(0.4,0.4)(1.6,0.4)
\end{pspicture}
\end{center}
\vskip15pt
\begin{center}
\begin{pspicture}(0,-0.5)(2,2.6)
%
\psset{linewidth=0.5pt}
\pscircle[fillstyle=solid, fillcolor=black](1,0){0.08} 
%
\rput[c]{0}(1,2){$\otimes$}
\rput[c]{0}(1,1.4){$\times$}
\rput[c]{0}(1,0.6){$\otimes$}
\rput[c]{0}(1,-0.5){$G_{+\delta}$}
\psline(1,0.47)(1,0)
\pscurve(1,0)(0.6,0.5)(0.6,0.9)(1,1.4)
\pscurve(1,0)(1.4,0.5)(1.4,0.9)(1,1.4)
\pscurve(1,0)(0.4,0.6)(0.55,1.2)(0.9,1.7)(1,1.9)
\pscurve(1,0)(0.3,0.6)(0.3,1.6)(1,2.3)(1.3,2.2)(1.3,1.9)(1,1.4)
\pscurve[linestyle=dotted,linewidth=0.8pt](1,0)(0.35,0.6)
(0.35,1.2)(1,2.2)(1.2,2.1)(1.2,1.95)
(1.1,1.7)(0.65,1.2)(0.45,0.6)(1,0)
\pscurve[linestyle=dotted,linewidth=0.8pt](1,0)(0.2,0.6)
(0.2,1.8)(1,2.4)(1.8,1.8)(1.8,0.6)(1,0)
\pscurve[linestyle=dotted,linewidth=0.8pt](1,0)(0.7,0.7)(1,1)(1.3,0.7)(1,0)
\end{pspicture}
\hskip20pt
\begin{pspicture}(0,-0.5)(2,2.2)
%
\psset{linewidth=0.5pt}
\pscircle[fillstyle=solid, fillcolor=black](1,0){0.08} 
%
\rput[c]{0}(1,2){$\otimes$}
\rput[c]{0}(1,1.4){$\times$}
\rput[c]{0}(1,0.6){$\otimes$}
\rput[c]{0}(1,-0.5){$G_{0}$: type ii-a}
\psline[linewidth=2pt](1,1.9)(1,1.4)
\psline(1,0.47)(1,0)
\pscurve(1,0)(0.6,0.5)(0.6,0.9)(1,1.4)
\pscurve(1,0)(1.4,0.5)(1.4,0.9)(1,1.4)
\end{pspicture}
\hskip20pt
\begin{pspicture}(0,-0.5)(2,2.2)
%
\psset{linewidth=0.5pt}
\pscircle[fillstyle=solid, fillcolor=black](1,0){0.08} 
%
\rput[c]{0}(1,2){$\otimes$}
\rput[c]{0}(1,1.4){$\times$}
\rput[c]{0}(1,0.6){$\otimes$}
\rput[c]{0}(1,-0.5){$G_{-\delta}$}
\psline(1,0.47)(1,0)
\pscurve(1,0)(0.6,0.5)(0.6,0.9)(1,1.4)
\pscurve(1,0)(1.4,0.5)(1.4,0.9)(1,1.4)
%
\pscurve(1,0)(1.6,0.6)(1.45,1.2)(1.1,1.7)(1,1.9)
\pscurve(1,0)(1.7,0.6)(1.7,1.6)(1,2.3)(0.7,2.2)(0.7,1.9)(1,1.4)
\pscurve[linestyle=dotted,linewidth=0.8pt](1,0)(1.65,0.6)
(1.65,1.2)(1,2.2)(0.8,2.1)(0.8,1.95)
(0.9,1.7)(1.35,1.2)(1.55,0.6)(1,0)
\pscurve[linestyle=dotted,linewidth=0.8pt](1,0)(0.2,0.6)
(0.2,1.8)(1,2.4)(1.8,1.8)(1.8,0.6)(1,0)
\pscurve[linestyle=dotted,linewidth=0.8pt](1,0)(0.7,0.7)
(1,1)(1.3,0.7)(1,0)
\end{pspicture}
\end{center}
\vskip5pt
\begin{center}
\begin{pspicture}(0,-0.5)(2,2.6)
%
\psset{linewidth=0.5pt}
\pscircle[fillstyle=solid, fillcolor=black](1,0){0.08} 
%
\rput[c]{0}(1,2){$\otimes$}
\rput[c]{0}(1,1.4){$\times$}
\rput[c]{0}(1,0.6){$\otimes$}
\rput[c]{0}(0.5,1.6){1}
\rput[c]{0}(1,-0.5){$T_{+\delta}$: type ii}
\psline(1,0.48)(1,0)
\pscurve(1,0)(0.4,0.6)(0.55,1.2)(0.9,1.7)(1,1.9)
\pscurve(1,0)(0.2,0.6)(0.2,1.8)(1,2.4)(1.8,1.8)(1.8,0.6)(1,0)
%
\end{pspicture}
\hskip20pt
\begin{pspicture}(0,-0.5)(2,2.2)
\psline[arrows=->](0.7,1.2)(1.3,1.2)
\psline[arrows=<-](0.7,0.8)(1.3,0.8)
\rput[c]{0}(1,1.6){$\mu_{1}^{(+)}$}
\rput[c]{0}(1,0.5){$\mu_{1}^{(-)}$}
\end{pspicture}
\hskip20pt
\begin{pspicture}(0,-0.5)(2,2.2)
%
\psset{linewidth=0.5pt}
\pscircle[fillstyle=solid, fillcolor=black](1,0){0.08} 
%
\rput[c]{0}(1,2){$\otimes$}
\rput[c]{0}(1,1.4){$\times$}
\rput[c]{0}(1,0.6){$\otimes$}
\rput[c]{0}(1.5,1.6){1}
\rput[c]{0}(1,-0.5){$T_{-\delta}$: type ii}
\psline(1,0.48)(1,0)
%
\pscurve(1,0)(1.6,0.6)(1.45,1.2)(1.1,1.7)(1,1.9)
\pscurve(1,0)(0.2,0.6)(0.2,1.8)(1,2.4)(1.8,1.8)(1.8,0.6)(1,0)
%
\end{pspicture}
\end{center}
\vskip15pt
\begin{center}
\begin{pspicture}(0,-0.9)(2,2.2)
%
\psset{linewidth=0.5pt}
\pscircle[fillstyle=solid, fillcolor=black](1,2){0.08} 
\rput[c]{0}(1,0){$\times$}
\rput[c]{0}(1,0.6){$\otimes$}
\rput[c]{0}(1,1.4){$\otimes$}
\rput[c]{0}(1,-0.9){$G_{+\delta}$}
\psline(1,2)(1,1.5)
\pscurve(1,0)(0.3,0.6)(0.3,1.4)(1,2)
\pscurve(1,0)(1.7,0.6)(1.7,1.4)(1,2)
\pscurve[linestyle=dotted,linewidth=0.8pt](1,2)(0.2,1.4)
(0.2,0.2)(1,-0.4)(1.8,0.2)(1.8,1.4)(1,2)
\pscurve[linestyle=dotted,linewidth=0.8pt](1,2)(0.7,1.3)
(1,1)(1.3,1.3)(1,2)
\pscurve(1,2)(0.4,1)(0.92,0.4)(1,0.5)
\pscurve(1,2)(0.6,1.1)(1.2,0.8)(1.3,0.4)(1,0)
\pscurve[linestyle=dotted,linewidth=0.8pt](1,2)(0.5,1)
(1,0.8)(1.2,0.6)(1,0.3)(0.4,0.6)(0.4,1.4)(1,2)
\end{pspicture}
\hskip20pt
\begin{pspicture}(0,-0.9)(2,2.2)
%
\psset{linewidth=0.5pt}
\pscircle[fillstyle=solid, fillcolor=black](1,2){0.08} 
\rput[c]{0}(1,0){$\times$}
\rput[c]{0}(1,0.6){$\otimes$}
\rput[c]{0}(1,1.4){$\otimes$}
\rput[c]{0}(1,-0.9){$G_{0}$: type ii-b}
\psline[linewidth=2pt](1,0.5)(1,0)
\psline(1,2)(1,1.5)
\pscurve(1,0)(0.3,0.6)(0.3,1.4)(1,2)
\pscurve(1,0)(1.7,0.6)(1.7,1.4)(1,2)
\end{pspicture}
\hskip20pt
\begin{pspicture}(0,-0.9)(2,2.2)
%
\psset{linewidth=0.5pt}
\pscircle[fillstyle=solid, fillcolor=black](1,2){0.08} 
\rput[c]{0}(1,0){$\times$}
\rput[c]{0}(1,0.6){$\otimes$}
\rput[c]{0}(1,1.4){$\otimes$}
\rput[c]{0}(1,-0.9){$G_{-\delta}$}
\psline(1,2)(1,1.5)
\pscurve(1,0)(0.3,0.6)(0.3,1.4)(1,2)
\pscurve(1,0)(1.7,0.6)(1.7,1.4)(1,2)
\pscurve[linestyle=dotted,linewidth=0.8pt](1,2)(0.2,1.4)
(0.2,0.2)(1,-0.4)(1.8,0.2)(1.8,1.4)(1,2)
\pscurve[linestyle=dotted,linewidth=0.8pt](1,2)(0.7,1.3)(1,1)(1.3,1.3)(1,2)
\pscurve(1,2)(1.6,1)(1.07,0.4)(1,0.5)
\pscurve(1,2)(1.4,1.1)(0.8,0.8)(0.7,0.4)(1,0)
\pscurve[linestyle=dotted,linewidth=0.8pt](1,2)(1.5,1)
(1,0.8)(0.8,0.6)(1,0.3)(1.6,0.6)(1.6,1.4)(1,2)
\end{pspicture}
\end{center}
\vskip10pt
\begin{center}
\begin{pspicture}(0,-0.9)(2,2.2)
%
\psset{linewidth=0.5pt}
\pscircle[fillstyle=solid, fillcolor=black](1,2){0.08} 
\rput[c]{0}(1,0){$\times$}
\rput[c]{0}(1,0.6){$\otimes$}
\rput[c]{0}(1,1.4){$\otimes$}
\rput[c]{0}(1,-0.9){$T_{+\delta}$: type ii}
\rput[c]{0}(0.4,0.6){1}
\psline(1,2)(1,1.5)
%
\pscurve(1,2)(0.2,1.4)(0.2,0.2)(1,-0.4)(1.8,0.2)(1.8,1.4)(1,2)
%
\pscurve(1,2)(0.4,1)(0.93,0.4)(1,0.5)
\end{pspicture}
\hskip20pt
\begin{pspicture}(0,-0.9)(2,2.2)
\psline[arrows=->](0.7,1.2)(1.3,1.2)
\psline[arrows=<-](0.7,0.8)(1.3,0.8)
\rput[c]{0}(1,1.6){$\mu_{1}^{(+)}$}
\rput[c]{0}(1,0.5){$\mu_{1}^{(-)}$}
\end{pspicture}
\hskip20pt
\begin{pspicture}(0,-0.9)(2,2.2)
%
\psset{linewidth=0.5pt}
\pscircle[fillstyle=solid, fillcolor=black](1,2){0.08} 
\rput[c]{0}(1,0){$\times$}
\rput[c]{0}(1,0.6){$\otimes$}
\rput[c]{0}(1,1.4){$\otimes$}
\rput[c]{0}(1,-0.9){$T_{-\delta}$: type ii}
\rput[c]{0}(1.6,0.6){1}
\psline(1,2)(1,1.5)
%
%
\pscurve(1,2)(0.2,1.4)(0.2,0.2)(1,-0.4)(1.8,0.2)(1.8,1.4)(1,2)
%
\pscurve(1,2)(1.6,1)(1.07,0.4)(1,0.5)
\end{pspicture}
\end{center}
\caption{Mutations of Stokes graphs and signed flips of Stokes triangulations.} 
\label{fig:orbi3}
\end{figure}

 \subsection{Mutation of Stokes graphs}
 \label{subsec:mutation1}
 Now we are at the important stage of combining  two main ideas: 
 the mutation of Stokes graphs of the Schr{\"o}dinger equation 
 (see Section \ref{section:stokes-auto})
 and the signed flips of Stokes triangulations.
 Here, we concentrate on the situation
 where the jump formula with respect to a type II Stokes segment
in Theorem \ref{thm:Stokes-auto-II} is relevant.
See \cite[Section 6.3]{Iwaki14a} for a more general setting.

Let $\phi$ be the quadratic differential associated with the 
Schr{\"o}dinger equation \eqref{eq:Sch}, and $G_0=G(\phi)$ 
be its Stokes graph. We assume that
\begin{itemize}
\item $G_0$ has a unique Stokes segment, which is of type II.
\end{itemize}
Then, there are seven possible local configurations near a type II
Stokes segment in $G_0$ depicted in Figure \ref{fig:orbi2}.
We call  five of them {\em type i configurations\/} and two of them
 {\em type ii configurations\/}
as specified in Figure \ref{fig:orbi2}.

Let $Q^{(\theta)}(z,\eta)$ be the $S^1$-family of potentials 
defined in \eqref{eq:S1-potentials}. 
Take a sufficiently small $\delta>0$ and 
consider the Stokes graphs $G_{\pm \delta}=G(e^{\pm 2i\delta}\phi)$
for $Q^{({\pm 2i\delta})}(z,\eta)$.
Then, $G_{\pm \delta}$ are saddle-free 
as explained in Section \ref{section:stokes-auto}.
Let $T_{\pm\delta} = T(e^{\pm 2 i \delta} \phi)$ be the  Stokes triangulations
associated with $G_{\pm \delta}$ defined in the previous subsection.
Note that the zeros and poles are common for $G_{\pm \delta}$;
therefore, marked points, orbifold points, and midpoints, respectively, 
in $T_{-\delta}$ and  $T_{+\delta}$ are naturally identified.

The following property is the conclusion of this section.

\begin{thm}
\label{thm:mut1}
 If $G(\phi)$ contains a type i (resp. type ii) configuration, 
 each of the Stokes triangulation $T_{+\delta}$ and $T_{-\delta}$ 
 contains a type i (resp. type ii) piece
 such that
 $T_{\pm \delta}$
undergo signed flips $T_{-\delta}=\mu^{(+)}_k(T_{+\delta})$
and $T_{+\delta}=\mu^{(-)}_k(T_{-\delta})$
at the pending arcs therein
under the natural identification of the labels of
$T_{+ \delta}$ and $T_{- \delta}$.
\end{thm}

In Figure \ref{fig:orbi3}
we show explicitly how this happens
in  three essential cases of type i-a, ii-a, and ii-b.
The rest of the cases are similar.

\section{Mutation formula of Voros symbols for signed flips}
\label{section:main-result}

\subsection{Simple paths and simple cycles}

In this subsection we extend the notion of the
{\em simple paths\/} and the {\em simple cycles\/} 
of a saddle-free Stokes graph defined in \cite{Iwaki14a} 
to the simple pole case.

Let $\phi$ be the quadratic differential associated with 
the Schr{\"o}dinger equation \eqref{eq:Sch}. 
Recall that we call elements of 
$H_1(\hat{\Sigma}\setminus (\hat{P}_0\cup
\hat{P}_{\rm s}),\hat{P}_{\infty})$ and
$H_1(\hat{\Sigma}\setminus \hat{P})$
  {\em paths\/} and {\em cycles\/},
respectively (see Section \ref{section:Voros-symbols}).
Let $\beta^*$ (resp., $\gamma^*$)
 be the image of a path $\beta$ (resp.,  a cycle  $\gamma$) 
 by the  covering involution 
 of $\hat{\Sigma}$ while keeping the direction.
Let
\begin{align}
\label{eq:sym1}
\mathrm{Sym}(H_1(\hat{\Sigma}\setminus (\hat{P}_0\cup
\hat{P}_{\rm s}),\hat{P}_{\infty}))
&=
\{ \beta \in H_1(\hat{\Sigma}\setminus (\hat{P}_0\cup
\hat{P}_{\rm s}),\hat{P}_{\infty})
\mid
\beta^*=\beta
\},\\
\mathrm{Sym}(H_1(\hat{\Sigma}\setminus \hat{P}))
&=
\{ \gamma \in H_1(\hat{\Sigma}\setminus \hat{P})
\mid
\gamma^*=\gamma
\}.
\end{align}
We introduce the quotients of the homology groups,
\begin{align}
{\Gamma}^{\vee}&=
H_1(\hat{\Sigma}\setminus (\hat{P}_0\cup
\hat{P}_{\rm s}),\hat{P}_{\infty})/
\mathrm{Sym}(H_1(\hat{\Sigma}\setminus (\hat{P}_0\cup
\hat{P}_{\rm s}),\hat{P}_{\infty})),\\
{\Gamma}&=
H_1(\hat{\Sigma}\setminus \hat{P})/
\mathrm{Sym}(H_1(\hat{\Sigma}\setminus \hat{P})).
\end{align}

\begin{rem}
By definition, the formal series 
$S_{\rm odd}(z,\eta)$ is anti-invariant under the 
covering involution $\ast$ of $\hat{\Sigma}$: 
$S_{\rm odd}(z^{\ast},\eta) = -S_{\rm odd}(z,\eta)$. 
Hence, 
\begin{equation} 
\int_{\beta} \left( S_{\rm odd}(z,\eta)-\eta\sqrt{Q_0(z)} \right) dz = 0, 
\quad \oint_{\gamma} S_{\rm odd}(z,\eta) dz = 0
\end{equation}
hold for any $\beta \in {\rm Sym}(H_1(\hat{\Sigma}\setminus
(\hat{P}_0 \cup \hat{P}_{\rm s}),\hat{P}_{\infty}))$ 
and any $\gamma \in {\rm Sym}(H_1(\hat{\Sigma}\setminus\hat{P}))$.
Therefore, Voros symbols for $\beta \in \Gamma^{\vee}$ and 
$\gamma \in \Gamma$ are well-defined. 
\end{rem}

In this subsection, we assume 
\begin{itemize}
\item the Stokes graph $G(\phi)$ of \eqref{eq:Sch} is saddle-free.
\end{itemize}
Then, we can assign the {\em simple path\/} $\beta_i\in {\Gamma}^{\vee}$
and the {\em simple cycle\/} $\gamma_i\in {\Gamma}$ 
to each Stokes region $D_i$ ($i=1,\dots,n$) which is not a half plane,
as follows.
If $D_i$ is a regular or degenerate horizontal strip,
they are defined by Figures
\ref{fig:bg1} and \ref{fig:bg2} as in \cite[Section 6]{Iwaki14a}.
Here the signs $\oplus$ and $\ominus$ in the figures 
are assigned by the same rule as in 
Section \ref{section:proof-of-Stokes-auto}.
If $D_i$ is a horizontal strip of simple-pole type,
they are newly defined by Figure \ref{fig:bg3}.
Note that the intersection form \eqref{eq:intersection-form} 
is well-defined on the subgroups of $\Gamma^{\vee}$ 
and $\Gamma$ generated by the simple paths and simple cycles, 
respectively. 

Let $T=T(\phi)=(\alpha_i)_{i=1}^n$ be the labeled 
Stokes triangulation associated with $G(\phi)$ defined in 
Section \ref{subsection:SG-to-ST}, and let $B=B(T)$ be 
the signed adjacency matrix of $T$. Let $d_i$ ($i=1,\dots,n$) 
be the weight of the arc ${\alpha}_i$. 
Recall that $\tilde{B}= (\tilde{b}_{ij})_{i,j=1}^n =D^{-1}B$ 
is a skew-symmetric integer matrix 
(see \eqref{eq:bd1} and \eqref{eq:bd2}).
 
\begin{prop}
\label{prop:simple1}
 The following properties hold.
 \par
 (a).  Duality.
 \begin{align}
 \label{eq:dual1}
 \langle \gamma_i, \beta_j\rangle & = \delta_{ij}.
 \end{align}
 \par
  (b).  Intersection formula.
  \begin{align}
\label{eq:int1}
(\gamma_i,\gamma_j) &=\tilde{b}_{ij}. 
 \end{align}
  \par
  (c).  Decomposition formula.
  As an element of ${\Gamma}^{\vee}$,
  $\gamma_i$ decomposes as follows.
  \begin{align}
  \label{eq:decom1}
\gamma_i &=\sum_{j=1}^n\tilde{b}_{ji}\beta_j. 
 \end{align}
\end{prop}
\begin{proof} 
They are known in the non-simple pole case
in \cite[Propositions 6.25-6.27]{Iwaki14a}.
Thus, it is enough to prove them only when horizontal strips of
simple-pole type are involved.
This is done by inspecting Figures \ref{fig:bg1}--\ref{fig:bg3}.
\end{proof}

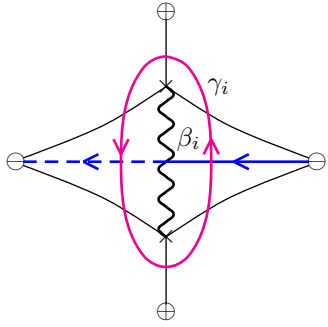
\begin{figure}
\begin{center}
\begin{pspicture}(-5,0)(-1,4.2)
\psset{unit=10mm}
%
\psset{fillstyle=solid, fillcolor=black}

\rput[c]{0}(-3,0){$\oplus$}
\rput[c]{0}(-3,4){$\oplus$}
\rput[c]{0}(-5,2){$\ominus$}
\rput[c]{0}(-1,2){$\ominus$}
\psset{fillstyle=none}
\pscurve[linewidth=1pt](-3,1)(-3.1,1.1)(-3,1.2)(-2.9,1.3)(-3,1.4)%
(-3.1,1.5)(-3,1.6)(-2.9,1.7)(-3,1.8)(-3.1,1.9)(-3,2)%
(-2.9,2.1)(-3,2.2)(-3.1,2.3)(-3,2.4)%
(-2.9,2.5)(-3,2.6)(-3.1,2.7)(-3,2.8)(-2.9,2.9)(-3,3)%
\psset{linewidth=0.5pt}

\psline(-3,3.9)(-3,3)
\psline(-3,0.1)(-3,1)
\pscurve(-3,3)(-2.2,2.5)(-1.1,2.05)
\pscurve(-3,1)(-2.2,1.5)(-1.1,1.95)
\pscurve(-3,3)(-3.8,2.5)(-4.9,2.05)
\pscurve(-3,1)(-3.8,1.5)(-4.9,1.95)

\psline[linewidth=1pt,linestyle=dashed,linecolor=blue](-4.9,2)(-3,2)
\psline[linewidth=1pt,linecolor=blue](-3,2)(-1.1,2)
\psecurve[linewidth=1pt,linecolor=magenta](-3.3,0.75)(-3,0.6)(-2.7,0.75)(-2.4,2)(-2.7,3.25)(-3,3.4)(-4.3,3.25)
\psecurve[linewidth=1pt,linecolor=magenta](-2.7,0.75)(-3,0.6)(-3.3,0.75)(-3.6,2)(-3.3,3.25)(-3,3.4)(-2.7,3.25)
\psline[linewidth=1pt,linecolor=magenta](-2.4,2.3)(-2.5,2.1)
\psline[linewidth=1pt,linecolor=magenta](-2.4,2.3)(-2.3,2.1)
\psline[linewidth=1pt,linecolor=magenta](-3.6,2.1)(-3.7,2.3)
\psline[linewidth=1pt,linecolor=magenta](-3.6,2.1)(-3.5,2.3)
\psline[linewidth=1pt,linecolor=blue](-2.1,2)(-1.9,2.1)
\psline[linewidth=1pt,linecolor=blue](-2.1,2)(-1.9,1.9)
\psline[linewidth=1pt,linecolor=blue](-4.1,2)(-3.9,2.1)
\psline[linewidth=1pt,linecolor=blue](-4.1,2)(-3.9,1.9)
\rput[c]{0}(-3,1){$\times$}
\rput[c]{0}(-3,3){$\times$}
\rput[c]{0}(-2.3,3){$\gamma_i$}
\rput[c]{0}(-2.7,2.3){$\beta_i$}
\end{pspicture}
\end{center}
\caption{Simple path and simple cycle
for a regular horizontal strip $D_i$ not surrounding 
a degenerate horizontal strip.}
\label{fig:bg1}
\end{figure}

%
\begin{figure}
\begin{center}
\begin{pspicture}(-1.4,0)(6,4.2)
\psset{unit=10mm}
\psset{linewidth=0.5pt}
\rput[c]{0}(0,0){$\oplus$}
\rput[c]{0}(0,2){$\oplus$}
\rput[c]{0}(0,4){$\ominus$}
\rput[c]{0}(0,3.2){$\times$}
\rput[c]{0}(0,1.2){$\times$}
\rput[c]{0}(-1.1,2.2){$\gamma_i{}$}
\rput[c]{0}(1.1,2.2){$\gamma_j{}$}
\rput[c]{0}(-0.53,0.8){$\beta_i{}$}
\rput[c]{0}(0.47,0.8){$\beta_j{}$}
\psset{linewidth=0.5pt,linestyle=solid}
\pscurve(0.1,0.05)(0.5,0.4)(0.9,1.6)(0.5,2.8)(0,3.2)
\pscurve(-0.1,0.05)(-0.5,0.4)(-0.9,1.6)(-0.5,2.8)(-0,3.2)
\pscurve(0,1.2)(0.4,1.4)(0.35,2.2)(0,2.4)(-0.5,1.8)(-0.1,0.1)

\pscurve(0,1.2)(-0.2,1.5)(-0.2,1.7)(-0.1,1.95)
\psline(0,0.1)(0,1.2)
\psline(0,3.2)(0,3.9)
\psset{linewidth=1pt,linestyle=solid}
\pscurve(0,1.2)(-0.1,1.1)(-0.2,1.2)(-0.3,1.3)(-0.4,1.2)
(-0.5,1.1)(-0.6,1.2)(-0.7,1.3)(-0.8,1.2)(-0.9,1.1)(-1.0,1.2)
(-1.1,1.3)(-1.2,1.2)(-1.3,1.1)(-1.4,1.2)
\pscurve(0,3.2)(-0.1,3.3)(-0.2,3.2)(-0.3,3.1)(-0.4,3.2)
(-0.5,3.3)(-0.6,3.2)(-0.7,3.1)(-0.8,3.2)(-0.9,3.3)(-1,3.2)
(-1.1,3.1)(-1.2,3.2)(-1.3,3.3)(-1.4,3.2)
\psecurve[linewidth=1pt,linestyle=solid,linecolor=magenta]
(-0.75,2.6)(-0.4,3.2)(-0.2,3.35)(0,3.4)(0.15,3.3)(0.15,3.1)(0,2.95)
(-0.4,2.2)(0,1.45)(0.15,1.3)(0.15,1.1)(0,1)(-0.2,1.05)(-0.4,1.2)(-0.75,1.8)
\psecurve[linewidth=1pt,linestyle=dashed,linecolor=magenta]
(-0.2,1.05)(-0.4,1.2)(-0.75,1.8)(-0.8,2.2)(-0.75,2.6)(-0.4,3.2)(-0.2,3.35)
\psecurve[linewidth=1pt,linestyle=solid,linecolor=magenta]
(-0.17,1.3)(-0.2,1.2)(-0.17,1.1)(0,1)(0.2,1.05)(0.4,1.2)(0.75,1.8)
(0.8,2.2)(0.75,2.6)(0.4,3.2)(0.2,3.35)(0,3.4)(-0.17,3.3)(-0.2,3.2)(-0.17,3.1)
\psecurve[linewidth=1pt,linestyle=dashed,linecolor=magenta]
(-0.17,3.3)(-0.2,3.2)(-0.17,3.1)(0,2.95)(0.4,2.2)(0,1.45)
(-0.17,1.3)(-0.2,1.2)(-0.17,1.1)
\psline[linewidth=1pt,linestyle=solid,linecolor=magenta](-0.4,2.2)(-0.3,2.4)
\psline[linewidth=1pt,linestyle=solid,linecolor=magenta](-0.4,2.2)(-0.5,2.4)
\psline[linewidth=1pt,linestyle=solid,linecolor=magenta](0.8,2.2)(0.7,2.4)
\psline[linewidth=1pt,linestyle=solid,linecolor=magenta](0.8,2.2)(0.9,2.4)
\psline[linewidth=1pt,linestyle=solid,linecolor=magenta](-0.8,2.4)(-0.7,2.2)
\psline[linewidth=1pt,linestyle=solid,linecolor=magenta](-0.8,2.4)(-0.9,2.2)
\psline[linewidth=1pt,linestyle=solid,linecolor=magenta](0.36,2.4)(0.5,2.22)
\psline[linewidth=1pt,linestyle=solid,linecolor=magenta](0.36,2.4)(0.32,2.18)

\pscurve[linewidth=1pt,linestyle=solid,linecolor=blue]
(0.1,2)(0.25,1.8)(0.3,1.1)(0.05,0.1)
\psecurve[linewidth=1pt,linestyle=solid,linecolor=blue]
(-0.1,2)(-0.1,2)(-0.25,1.8)(-0.3,1.3)(-0.05,0.1)
\psecurve[linewidth=1pt,linestyle=dashed,linecolor=blue]
(-0.25,1.8)(-0.3,1.3)(-0.05,0.1)(-0.05,0.1)
\psline[linewidth=1pt,linestyle=solid,linecolor=blue](-0.31,1.6)(-0.21,1.4)
\psline[linewidth=1pt,linestyle=solid,linecolor=blue](-0.31,1.6)(-0.41,1.4)
\psline[linewidth=1pt,linestyle=solid,linecolor=blue](0.32,1.4)(0.42,1.6)
\psline[linewidth=1pt,linestyle=solid,linecolor=blue](0.32,1.4)(0.22,1.6)
\psline[linewidth=1pt,linestyle=solid,linecolor=blue](-0.22,0.8)(-0.28,0.6)
\psline[linewidth=1pt,linestyle=solid,linecolor=blue](-0.22,0.8)(-0.1,0.62)

%
%
\psset{linewidth=0.5pt}
\rput[c]{0}(5,0){$\oplus$}
\rput[c]{0}(5,2){$\ominus$}
\rput[c]{0}(5,4){$\ominus$}
\rput[c]{0}(5,3.2){$\times$}
\rput[c]{0}(5,1.2){$\times$}
\rput[c]{0}(3.9,2.2){$\gamma_j$}
\rput[c]{0}(6.1,2.2){$\gamma_i$}
\rput[c]{0}(4.55,0.8){$\beta_j{}$}
\rput[c]{0}(5.52,0.8){$\beta_i{}$}

\psset{linewidth=0.5pt,linestyle=solid}
\pscurve(5.1,0.05)(5.5,0.4)(5.9,1.6)(5.5,2.8)(5,3.2)
\pscurve(4.9,0.05)(4.5,0.4)(4.1,1.6)(4.5,2.8)(5,3.2)
\pscurve(5,1.2)(4.6,1.4)(4.65,2.2)(5,2.4)(5.5,1.8)(5.1,0.1)
\pscurve(5,1.2)(5.2,1.5)(5.2,1.7)(5.1,1.95)
\psline(5,0.1)(5,1.2)
\psline(5,3.2)(5,3.9)
\psset{linewidth=1pt,linestyle=solid}
\pscurve(5,1.2)(4.9,1.1)(4.8,1.2)(4.7,1.3)(4.6,1.2)
(4.5,1.1)(4.4,1.2)(4.3,1.3)(4.2,1.2)(4.1,1.1)(4,1.2)
(3.9,1.3)(3.8,1.2)(3.7,1.1)(3.6,1.2)
\pscurve(5,3.2)(4.9,3.3)(4.8,3.2)(4.7,3.1)(4.6,3.2)
(4.5,3.3)(4.4,3.2)(4.3,3.1)(4.2,3.2)(4.1,3.3)(4,3.2)
(3.9,3.1)(3.8,3.2)(3.7,3.3)(3.6,3.2)
\psecurve[linewidth=1pt,linestyle=solid,linecolor=magenta]
(4.25,2.6)(4.6,3.2)(4.8,3.35)(5,3.4)(5.15,3.3)(5.15,3.1)(5,2.95)(4.6,2.2)(5,1.45)
(5.15,1.3)(5.15,1.1)(5,1)(4.8,1.05)(4.6,1.2)(4.25,1.8)
\psecurve[linewidth=1pt,linestyle=dashed,linecolor=magenta]
(4.8,1.05)(4.6,1.2)(4.25,1.8)(4.2,2.2)(4.25,2.6)(4.6,3.2)(4.8,3.35)
\psecurve[linewidth=1pt,linestyle=solid,linecolor=magenta]
(4.83,1.3)(4.8,1.2)(4.83,1.1)(5,1)(5.2,1.05)(5.4,1.2)(5.75,1.8)(5.8,2.2)
(5.75,2.6)(5.4,3.2)(5.2,3.35)(5,3.4)(4.83,3.3)(4.8,3.2)(4.83,3.1)
\psecurve[linewidth=1pt,linestyle=dashed,linecolor=magenta]
(4.83,3.3)(4.8,3.2)(4.83,3.1)(5,2.95)(5.4,2.2)(5,1.45)
(4.83,1.3)(4.8,1.2)(4.83,1.1)

\psline[linewidth=1pt,linestyle=solid,linecolor=magenta](4.6,2.2)(4.7,2.4)
\psline[linewidth=1pt,linestyle=solid,linecolor=magenta](4.6,2.2)(4.5,2.4)
\psline[linewidth=1pt,linestyle=solid,linecolor=magenta](5.8,2.2)(5.7,2.4)
\psline[linewidth=1pt,linestyle=solid,linecolor=magenta](5.8,2.2)(5.9,2.4)
\psline[linewidth=1pt,linestyle=solid,linecolor=magenta](4.2,2.4)(4.3,2.2)
\psline[linewidth=1pt,linestyle=solid,linecolor=magenta](4.2,2.4)(4.1,2.2)
\psline[linewidth=1pt,linestyle=solid,linecolor=magenta](5.36,2.4)(5.5,2.22)
\psline[linewidth=1pt,linestyle=solid,linecolor=magenta](5.36,2.4)(5.32,2.18)
\pscurve[linewidth=1pt,linestyle=solid,linecolor=blue]
(5.1,2)(5.25,1.8)(5.3,1.1)(5.05,0.1)
\psecurve[linewidth=1pt,linestyle=solid,linecolor=blue]
(4.9,2)(4.9,2)(4.75,1.8)(4.7,1.3)(4.95,0.1)
\psecurve[linewidth=1pt,linestyle=dashed,linecolor=blue]
(4.75,1.8)(4.7,1.3)(4.95,0.1)(4.95,0.1)
\psline[linewidth=1pt,linestyle=solid,linecolor=blue](4.69,1.6)(4.79,1.4)
\psline[linewidth=1pt,linestyle=solid,linecolor=blue](4.69,1.6)(4.59,1.4)
\psline[linewidth=1pt,linestyle=solid,linecolor=blue](5.32,1.4)(5.42,1.6)
\psline[linewidth=1pt,linestyle=solid,linecolor=blue](5.32,1.4)(5.22,1.6)
\psline[linewidth=1pt,linestyle=solid,linecolor=blue](4.78,0.8)(4.72,0.6)
\psline[linewidth=1pt,linestyle=solid,linecolor=blue](4.78,0.8)(4.9,0.62)
\end{pspicture}
\end{center}
\caption{Simple paths and simple cycles
for a degenerate horizontal strip $D_i$
and for the regular horizontal strip $D_j$ surrounding $D_i$.}
\label{fig:bg2}
\end{figure}
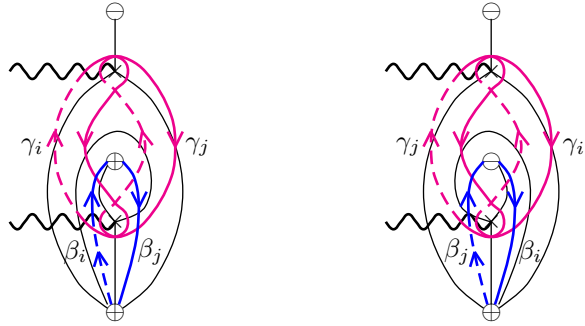

\begin{figure}
\begin{center}
\begin{pspicture}(-4.6,0)(-2.35,3.2)
\psset{unit=10mm}
%
%
\rput[c]{0}(-3,3){$\oplus$}
\rput[c]{0}(-3,0){$\ominus$}
\rput[c]{0}(-3,1.4){$\otimes$}
\rput[c]{0}(-3,2){$\times$}
\pscurve[linewidth=1pt]
(-3.1,1.3)(-3.2,1.4)(-3.3,1.5)(-3.4,1.4)%
(-3.5,1.3)(-3.6,1.4)(-3.7,1.5)(-3.8,1.4)%
(-3.9,1.3)(-4.0,1.4)(-4.1,1.5)(-4.2,1.4)%
(-4.3,1.3)(-4.4,1.4)(-4.5,1.5)(-4.6,1.4)%
\pscurve[linewidth=1pt](-3,2)(-3.1,1.9)(-3.2,2)(-3.3,2.1)(-3.4,2)%
(-3.5,1.9)(-3.6,2)(-3.7,2.1)(-3.8,2)%
(-3.9,1.9)(-4.0,2)(-4.1,2.1)(-4.2,2)%
(-4.3,1.9)(-4.4,2)(-4.5,2.1)(-4.6,2)%
\psline[linewidth=0.5pt](-3,2.9)(-3,2)
\psline[linewidth=0.5pt](-3,1.3)(-3,0.1)
\pscurve[linewidth=0.5pt](-3.1,0.02)(-3.65,0.5)(-3.65,1.5)(-3,2)
\pscurve[linewidth=0.5pt](-2.9,0.02)(-2.35,0.5)(-2.35,1.5)(-3,2)

\psecurve[linewidth=1pt,linecolor=magenta](-3.4,2)(-3.4,1.4)
(-3,1.05)(-2.6,1.4)(-2.6,2)(-3,2.35)(-3.4,2)(-3.4,1.4)
\psecurve[linewidth=1pt,linestyle=dashed,linecolor=magenta]%
(-3,2.35)(-3.4,2)(-3.4,1.4)(-3,1.05)
\psline[linewidth=1pt,linecolor=magenta](-2.71,2.2)(-2.7,1.95)
\psline[linewidth=1pt,linecolor=magenta](-2.71,2.2)(-2.5,2.05)
\psecurve[linewidth=1pt,linecolor=blue](-2.95,0.08)(-2.95,0.08)
(-2.7,1.5)(-3,1.7)(-3.3,1.5)(-3.05,0.08)
\psecurve[linewidth=1pt,linestyle=dashed,linecolor=blue](-3,1.7)
(-3.3,1.5)(-3.05,0.08)(-3.05,0.08)
\psline[linewidth=1pt,linecolor=blue](-2.67,1)(-2.61,0.72)
\psline[linewidth=1pt,linecolor=blue](-2.67,1)(-2.83,0.76)
\rput[c]{0}(-2.3,2){$\gamma_i$}
\rput[c]{0}(-2.5,0.6){$\beta_i$}
\end{pspicture}
\end{center}
\caption{Simple path and simple cycle
for a  horizontal strip $D_i$ of simple-pole type.}
\label{fig:bg3}
\end{figure}
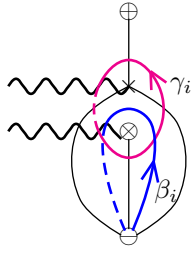

\subsection{Mutation formula of the simple paths and the simple cycles}

Next let us give the mutation formula of 
the simple paths and the simple cycles for signed flips.

We consider the situation of Section \ref{section:stokes-auto} 
(and Theorem \ref{thm:mut1}). Let 
$Q^{(\theta)}(z,\eta)$ be the $S^1$-family of the potentials 
of the Schr{\"o}dinger equation defined in \eqref{eq:S1-potentials}.  
In this subsection we assume that 
\begin{itemize}
\item 
the Stokes graph $G_0 = G(\phi)$ of \eqref{eq:Sch} 
(which corresponds to $\theta = 0$)
has a unique Stokes segment $\ell_0$, which is of type II. 
\end{itemize}

Fix a sign $\ve=\pm$. As is explained in Section \ref{section:stokes-auto}, 
for a sufficiently small $\delta>0$, the Stokes graphs 
$G=G(e^{2i\ve\delta}\phi)$ and $G'=G(e^{-2i\ve\delta}\phi)$ 
for $Q^{(+2i\ve\delta)}(z,\eta)$ and $Q^{(-2i\ve\delta)}(z,\eta)$, 
respectively, are saddle-free. 
Therefore, to the Stokes graph $G$ (resp., $G'$) 
and labels $i = 1,\dots,n$ assigned to its Stokes regions, 
we can associate the labeled Stokes triangulation 
$T = T(e^{2i\ve\delta}\phi)$ 
(resp., $T' = T(e^{-2i\ve\delta}\phi)$), 
the simple paths $\beta_1,\dots,\beta_n \in \Gamma^{\vee}$ 
(resp., $\beta'_1,\dots,\beta'_n \in \Gamma^{\vee}$), 
and the simple cycles $\gamma_1,\dots,\gamma_n \in \Gamma$ 
(resp., $\gamma'_1,\dots,\gamma'_n \in \Gamma$). 
Here we have used the same identification of paths and cycles 
on the Riemann surfaces $\hat{\Sigma}$ of $e^{+i\ve\delta}\sqrt{\phi}$
and $\hat{\Sigma}'$ for $e^{-i\ve\delta}\sqrt{\phi}$ explained 
in Section \ref{section:stokes-auto}. 
We assign the labels by the same way as in Theorem \ref{thm:mut1}; 
these labeled Stokes triangulations are related by the signed mutation
$T'=\mu_k^{(\ve)}(T)$ with $k$ being the label of the pending arc 
in $T$ to be flipped. 

\begin{prop}
\label{prop:mut2}
 Mutation formula. Let $B=B(T)$.
 Then, the following relations hold.
\begin{align}
\label{eq:bemut1}
\beta_i'&=
\begin{cases}
\displaystyle
-\beta_k
+d_k\sum_{j=1}^n [-\varepsilon \tilde{b}_{jk}]_+ \beta_j
 & i=k\\
\beta_i 
& i \neq k,
\end{cases}\\
\label{eq:gamut1}
\gamma_i'&=
\begin{cases}
-\gamma_k & i=k\\
\gamma_i + d_k [\varepsilon \tilde{b}_{ki}]_+ \gamma_k
& i \neq k.
\end{cases}
\end{align}
\end{prop}
\begin{proof}
Again, they are known in the non-simple pole case
in \cite[Propositions 6.28]{Iwaki14a}.
Thus, it is enough to prove them only when horizontal strips of
simple-pole type are involved.
Furthermore, it is enough to prove only 
\eqref{eq:gamut1} thanks to the duality \eqref{eq:dual1}.
This is done by inspecting Figures \ref{fig:orbi3}--\ref{fig:bg3} and
also \cite[Figures 36]{Iwaki14a}.
\end{proof}

\subsection{Mutation formula of Voros symbols for signed flips}
Here we reformulate Theorem \ref{thm:Stokes-auto-II} 
in view of cluster algebra theory.
We follow the same argument for the non-simple pole case 
treated in \cite[Section 7]{Iwaki14a}.

Below we consider the same situation as the previous subsection: 
Fix a sign $\ve=\pm$, and consider a mutation of Stokes graphs 
(for a reduction of the Stokes segment $\ell_0$ which is of type II) 
which induces the signed flip $\mu^{(\ve)}_k$
to the associated labeled Stokes triangulations. 

Let $G$, $G'$, $T$, $T'$, 
$\beta_i, \beta'_i \in \Gamma^{\vee}$ ($i=1,\dots,n$) and  
$\gamma_i, \gamma'_i \in \Gamma$ ($i=1,\dots,n$) be the same ones 
defined in the previous subsection. 
Let $\gamma_0$ be the cycle surrounding the Stokes segment $\ell_0$
as indicated in Figure \ref{fig:saddle-classes} whose orientation 
is determined by the condition \eqref{eq:saddle-class}.

\begin{lem}
\label{lem:cycle1}
 The equality $\gamma_0=\ve \gamma_k$ holds.
\end{lem}

\begin{proof}
This is checked by inspecting Figure \ref{fig:cycle1},
where we assume that $\ell_0$  is in the type i-a configuration in
Figure \ref{fig:orbi2}.
The other cases can be checked similarly by using
Figure \ref{fig:orbi3}.
\end{proof}

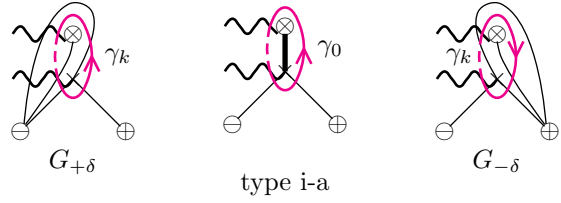
\begin{figure}
\begin{center}
\begin{pspicture}(0,-0.1)(2,2.2)
%
\psset{linewidth=0.5pt}
%
\rput[c]{0}(1,1.6){$\otimes$}
\rput[c]{0}(1,1){$\times$}
\rput[c]{0}(1.7,0.3){$\oplus$}
\rput[c]{0}(0.3,0.3){$\ominus$}
\rput[c]{0}(1,-0.1){$G_{+\delta}$}
\rput[c]{0}(1.6,1.3){$\gamma_k$}
%
\psline(1,1)(0.37,0.37)
\psline(1,1)(1.63,0.37)
\pscurve(0.37,0.4)(0.95,1.3)(1,1.5)
\pscurve(0.35,0.4)(0.95,2)(1.25,1.9)(1,1)
\pscurve[linewidth=1pt]
(0.9,1.5)(0.8,1.6)(0.7,1.7)(0.6,1.6)%
(0.5,1.5)(0.4,1.6)(0.3,1.7)(0.2,1.6)%
\pscurve[linewidth=1pt](1,1)(0.9,0.9)(0.8,1)(0.7,1.1)(0.6,1)%
(0.5,0.9)(0.4,1)(0.3,1.1)(0.2,1)%
\psecurve[linewidth=1pt,linecolor=magenta](0.8,1.6)(0.8,1)(1,0.75)(1.2,1)%
(1.2,1.6)(1,1.85)(0.8,1.6)(0.8,1)
\psecurve[linewidth=1pt,linestyle=dashed,linecolor=magenta]%
(1,1.95)(0.8,1.6)(0.8,1)(1,0.65)
\psline[linewidth=1pt,linecolor=magenta](1.25,1.3)(1.35,1.1)
\psline[linewidth=1pt,linecolor=magenta](1.25,1.3)(1.15,1.1)
\end{pspicture}
\hskip20pt
\begin{pspicture}(0,-0.2)(2,2.2)
%
\psset{linewidth=0.5pt}
%
\rput[c]{0}(1,1.6){$\otimes$}
\rput[c]{0}(1,1){$\times$}
\rput[c]{0}(1.7,0.3){$\oplus$}
\rput[c]{0}(0.3,0.3){$\ominus$}
\rput[c]{0}(1,-0.5){type i-a}
\rput[c]{0}(1.6,1.3){$\gamma_0$}
%
\psline[linewidth=2pt](1,1)(1,1.5)
\psline(1,1)(0.37,0.37)
\psline(1,1)(1.63,0.37)
\pscurve[linewidth=1pt]
(0.9,1.5)(0.8,1.6)(0.7,1.7)(0.6,1.6)%
(0.5,1.5)(0.4,1.6)(0.3,1.7)(0.2,1.6)%
\pscurve[linewidth=1pt](1,1)(0.9,0.9)(0.8,1)(0.7,1.1)(0.6,1)%
(0.5,0.9)(0.4,1)(0.3,1.1)(0.2,1)%
\psecurve[linewidth=1pt,linecolor=magenta](0.8,1.6)(0.8,1)(1,0.75)(1.2,1)%
(1.2,1.6)(1,1.85)(0.8,1.6)(0.8,1)
\psecurve[linewidth=1pt,linestyle=dashed,linecolor=magenta]%
(1,1.95)(0.8,1.6)(0.8,1)(1,0.65)
\psline[linewidth=1pt,linecolor=magenta](1.25,1.3)(1.35,1.1)
\psline[linewidth=1pt,linecolor=magenta](1.25,1.3)(1.15,1.1)
\end{pspicture}
\hskip20pt
\begin{pspicture}(0,-0.1)(2,2.2)
%
\psset{linewidth=0.5pt}
%
\rput[c]{0}(1,1.6){$\otimes$}
\rput[c]{0}(1,1){$\times$}
\rput[c]{0}(1.7,0.3){$\oplus$}
\rput[c]{0}(0.3,0.3){$\ominus$}
\rput[c]{0}(1,-0.1){$G_{-\delta}$}
\rput[c]{0}(0.5,1.3){$\gamma_k$}
%
\psline(1,1)(0.37,0.37)
\psline(1,1)(1.63,0.37)
\pscurve(1.63,0.4)(1.05,1.3)(1,1.5)
\pscurve(1.65,0.4)(1.05,2)(0.75,1.9)(1,1)
\pscurve[linewidth=1pt]
(0.9,1.5)(0.8,1.6)(0.7,1.7)(0.6,1.6)%
(0.5,1.5)(0.4,1.6)(0.3,1.7)(0.2,1.6)%
\pscurve[linewidth=1pt](1,1)(0.9,0.9)(0.8,1)(0.7,1.1)(0.6,1)%
(0.5,0.9)(0.4,1)(0.3,1.1)(0.2,1)%
\psecurve[linewidth=1pt,linecolor=magenta](0.8,1.6)(0.8,1)(1,0.75)(1.2,1)%
(1.2,1.6)(1,1.85)(0.8,1.6)(0.8,1)
\psecurve[linewidth=1pt,linestyle=dashed,linecolor=magenta]%
(1,1.95)(0.8,1.6)(0.8,1)(1,0.65)
\psline[linewidth=1pt,linecolor=magenta](1.25,1.3)(1.35,1.5)
\psline[linewidth=1pt,linecolor=magenta](1.25,1.3)(1.15,1.5)
\end{pspicture}
\end{center}
\caption{Relation between $\gamma_0$ and $\gamma_k$.
}
\label{fig:cycle1}
\end{figure}

Let $e^{W_{\beta}^{(\theta)}}=e^{W_{\beta}^{(\theta)}(\eta)}$ and
$e^{V_{\gamma}^{(\theta)}}=e^{V_{\gamma}^{(\theta)}(\eta)}$
be the Voros symbols for the Schr{\"o}dinger equation with 
the potential $Q^{(\theta)}(z,\eta)$, 
respectively. We also set 
\begin{equation}
{\mathcal S}_{\pm \varepsilon}[e^{W_{\beta}}](\eta) = 
\lim_{\delta\rightarrow + 0} 
\mathcal{S}[e^{W_{\beta}^{(\pm \varepsilon\delta)}}](\eta), \quad 
{\mathcal S}_{\pm \varepsilon}[e^{V_{\gamma}}](\eta) = 
\lim_{\delta\rightarrow + 0} 
\mathcal{S}[e^{V_{\gamma}^{(\pm \varepsilon\delta)}}](\eta). 
\end{equation}
First, we present the following intermediate result obtained 
from Theorem \ref{thm:Stokes-auto-II}.
\begin{prop}
For any path $\beta\in \Gamma^{\vee}$ and any cycle
$\gamma\in \Gamma$, we have
\begin{align} 
\label{eq:epsilon-expression-of-DDP-analytic1}
{\mathcal S}_{-\varepsilon}[e^{W_{\beta}}] &=
 {\mathcal S}_{+\varepsilon}[e^{W_{\beta}}] 
\Bigl( 1 + (t_k+t_k^{-1})
\Bigl( {\mathcal S}_{+\varepsilon}[e^{V_{\gamma_k}}]
\Bigr)^{\varepsilon} + 
\Bigl({\mathcal S}_{+\varepsilon}[e^{2V_{\gamma_k}}]
\Bigr)^{\varepsilon} \Bigr)^{-\langle\gamma_k,\beta\rangle}, 
\hspace{-.5em} \\
\label{eq:epsilon-expression-of-DDP-analytic2}
{\mathcal S}_{-\varepsilon}[e^{V_{\gamma}}] &=
 {\mathcal S}_{+\varepsilon}[e^{V_{\gamma}}] 
\Bigl( 1 + (t_k+t_k^{-1})
\Bigl( {\mathcal S}_{+\varepsilon}[e^{V_{\gamma_k}}]
\Bigr)^{\varepsilon} + 
\Bigl({\mathcal S}_{+\varepsilon}[e^{2V_{\gamma_k}}]
\Bigr)^{\varepsilon} \Bigr)^{-(\gamma_k,\gamma)},
\end{align}
where $t_k=\exp(\pi i \sqrt{1+4b_k})$,
$b_k=\lim_{z\rightarrow s_k}((z-s_k)^2Q_2(z))$ with $s_k$
being the simple pole attached to $\ell_0$.
\end{prop}
\begin{proof}
They follow from Theorem \ref{thm:Stokes-auto-II},
by using Lemma \ref{lem:cycle1} and the equality 
\eqref{eq:no-jump-gamma-0}. 
See the proof of \cite[Proposition 7.3]{Iwaki14a} for more detail.
\end{proof}

Let us introduce
\begin{align}
\label{eq:Voros1}
e^{{W}_i} = e^{{W}^{(\ve \delta)}_{\beta_i}},
\quad 
e^{{V}_i} = e^{{V}^{(\ve \delta)}_{\gamma_i}},
\quad
e^{ v_i} =e^{\eta {v}^{(\ve \delta)}_{\gamma_i}},
\end{align}
\begin{align}
\label{eq:Voros2}
e^{{W}'_i} = e^{{W}^{(-\ve \delta)}_{\beta'_i}},
\quad 
e^{{V}'_i} = e^{{V}^{(-\ve \delta)}_{\gamma'_i}},
\quad
e^{ v'_i} =e^{\eta {v}^{(-\ve \delta)}_{\gamma'_i}}.
\end{align}
where 
\begin{align}
\label{eq:vdef1}
{v}^{(\theta)}_{\gamma} 
= e^{i\theta} \int_{\gamma} \sqrt{Q_0(z)}\, dz.
\end{align}
By \eqref{eq:decom1} we have
\begin{align}
\label{eq:vv1}
e^{V_i}=e^{v_i}
\prod_{j=1}^n (e^{W_j})^{\tilde{b}_{ji}},
\quad
e^{V'_i}=e^{v'_i}
\prod_{j=1}^n (e^{W'_j})^{\tilde{b}'_{ji}},
\end{align}
where $\tilde{B}=D^{-1}B$ and $\tilde{B}'=D^{-1}B'$ with $B$ and $B'$
being the signed adjacency matrices
 of $T$ and $T'$, respectively.
 
 Now we present the second main result of the paper.
 
\begin{thm}
\label{thm:mut2}
 We set
\begin{alignat}{5}
y_i &= \lim_{\delta\rightarrow +0} e^{v_i},
&
\quad
x_i &= \lim_{\delta\rightarrow +0} {\mathcal S}[e^{W_i}],
&\quad
\hat{y}_i &= \lim_{\delta\rightarrow +0} {\mathcal S}[e^{V_i}],\\
y'_i &= \lim_{\delta\rightarrow +0} e^{v'_i},
&
\quad
x'_i &= \lim_{\delta\rightarrow +0} {\mathcal S}[e^{W'_i}],
&\quad
\hat{y}'_i &= \lim_{\delta\rightarrow +0} {\mathcal S}[e^{V'_i}].
\end{alignat}
Then, the following relations hold:
\begin{align}
\label{eq:mut2}
y'_i&=
\begin{cases}
y_k^{-1} & i=k\\
y_i \left(y_k^{[\ve \tilde{b}_{ki}]_+}\right)^2 & i\neq k,\\
\end{cases}
\\
\label{eq:mut3}
x'_i&=
\begin{cases}
\displaystyle
x_k^{-1}
\Bigl(\prod_{j=1}^n 
x_j^{[-\ve \tilde{b}_{jk}]_+}
\Bigr)^2
(1 + (t_k+t_k^{-1})\hat{y}_k^{\ve} + \hat{y}_k^{2\ve})
 & i=k\\
x_i& i\neq k,\\
\end{cases}
\\
\label{eq:mut4}
\hat{y}'_i&=
\begin{cases}
\displaystyle
\hat{y}_k^{-1}
 & i=k\\
\hat{y}_i
\left(\hat{y}_k^{[\ve \tilde{b}_{ki}]_+}\right)^2
(1 + (t_k+t_k^{-1})\hat{y}_k^{\ve} + \hat{y}_k^{2\ve})^{-\tilde{b}_{ki}}
& i\neq k.\\
\end{cases}
\end{align}
\end{thm}
\begin{proof}
The relation \eqref{eq:mut2} follows from
\eqref{eq:gamut1}.
To get the relation \eqref{eq:mut3}, we
specialize
\eqref{eq:epsilon-expression-of-DDP-analytic1} 
with $\beta=\beta'_i$,
then use \eqref{eq:dual1} and
\eqref{eq:bemut1}.
To get the relation \eqref{eq:mut4}, we
specialize
\eqref{eq:epsilon-expression-of-DDP-analytic2} 
with $\gamma=\gamma'_i$,
then use \eqref{eq:int1} and
\eqref{eq:gamut1}.
\end{proof}

\subsection{Generalized cluster algebras}

Let us explain that the relations 
\eqref{eq:mut2}--\eqref{eq:mut4} are indeed 
(an enhancement of)  the mutation in
{\em generalized cluster algebras} recently introduced by
\cite{Chekhov11}.
Here, for  clarity, we minimize the setting
of \cite{Chekhov11}  just to be fitted for our situation.

Let us recall the definition of a {\em seed\/} in cluster algebras.
See  \cite{Fomin07} for detail.
We fix a {\em semifield\/} $\bbP$,
which is a multiplicative abelian group with addition $\oplus$.
Let $\mathbb{Z}\bbP$ be the group ring of $\bbP$,
 and let $\bbQ\bbP$ be the field of fractions of $\bbZ\bbP$.
Let $\bbQ\bbP(x^0)$ be the rational function field in 
the {\em initial $x$-variables\/} $x^0_1,\dots,x^0_n$ over $\bbQ\bbP$.
A {\em seed \/} is a triplet $(B,x,y)$, where $B=(b_{ij})_{i,j=1}^n$
is a skew-symmetrizable matrix, $x=(x_1,\dots,x_n)$ is an $n$-tuple
with $x_i\in \bbQ\bbP(x^0)$, and $y=(y_1,\dots,y_n)$ is an $n$-tuple
with $y_i\in \bbP$.

To consider the generalized cluster algebras limited to our case,
it is crucial to assume that the matrix $B$ factorizes as $B=D\tilde{B}$,
where $D=(d_i \delta_{ij})$ is a diagonal matrix with
$d_i\in \{1,2\}$ and $\tilde{B}$ is an integer matrix,
which is indeed true in our situation (see \eqref{eq:bd1} and 
\eqref{eq:bd2}).
(In general, $\tilde{B}$ is not necessarily skew-symmetric.)
Let $I_2:=\{i\in \{1,\dots,n\} \mid d_i =2 \}$,
and let us fix $z_i\in \bbP$ for each $i\in I_2$ arbitrarily.
We consider the following ``mutation"
$(B',x',y')=\mu_k(B,x,y)$ of seeds at $k=1,\dots,n$:
$B'=\mu_k(B)$ is the usual one in \eqref{eq:bmut}, and
\begin{align}
\label{eq:mut5}
{y}'_i&=
\begin{cases}
\displaystyle
{y}_k^{-1}
 & i=k\\
 {y}_i
{y}_k^{[\ve \tilde{b}_{ki}]_+} 
(1 \oplus {y}_k^{\ve})^{-\tilde{b}_{ki}}
& i\neq k, d_k=1\\
{y}_i
\Bigl({y}_k^{[\ve \tilde{b}_{ki}]_+} 
\Bigr)^2
(1  \oplus z_k{y}_k^{\ve} \oplus {y}_k^{2\ve})^{-\tilde{b}_{ki}}
& i\neq k,d_k=2,\\
\end{cases}
\\
\label{eq:mut6}
x'_i&=
\begin{cases}
\displaystyle
x_k^{-1}
\Bigl(\prod_{j=1}^n 
x_j^{[-\ve \tilde{b}_{jk}]_+}
\Bigr)
\frac{
1  + \hat{y}_k^{\ve}}
{1  \oplus {y}_k^{\ve}}
 & i=k, d_k=1\\
\displaystyle
x_k^{-1}
\Bigl(\prod_{j=1}^n 
x_j^{[-\ve \tilde{b}_{jk}]_+}
\Bigr)^2
\frac{
1 + z_k\hat{y}_k^{\ve} + \hat{y}_k^{2\ve}
}
{
1  \oplus z_k{y}_k^{\ve} \oplus {y}_k^{2\ve}
}
 & i=k, d_k=2\\
x_i& i\neq k,\\
\end{cases}
\end{align}
where $\ve=\pm$ and we set
\begin{align}
\label{eq:yhat1}
\hat{y}_i
= y_i \prod_{j=1}^n
x_j^{\tilde{b}_{ji}}.
\end{align}
We note that $B'$ again factorizes as $B'=D\tilde{B}'$ with the same $D$.
The relations involving $k$ with $d_k=2$ in
\eqref{eq:mut5} and \eqref{eq:mut6}   are unusual
in  ordinary cluster algebras, but they are indeed an example of
the mutation for generalized cluster algebras defined in \cite{Chekhov11}.
It is easy to see that the above expressions do not depend on the choice of
the sign $\ve$.
 Furthermore, 
$\hat{y}$-variables  in \eqref{eq:yhat1}  transform as $y$-variables,
namely,
\begin{align}
\label{eq:mut7}
\hat{y}'_i&=
\begin{cases}
\displaystyle
\hat{y}_k^{-1}
 & i=k\\
 \hat{y}_i
\hat{y}_k^{[\ve \tilde{b}_{ki}]_+} 
(1 + \hat{y}_k^{\ve})^{-\tilde{b}_{ki}}
& i\neq k, d_k=1\\
\hat{y}_i
\Bigl(
\hat{y}_k^{[\ve \tilde{b}_{ki}]_+} 
\Bigr)^2
(1 + z_k\hat{y}_k^{\ve} + \hat{y}_k^{2\ve})^{-\tilde{b}_{ki}}
& i\neq k,d_k=2.\\
\end{cases}
\end{align}

The relations \eqref{eq:mut5}--\eqref{eq:mut7} with $d_k = 2$
are already very close to \eqref{eq:mut2}--\eqref{eq:mut4} and 
\eqref{eq:vv1} under the identification $z_k = t_k + t_k^{-1}$.
In fact, the only  difference between them is
 the absence of the factors $1  \oplus {y}_k{}$ and
$1 \oplus z_k {y}_k^{\ve} \oplus {y}_k^{2\ve}$
 in \eqref{eq:mut2} and \eqref{eq:mut3}.

Having this discrepancy in mind, we now specialize 
the coefficient semifield $\bbP$ to
the {\em tropical semifield\/} $\mathrm{Trop}(y^0,z)$,
which is the multiplicative free abelian group generated by
the {\em initial $y$-variables\/}
$y^0_1,\dots,y^0_n$ and $z_i$ ($i\in I_2$) with the tropical sum
\begin{align}
\begin{split}
&\left(\prod_{i=1}^n (y^0_i)^{a_i}\right)
\left(\prod_{i\in I_2} z_i^{a'_i}\right)
\oplus
\left(\prod_{i=1}^n (y^0_i)^{b_i}\right)
\left(\prod_{i\in I_2} z_i^{b'_i}\right)\\
=&
\left(\prod_{i=1}^n (y^0_i)^{\min(a_i,b_i)}\right)
\left(\prod_{i\in I_2} z_i^{\min(a'_i,b'_i)}\right).
\end{split}
\end{align}
As a rule, we start mutation from the {\em initial seed\/} $(B^0,x^0,y^0)$,
where $B^0$ is any skew-symmetrizable matrix which factorizes as $B^0
=D \tilde{B}^0$, and $x^0$ and $y^0$ are the initial $x$- and $y$-variables;
then, repeat mutations.
Our first observation is that
$z_k$ in the factor
$1  \oplus z_k{y}_k^{\ve} \oplus {y}_k^{2\ve}$ in
\eqref{eq:mut6} is superficial, and it never enters in any $y$-variable $y_i$
after any sequence of mutations.
Thus, the factor
$1  \oplus z_k{y}_k^{\ve} \oplus {y}_k^{2\ve}$
in \eqref{eq:mut5} and \eqref{eq:mut6}
are replaced with
the factor
$1  \oplus {y}_k^{\ve} \oplus {y}_k^{2\ve}$
without changing the mutation.

We say that $y_i =\prod_{j=1}^n (y^0_j)^{c_j}\in \mathrm{Trop}(y,z)$ is 
{\em sign-coherent\/} if $(c_j)_{j=1}^n$ is nonzero and either $c_j\geq 0$ 
for any $j$ or $c_j \leq 0 $ for any $j$. In this case we set the 
{\em tropical sign\/} $\ve(y_i)$ of $y_i$ as
$+$ or $-$,  according to the sign of $(c_j)_{j=1}^n$.
Our second observation, which is a familiar idea for ordinary 
cluster algebras, is that if $y_k$ is sign-coherent and 
if we set $\ve=\ve(y_k)$, then we have 
$1  \oplus {y}_k^{\ve} \oplus {y}_k^{2\ve}=1$.

Let us introduce the {\em signed mutation\/} $(B',x',y')=\mu_k^{(\ve)}(B,x,y)$ 
 of seeds ($k=1,\dots,n$; $\ve=\pm$)
for our generalized cluster algebras by replacing 
\eqref{eq:mut5} and \eqref{eq:mut6} with 
\begin{align}
\label{eq:mut8}
{y}'_i&=
\begin{cases}
\displaystyle
{y}_k^{-1}
 & i=k\\
 {y}_i
\Bigl({y}_k^{[\ve \tilde{b}_{ki}]_+} 
\Bigr)^{d_k}
& i\neq k,\\
\end{cases}
\\
\label{eq:mut9}
x'_i&=
\begin{cases}
\displaystyle
x_k^{-1}
\Bigl(\prod_{j=1}^n 
x_j^{[-\ve \tilde{b}_{jk}]_+}
\Bigr)
(1  + \hat{y}_k^{\ve})
 & i=k, d_k=1\\
\displaystyle
x_k^{-1}
\Bigl(\prod_{j=1}^n 
x_j^{[-\ve \tilde{b}_{jk}]_+}
\Bigr)^2
(1 + z_k\hat{y}_k^{\ve} + \hat{y}_k^{2\ve})
 & i=k, d_k=2\\
x_i& i\neq k,\\
\end{cases}
\end{align}
which completely agree with \eqref{eq:mut2} and \eqref{eq:mut3}.
We note that the relation \eqref{eq:mut7} still holds.
This generalizes the signed mutations for ordinary cluster algebras 
introduced in \cite[Section 4]{Iwaki14a}.

Now we have reached to the conclusion that the signed mutation $\mu_k^{(\ve)}$,
which appears in Theorem \ref{thm:mut2}, coincides with the ordinary mutation
$\mu_k$ in a generalized cluster algebra if $y_k$ is sign-coherent and
the sign $\ve$ is chosen to be the tropical sign of $y_k$;
otherwise, it is an extension of the ordinary mutation therein.

\begin{rem} 
In this paper we concentrated on the signed flips of
{\em pending arcs\/} for simplicity.
As for the signed flips of {\em ordinary arcs},
it was already shown in \cite[Theorem 7.5]{Iwaki14a} that
the Voros symbols mutate by \eqref{eq:mut8}, \eqref{eq:mut9},
and \eqref{eq:mut7} with $d_k=1$
{\em in the absence of simple poles}.
Using Propositions \ref{prop:simple1}
and \ref{prop:mut2},
we can easily extend the proof therein
to show that this is true 
{\em even in the presence of simple poles}. 
Therefore, the formulas \eqref{eq:mut8} and \eqref{eq:mut9}
integrate the mutation formula of the Voros
symbols for all cases of signed flips.
\end{rem}

Finally, let us present a conjecture clarifying the role of
Stokes triangulations from the cluster algebraic point of view.
Let $T$ be a labeled Stokes triangulation of an orbifold $(\bfO,\bfA)$.
Let us consider a sequence of signed flips
$T'=\mu_{k_N}^{(\ve_N)}\circ \cdots \circ \mu_{k_1}^{(\ve_1)}(T)$.
Let $(B,x,y)$ be a seed with $B=B(T)$,
and apply the same sequence of signed mutations to $(B,x,y)$ as
$(B',x',y')=\mu_{k_N}^{(\ve_N)}\circ \cdots \circ \mu_{k_1}^{(\ve_1)}(B,x,y)$.

\begin{conj}
We have $T=T'$ if and only if
 $(B',x',y')=(B,x,y)$.
\end{conj}

This is an extension of \cite[Conjecture 6.6 (i)]{Iwaki14a}.



\begin{thebibliography}{99}
\small

\bibitem[AIT]{Aoki14} 
 T. Aoki, K. Iwaki and T. Takahashi, 
 {\em Exact WKB analysis of Schr{\"o}dinger equations with 
 a Stokes curve of loop type}, in preparation. 

\bibitem[AKT91]{Aoki91}
 T. Aoki, T. Kawai and Y. Takei, 
 {\em The Bender-Wu analysis and the Voros theory},  
 ICM-90 Satellite Conf. Proc. ``Special Functions", 
 Springer-Verlag, 1991, pp. 1--29.

\bibitem[AKT09]{Aoki09}
 \bysame, 
 {\em The Bender-Wu analysis and the Voros theory II}, 
 Adv. Stud. in Pure Math., {\bf 54}, Math. Soc. Japan, 2009, 
 pp.19--94.

\bibitem[BS13]{Bridgeland13} 
 T. Bridgeland and I. Smith, 
 {\em Quadratic differentials as stability conditions}, 
 2013, arXiv:1302.7030 [math.AG]. 

\bibitem[CS14]{Chekhov11}
 L. Chekhov and M. Shapiro, 
 {\em Teichm\"uller spaces of Riemann surfaces with orbifold 
 points of arbitrary order and cluster variables}, 
 Int. Math. Res. Notices {\bf 2014} (2014), 2746--2772;
 arXiv:1111.3963 [math-ph].
 
\bibitem[Cos09]{Costin09}
 O. Costin, {\em Asymptotics and Borel Summability}, Monographs and Surveys 
 in Pure and Applied Mathematics 141, 
 Chapmann \verb+&+ Hall/CRC, New York, 2009.

\bibitem[DDP93]{Delabaere93}
 E. Delabaere, H. Dillinger and F. Pham,   
 R\'esurgence de Voros et p\'eriodes des courbes hyperelliptiques,
 Ann. Inst. Fourier (Grenoble) {\bf 43} (1993), 163--199.

\bibitem[DP99]{Delabaere99}
 E. Delabaere and F. Pham, 
 {\em Resurgent methods in semi-classical asymptotics}, 
 Ann. Inst. Henri Poincar\'e {\bf 71} (1999), 1--94.

\bibitem[FST11]{Felikson11}
A. Felikson, M. Shapiro, and P. Tumarkin, 
{\it Cluster algebras and triangulated orbifolds},
Adv. in Math. {\bf 231} (2012), 2953--3002; arXiv:1111.3449 [math.CO].

\bibitem[FG06]{Fock03b}
V.V. Fock and A.B. Goncharov, {\it Moduli spaces of local systems and 
higher Teichm\"uller theory}, 
Publ. Math. IHES {\bf 103} (2006), 1--211; 
arXiv:math/0311149 [math.AG].

\bibitem[FST08]{Fomin08}
S. Fomin, M. Shapiro, and D. Thurston, 
{\it Cluster algebras and triangulated surfaces. Part I: Cluster complexes}, 
Acta Math. {\bf 201} (2008), 83--146; arXiv:math/0608367 [math.RA].

\bibitem[FT12]{Fomin08b}
S. Fomin and D. Thurston, 
{\it Cluster algebras and triangulated surfaces. Part II: Lambda lengths}, 
2012, arXiv:1210.5569 [math.GT].

\bibitem[FZ03]{Fomin03a}
S. Fomin and A. Zelevinsky, 
{\it Cluster algebras II. Finite type classification}, 
Invent. Math. {\bf 154} (2003), 63--121; arXiv:math/0208229 [math.RA].

\bibitem[FZ07]{Fomin07}
\bysame , {\it Cluster algebras IV. Coefficients}, 
Compositio Mathematica {\bf 143} (2007), 112--164; 
arXiv:math/0602259 [math.RT].

\bibitem[GMN13]{Gaiotto09}
D. Gaiotto, G.W. Moore, and A. Neitzke, 
{\em Wall-crossing, Hitchin systems, and the WKB approximation}, 
Adv. in Math. {\bf 234} (2013), 239--403; arXiv:0907.3987 [hep-th].

\bibitem[GSV05]{Gekhtman05} 
M. Gekhtman, M. Shapiro, and A. Vainshtein, 
{\it Cluster algebras and Weil-Petersson forms}, 
Duke Math. J. {\bf 127} (2005), 291--311; arXiv:math/0309138 [math.QA].

\bibitem[IN14]{Iwaki14a}
K. Iwaki and T. Nakanishi, 
{\it Exact WKB analysis and cluster algebras}, 
2014, arXiv:1401.7094 [math CA], 
to appear in J. Phys. A.

\bibitem[KKKT10]{Kamimoto10}
S. Kamimoto, T. Kawai, T. Koike and Y. Takei, 
{\it On the WKB-theoretic structure of a Schr{\"o}dinger operator 
with a merging pair of a simple pole and a simple turning point},
Kyoto J. Math. {\bf 50} (2010), 101--164.

\bibitem[KK11]{Kamimoto11}
S. Kamimoto and T. Koike, 
{\it On the Borel summability of WKB-theoretic transformation series},
preprint of RIMS-1726. 

\bibitem[KT05]{Kawai05} 
T. Kawai and Y. Takei, 
{\it Algebraic Analysis of Singular Perturbation Theory}, 
Translations of Mathematical Monographs, 
volume 227, American Mathematical Society, 2005.
  
\bibitem[Koi00]{Koike00} T. Koike, 
{\em On the exact WKB analysis of second order linear ordinary 
differential equations with simple poles}, 
Publ. RIMS, Kyoto Univ. {\bf 36} (2000), 297--319.
  
\bibitem[KS]{Koike-Schafke} 
T. Koike and R. Sch\"afke, 
 {\it On the Borel summability of WKB solutions of 
 Schr\"odinger equations with polynomial potentials 
 and its application}, in preparation; also Talk given 
 by T. Koike in the RIMS workshop 
 ``Exact WKB analysis -- Borel summability of WKB solutions", 
 September, 2010.
  
\bibitem[KT11]{Koike11}
T. Koike and Y. Takei, 
 {\it On the Voros Coefficient for the Whittaker Equation with 
 a Large Parameter -- Some Progress around Sato's Conjecture 
 in Exact WKB Analysis}, 
 Publ. RIMS, Kyoto Univ. {\bf 47} (2011), 375--395.
  
\bibitem[Qiu14]{Qiu14}
Y. Qiu, {\it On the spherical twists on 3-Calabi-Yau categories 
from marked surfaces}, 2014, arXiv:1407.0806.
  
\bibitem[Str84]{Strebel84}
K. Strebel, {\it Quadratic Differentials}, Springer-Verlag, 1984.

\bibitem[Vor83]{Voros83}
A. Voros, 
{\it The return of the quartic oscillator. The complex WKB method}, 
Ann. Inst. Henri Poincar\'e \textbf{39}(1983), 211--338.
  
  
\end{thebibliography}


\end{document}